\documentclass[11pt]{article}
\usepackage[margin=1in]{geometry}
\usepackage[english]{babel}

\usepackage[latin9]{inputenc}
\usepackage{amsthm}
\usepackage{amsmath}
\usepackage{amssymb}
\usepackage{authblk}
\usepackage[margin=2cm, labelfont = bf]{caption}
\usepackage{setspace}
\usepackage{esint}
\usepackage{enumerate}
\usepackage{etoolbox}
\usepackage{url}
\usepackage{graphicx, epstopdf}
\usepackage[english]{babel}
\usepackage{xcolor}
\usepackage{soul}
\usepackage{tikz}
\usetikzlibrary{automata,arrows,positioning,calc}

\numberwithin{equation}{section}

\makeatletter
\DeclareMathSymbol{\shortminus}{\mathbin}{AMSa}{"39}
\theoremstyle{plain}
\newtheorem{thm}{\protect\theoremname}[section]
\ifx\proof\undefined
\newenvironment{proof}[1][\protect\proofname]{\par
	\normalfont\topsep6\p@\@plus6\p@\relax
	\trivlist
	\itemindent\parindent
	\item[\hskip\labelsep
	\scshape
	#1]\ignorespaces
}{%
	\endtrivlist\@endpefalse
}
\providecommand{\proofname}{Proof}
\fi
\theoremstyle{plain}
\newtheorem{lem}[thm]{\protect\lemmaname}
\theoremstyle{plain}
\newtheorem{prop}[thm]{\protect\propositionname}
\theoremstyle{plain}

\theoremstyle{definition}
\newtheorem{defn}[thm]{\protect\definitionname}
\theoremstyle{remark}
\newtheorem{rem}[thm]{\protect\remarkname}
\theoremstyle{plain}

\theoremstyle{definition}

\numberwithin{figure}{section}

\makeatother

\usepackage{babel}
\providecommand{\conjecturename}{Conjecture}
\providecommand{\corollaryname}{Corollary}
\providecommand{\definitionname}{Definition}
\providecommand{\examplename}{Example}
\providecommand{\lemmaname}{Lemma}
\providecommand{\propositionname}{Proposition}
\providecommand{\remarkname}{Remark}
\providecommand{\theoremname}{Theorem}

\newcommand{\eps}{\varepsilon}
\newtheorem{assump}[thm]{Assumption}
\usepackage{mathrsfs}

\newcommand{\loss}{\mathbf{L}}
\begin{document}
	
	
\title{Contagious McKean--Vlasov systems with \\ heterogeneous impact and exposure}
\author{Zachary Feinstein\thanks{Stevens Institute of Technology, School of Business, Hoboken, NJ 07030, USA. \tt{zfeinste@stevens.edu}} \;\;\; Andreas S{\o}jmark\thanks{The London School of Economics, Department of Statistics, London, WC2A 2AE, UK. \tt{a.sojmark@lse.ac.uk}}}

	\maketitle

	\begin{abstract} We introduce a heterogeneous formulation of a contagious McKean--Vlasov system, whose inherent heterogeneity comes from asymmetric interactions with a natural and highly tractable structure. It is shown that this formulation characterises the limit points of a finite particle system, deriving from a balance sheet based model of solvency contagion in interbank markets, where banks have heterogeneous exposure to and impact on the distress within the system. We also provide a simple result on global uniqueness for the full problem with common noise under a smallness condition on the strength of interactions, and we show that, in the problem without common noise, there is a unique differentiable solution up to an explosion time. Finally, we discuss an intuitive and consistent way of specifying how the system should jump to resolve an instability when the contagious pressures become too large. This is known to happen even in the homogeneous version of the problem, where jumps are specified by a `physical' notion of solution, but no such notion currently exists for a heterogeneous formulation of the system.
	\end{abstract}

\section{Introduction}

This paper studies a family of contagious McKean--Vlasov problems, modelling a large cloud of stochastically evolving mean-field particles, for which contagion materialises through asymmetric interactions. More concretely, the different particles suffer a negative impact on their `healthiness' (measured by their distance from the origin) as the probability of absorption increases for the other particles to which they are linked. Naturally, the degree to which the particles are affected depends on the strength of the links.

The study of this problem is motivated by the macroscopic quantification of systemic risk in large financial markets, when taking into account the heterogeneous nature of how financial institutions have an effect on and are exposed to the distress of other entities at the microscopic level. Specifically, in \cite{FS20}, the authors of the present paper have proposed a dynamic balance sheet based model for solvency contagion in interbank markets building on the Gai--Kapadia approach to financial contagion \cite{Gai-Kapadia}. When passing from a finite setting to a mean-field approximation, we will show here that this model leads to precisely the type of contagious McKean--Vlasov system with heterogeneity that we now describe.

For given parameters, to be specified below, the mean-field problem is formulated as a coupled system of conditional McKean--Vlasov equations
	\begin{equation}\label{CMV1}
\begin{cases}
dX_{u,v}(t) =  b_{u,v}(t)dt  + \!\sigma_{u,v}(t)dW_{u,v}(t)- d F\Bigl(  \displaystyle\int_0^t g_{u,v}(s) d\mathbf{L}_v(s)   \Big),
\\[6pt]
\mathbf{L}_{v}(t)  = \!\displaystyle\int_{\mathbb{R}^k\times\mathbb{R}^k}  \kappa(\hat{u},v) \mathbb{P}( t\geq \tau_{\hat{u},\hat{v} }\, | \, B_0) d\varpi(\hat{u},\hat{v}), \\[9pt]
 \tau_{u,v} =\inf \{ t > 0 : X_{u,v}(t) \leq 0 \}, \quad W_{u,v}(t)= \rho B_0(t)+\sqrt{1-\rho^2}B_{u,v}(t),\vspace{2pt}
\end{cases}
\end{equation}
where each $B_{u,v}$ is a Brownian motion independent of the `common' Brownian motion $B_0$. We take the initial conditions for $X_{u,v}(0)$ to be independent of the Brownian motions, and we require that solutions of \eqref{CMV1} must satisfy $\mathbf{L}_v(0)=0$ for all $v$ in the support of $\varpi$.

Note that $B_0$ serves as a `common factor' in the sense that the dynamics of the entire system are conditional upon its movements. In contrast, each $B_{u,v}$ models instead the random fluctuations that are specific to a given mean-field particle $X_{u,v}$. Notice also that, since $\mathbf{L}$ is required to be $B_0$-measurable, computing $t\mapsto \mathbb{P}(t\geq \tau_{u,v} \mid B_0)$ only relies on $B_{u,v}$ being a Brownian motion independent of $B_0$. Any relationship between the $B_{u,v}$'s is irrelevant for this, and solving the system for a different set of particle-specific Brownian drivers $B_{u,v}$ does not change $\mathbf{L}$ as long as they are all independent of $B_0$ and the initial conditions. When the exogenous correlation parameter $\rho$ is zero, we see that the common factor $B_0$ plays no role and so $\mathbf{L}$ becomes deterministic.

The key reason for our interest in \eqref{CMV1} is that it has a quite general, but also highly tractable, heterogeneous structure. This contrasts with the focus on purely symmetric formulations of the McKean--Vlasov problem in most of the existing literature, see e.g.~\cite{Bayraktar89, Cuchiero20, DIRT15a,DIRT15b, DNS21,HLS18,LS18b,LS18a, NS17}. As regards the measure $\varpi$ in \eqref{CMV1}, this is taken to be a probability distribution on $\mathbb{R}^k\times\mathbb{R}^k$ specifying the density or discrete support of the `indexing' vectors $(u,v)\in\mathbb{R}^k\times \mathbb{R}^k$. One should think of each $X_{u,v}$ as a tagged `infinitesimal' mean-field particle within a heterogeneous `cloud' or `continuum' of such particles. In this respect, the role of $\varpi$ is to model how these infinitesimal particles are distributed along a possible continuum of types, as identified by the `indexing' vectors which describe how the particles interact through the interaction kernel $\kappa$.

Naturally, if $\varpi$ is supported on a finite set of indexing vectors, then \eqref{CMV1} consists, in effect, only of a finite number of mean-field particles, but one should still think of it as there being an infinitude of identical particles for each `type', where the values that $\varpi$ assign to the indexing vectors give the proportions of these types (see also Remark \ref{sec:multi-type} below and \cite{NS18}).

The strength with which a given `infinitesimal' mean-field particle $X_{u,v}$ feels the impact of another one, say $X_{\hat{u},\hat{v}}$, is specified by the value $\kappa(\hat{u},v)\geq 0$, modelling an `exposure' of $X_{u,v}(t)$ to the probability of $X_{\hat{u},\hat{v}}$ being absorbed by time $t$. This is the nature of the \emph{contagious} element in this system: a higher likelihood of absorption for any given $X_{\hat{u},\hat{v}}$ puts upward pressure on the likelihood of absorption for $X_{u,v}$ in proportion to $\kappa(\hat{u},v)\geq0$, and likewise throughout the system, in turn forming a positive feedback loop. Notice that, for a given particle $X_{u,v}$, the vector $v$ influences only its exposure to impacts from other particles, while the vector $u$ influences only how it impacts the other particles (but, of course, the full exposure and impact depends also on the indexing vectors of all the other particles). This decomposition highlights why it is natural to work with a \emph{pair} of indexing vectors, and we will further exploit this structure below.

While we did not specify it, for the above interpretation of the contagious element, we were implicitly assuming that the functions $t\mapsto g_{u,v}(t)$ are all non-negative and that the map $x\mapsto F(x)$ is both non-negative and increasing. That way, the effect of an increasing probability of absorption is always to decrease the `healthiness' of each mean-field particle. The precise assumptions for the various parameters are presented in the next subsection: see Assumptions \ref{uv_assump}, \ref{MV_assump}, and \ref{X0_assump}.

\subsection{A specific formulation and structural conditions}\label{intro_applications}

For notational simplicity, given a pair of random vectors $(\mathbf{u},\mathbf{v})$ in $\mathbb{R}^k\times\mathbb{R}^k$ distributed according to $\varpi$, we let $S(\mathbf{u})$, $S(\mathbf{v})$, and $S(\mathbf{u},\mathbf{v})$ denote the support of the random vectors $\mathbf{u}$, $\mathbf{v}$, and $(\mathbf{u},\mathbf{v})$, respectively. The heterogeneity of the interactions in \eqref{CMV1} are  then structured according to the interaction kernel $(u,v)\mapsto \kappa(u,v)$, for $u\in S(\mathbf{u})$ and $v\in S(\mathbf{v})$, where $\kappa: S(\mathbf{u}) \times S(\mathbf{v}) \rightarrow \mathbb{R}_+$ is a continuous non-negative function. We shall rely on this notation throughout the paper. For the applications we are interested in, it is natural to take $\kappa(u,v):=u \cdot v$ together with a distribution $\varpi$ such that $u \cdot v \geq 0$ for all $u\in S(\mathbf{u})$ and $v \in S(\mathbf{v}) $. For concreteness, we will thus restrict to this case throughout, but we note that our arguments are performed in a way that can be easily extended to cover a general continuous non-negative interaction kernel $\kappa: S(\mathbf{u})\times S(\mathbf{v})  \rightarrow \mathbb{R}_+$.

Restricting to the case $\kappa(u,v)=u \cdot v$, the system \eqref{CMV1} can be rewritten as
	\begin{equation}\label{CMV}
\begin{cases}
dX_{u,v}(t) =  b_{u,v}(t)dt  + \!\sigma_{u,v}(t)dW_{u,v}(t) - d F\Bigl(  \displaystyle\sum_{l=1}^k v_l \!\int_0^t g_{u,v}(s) d\mathcal{L}_l(s)   \Big) \\[4pt]
\mathcal{L}_{l}(t)  = \!\displaystyle\int_{\mathbb{R}^k\times\mathbb{R}^k}  u_l \mathbb{P}( t\geq \tau_{u,v} \, | \, B_0) d\varpi(u,v), \quad l=1,\ldots,k,\\[10pt]
\tau_{u,v} =\inf \{ t > 0 : X_{u,v}(t) \leq 0 \}, \quad W_{u,v}(t)=\rho B_0(t) + \sqrt{1-\rho^2}B_{u,v}(t),\vspace{2pt}
\end{cases}
\end{equation}
where we insist on $\mathcal{L}_l(0)=0$ for $l=1,\ldots,k$. A nice feature of this formulation is the decomposition of $\mathbf{L}_v(t)$ into a weighted sum of $k$ \emph{contagion processes} $\mathcal{L}_l(t)$, for $l=1,\ldots,k$, thereby organising the feedback felt by each particle according to $k$ distinct characteristics modelled by the dimension $k$ of the indexing vectors $(u,v)\in \mathbb{R}^k\times \mathbb{R}^k$. More generally, \eqref{CMV} fixes a sensible choice of the interaction kernel $\kappa$, thereby eliminating any further free parameters. Of course, we are still left with a single free parameter $k \in \mathbb{N}$, but this should rightly be seen as a scale for the granularity of the analysis. We refer to \cite{FS20} for further details.

The relevant coefficients for the model of interbank contagion in \cite{FS20} are $F(x)=\log(1+x)$ and $g_{u,v}(t)=c_{u,v}\psi(t,T)$, for $t\in[0,T]$, where each  $c_{u,v}>0$ is a positive constant and $\psi(\cdot,T):[0,T]\rightarrow \mathbb{R}_+$ is a non-negative continuous decreasing function, modelling the rate at which outstanding liabilities are gradually settled over the period $[0,T]$. There are also closely related applications in neuroscience \cite{DIRT15a, DIRT15b, Inglis_Talay}, for which the relevant choices are $F(x)=x$ and $g_{u,v}(t)=c_{u,v}$ for constants $c_{u,v}>0$, although we note that one would also need to consider re-setting of the particles when they hit the origin in this case. These considerations motivate the following assumptions.

\begin{assump}[Distribution of the indexing vectors]\label{uv_assump}
	For a given probability distribution $\varpi$ on $\mathbb{R}^k\times\mathbb{R}^k$, determining the network structure, we write $(\mathbf{u},\mathbf{v})\sim \varpi$. We then assume that (i) the marginal supports $S(\mathbf{u})$ and $S(\mathbf{v})$ are both compact, and (ii) we have $u\cdot v\geq 0$ for all $u\in S(\mathbf{u})$ and $v\in S(\mathbf{v})$. We write $S(\mathbf{u},\mathbf{v})$ for the joint support, which is also compact.
\end{assump}

\begin{assump}[Properties of the coefficients]\label{MV_assump}
	The function $x 
	\mapsto F(x)$ is Lipschitz continuous, non-negative, and non-decreasing with $F(0)=0$. 
	Writing $b_{u,v}(s)=b(u,v,s)$, $\sigma_{u,v}(s)=\sigma(u,v,s)$, and $g_{u,v}(s)=g(u,v,s)$, these are all deterministic continuous functions in $(u,v,s)$. Furthermore, each $s\mapsto g_{u,v}(s)$ is non-negative and non-increasing. Finally, we impose the non-degeneracy conditions $\rho\in[0,1)$ and $c\leq \sigma_{u,v}\leq C$, for given $c,C>0$, uniformly in $(u,v)\in S(\mathbf{u},\mathbf{v})$.
\end{assump}
\begin{assump}[Initial conditions]\label{X0_assump} Letting $\nu_{0}(\cdot|u,v)$ denote the law of  $X_{u,v}(0)$, we assume that
$d\nu_{0}(x|u,v)=V_0(x|u,v)dx$ for a density $V_0(\cdot|u,v)\in L^\infty(0,\infty)$ satisfying $ V_0(x|u,v)\leq C_1 x^\beta$ for all $x$ near $0$, for some $\beta>0$, as well as	$\Vert V_0(\cdot|u,v) \Vert_\infty \leq C_2$  and $\int_0^\infty x d\nu_0(x|u,v) \leq C_3$, for given constants $C_1,C_2,C_3>0$, uniformly in $(u,v)\in S(\mathbf{u},\mathbf{v})$.
\end{assump}

In Section \ref{subsect:uniqueness_regularity}, we present some well-posedness results under these assumptions. Thereafter, we introduce a finite interacting particle system in Section \ref{Subsect:particle_sys}, which will correspond to the coupled McKean--Vlasov system \eqref{CMV} in the mean-field limit. We prove this under Assumption \ref{assump:paramters_emperical}, which in particular ensures that the above Assumptions \ref{uv_assump}, \ref{MV_assump}, and \ref{X0_assump} are all satisfied in the limit.

\subsection{A further look at applications and related literature}\label{sec:further_applications}

In many practical applications, heterogeneity plays a critical role---something that becomes particularly pertinent when seeking to understand the outcomes of contagion spreading through a complex system. In short, a homogeneous version of the McKean--Vlasov problem would be at risk of oversimplifying the conclusions that one can draw. At the same time, the streamlined analysis of a mean-field formulation can be highly instructive, so it becomes important to look for ways of exploiting such macroscopic `averaging' while not throwing away all the microscopic heterogeneity.

For some concrete examples of how \eqref{CMV} can capture features of realistic core-periphery networks in interbank markets, we refer to the discussion at the start of \cite[Section 3.1]{FS20} as well as \cite[Online Supplement C]{FS20}. Another tractable special case is the multi-type system with homogeneity within each type, as considered for the closely related mean-field problem in \cite{NS18} (see Remark \ref{sec:multi-type} below). In relation to this, it should be noted that \cite{NS18} also studies a mean-field type game, whereby jumps are ruled out due to the strategic possibility of disconnecting gradually from a given particle as it gets closer to absorption. We do not consider any game components in this paper.

The formulation \eqref{CMV} can also be highly relevant for integrate-and-fire models in mathematical neuroscience. For example, \cite[Chapter 13.4]{neuronal_dynamics} introduces a multi-type mean-field model, while \cite{Grazieschi2019} considers a random graph model, where dependent random synaptic weights determine if the spiking of neuron $i$ impacts the voltage potential of another neuron $j\neq i$. Phrasing \cite[Equation (12)]{Grazieschi2019} in the language of the particle system \eqref{particle_sys} in Section \ref{Subsect:particle_sys}, this corresponds to having $k=1$ and letting both $v^i:=\chi^i$ and $u^i:=\chi^i$ for a family of i.i.d.~Bernoulli random variables $\{\chi^i\}_{i=1}^n$, up to a scalar multiple. Thus, the directed connection from $i$ to $j$ is either active or inactive depending on the realisation of $\chi^i \chi^j$. It is conjectured in \cite{Grazieschi2019} that the mean-field limit of this particle system should take a particular form \cite[Equation (13)]{Grazieschi2019}, and we note that this conjecture is indeed confirmed by our convergence result, saying that the limit is \eqref{CMV} with $k=1$ and $\varpi=\text{Law}(\chi ,\chi)$ on $\mathbb{R}\times\mathbb{R}$ for a Bernoulli random variable $\chi$ with the same parameter as the $\chi^i$'s. In \cite{Inglis_Talay} the authors also consider an integrate-and-fire model with asymmetric synaptic weights, but there the scaling is chosen in such a way that the mean-field limit becomes homogeneous.

Finally, we mention an interesting line of research on agent-based modelling for macroeconomic business cycles \cite{bouchaud,bouchaud2020}. By looking at the average over a large class of agents, one is led to a nonlinear PDE formulation that is essentially the evolution equation for a homogeneous version of \eqref{CMV}. Our framework highlights a tractable route to rigorously incorporating heterogeneity in these models, even after averaging, as the analysis takes place at the mean-field level from the beginning.

\subsection{Summary of the main technical contributions}

The most important contribution of this paper is to give a rigorous justification of the convergence discussed in \cite{FS20}, connecting a finite particle system to a suitable notion of solution for the mean-field problem \eqref{CMV}. The precise statements are given in Section \ref{main_results} below, while the proofs follow in Sections 3, 4, and 5. As per Remark \ref{sec:multi-type}, our analysis also covers the mean-field problem of \cite{NS18}.

Our main technical contributions concern the proof of Theorem \ref{thm:mf_conv}, which provides the rigorous mean-field convergence result. Overall, the path we follow is similar to \cite{DIRT15b, LS18a, NS17}, but new ideas are needed to deal with the heterogeneity and higher generality, as we discuss here.

Since \eqref{CMV} allows for a continuum of types, it is not immediately clear that we can hope to have solutions so that the processes $X_{u,v}$ are measurable in the type vectors $(u,v)$. In relation to this, a first novel aspect compared to \cite{DIRT15b, LS18a, NS17} is the identification of a suitable notion of solution, given by the formulation \eqref{eq:limit_markov_kernel_law}--\eqref{CMV_relaxed}. This addresses the aforementioned issue by working instead with the laws of the processes through a notion of random Markov kernels (see Definition \ref{defn:random_markov_kernel}). Here the random aspect must also address both the fact that the common factor should produce conditional limiting laws and the fact that the finite empirical measures need not become adapted to the common factor in the limit, but we note that this last point is analogous to \cite{LS18a}.

Once we have a workable notion of solution, the second main difference from \cite{DIRT15b, LS18a, NS17} lies in how we set up the probabilistic framework for working with the limit points in Section \ref{subsect:characterise_limit}. This is ultimately what allows us to characterise the limit points as solutions to the desired McKean--Vlasov system. The approach is inspired in part by \cite{LS18a}, but the precise constructions are closely tied in with the identification of our new notion of solution, and the reliance of this on the concept of random Markov kernels necessitates a more careful approach based on disintegration of measures.

The details of the constructions in Section \ref{subsect:characterise_limit} are crucial to the proof of the key continuity result, Lemma \ref{lem:first_conv_result}, as well as the subsequent martingale arguments in Propositions \ref{prop:martingale_props} and \ref{prop:P_independent}, which complete the proof of Theorem \ref{thm:mf_conv}. The proofs of these two propositions draw inspiration from \cite{LS18a}, but the arguments need to be tailored to the setting of Section \ref{subsect:characterise_limit} and, most importantly, we need a more delicate use of Skorokhod's representation theorem to exploit the form of Lemma \ref{lem:first_conv_result}. In particular, the requisite continuity of the relevant functionals does not in general hold almost surely for the limiting laws, but only when passing along suitable sequences.

In terms of ensuring that limit points exist, the non-exchangeability of the particle system means that a little more care is needed in the arguments related to M1-tightness in Lemma \ref{lem:contagion_at_0}  and Proposition \ref{prop:tightness}. However, this does not present significant new challenges once it is observed that we can obtain good estimates for each particle that are uniform across the type space.

Finally, we return to the key Lemma \ref{lem:first_conv_result}, where we confront the convergence of the contagious component of the particle system. In an earlier version of this paper, we relied on an adaptation of \cite[Lemma 3.13]{LS18a}, but a referee brought our attention to an unfortunate lacuna in the proof of that lemma. We have resolved this through Proposition \ref{prop:hitting_time_cont} and its use in the proof of Lemma \ref{lem:first_conv_result}. In relation to the original work \cite{DIRT15b}, Proposition \ref{prop:hitting_time_cont} plays a role akin to \cite[Lemma 5.9]{DIRT15b} and Lemma \ref{lem:first_conv_result} then plays a role akin to \cite[Lemma 5.6 \& Prop.~5.8]{DIRT15b}. Already because of the common Brownian motion, the proof of \cite[Lemma 5.9]{DIRT15b} does not apply to our setting. Nevertheless, inspired by their approach, we can instead average over the laws of the particles (accounting also for the heterogeneity) and arrive at an almost sure continuity result for the first hitting time of zero as a function on path space with respect to the M1 topology. We then give a short self-contained proof of the key Lemma \ref{lem:first_conv_result}, which is more streamlined than its analogues \cite[Lemma 5.6 \& Prop.~5.8]{DIRT15b} in the sense that it follows more directly from the critical use of the M1 topology in Proposition \ref{prop:hitting_time_cont} rather than relying both on this and another technical M1 argument as in \cite[Prop.~5.8]{DIRT15b}.

\section{Well-posedness and mean-field convergence}\label{main_results}

As is known already from the homogeneous versions, we cannot in general expect \eqref{CMV}  to be well-posed globally as a continuously evolving system, see e.g.~\cite[Theorem 1.1]{HLS18}. Due to the feedback effect from the gradual loss of mass, the rate of change for the contagion processes $t\mapsto \mathcal{L}_l(t)$ can explode in finite time and jumps in the loss of mass may materialise from within the system itself.

First we show that this need not always be the case, by presenting a simple uniqueness result under a smallness condition on the interactions which guarantees that \eqref{CMV} evolve continuously in time independently of the realisations of the common factor. Moreover, we show that, when restricting to the idiosyncratic setting of $\rho=0$, where the common factor $B_0$ disappears, one can obtain a \emph{local} result on uniqueness and regularity up to an explosion time. Next, we introduce a finite particle system, which is a general formulation of the one coming from the interbank model in \cite{FS20}. It is this particle system that underpins our interest in \eqref{CMV}, and our main result comes down to showing that, in a suitable sense, there is weak convergence to the heterogeneous McKean--Vlasov system \eqref{CMV} as the number of particles tends to infinity.

\begin{rem}[Multi-type versions of the system]\label{sec:multi-type}

Also motivated by the modelling of contagion in financial markets, \cite{NS18} was the first to study a multi-type system very close to \eqref{CMV}. Specifically, the authors use a generalised Schauder fixed point approach to show existence of solutions for a coupled McKean--Vlasov system of the form
\begin{equation}\label{eq:NS_MV}
\begin{cases}
X^x_t = X^x_0 + Z^x_t + C(x)\displaystyle\int_{\mathcal{X}} G( \mathbb{P}(\tau^y > t) ) \mu(x,dy), \\[2pt]
\tau^x = \inf\{ t \geq 0 : X^x_t \leq 0 \},\quad x\in\mathcal{X},
\end{cases}
\end{equation}
where $\mathcal{X}$ is an abstract finite set and each $Z^x$ is an exogenous stochastic process with suitable regularity properties. A typical example would be $dZ^x_t=b_x(t)dt+\sigma_x(t)dW^x_t$. As per \cite[Remark 2.5]{NS18}, their analysis also allows for a common Brownian motion. By the assumptions on $G$ in  \cite{NS18}, we can define a non-decreasing and non-negative function $\tilde{F}:[0,1]\rightarrow \mathbb{R}_+$ by $\tilde{F}(s):=-G(1-s)$, for $s\in[0,1]$, with $\tilde{F}(0)=0$. Taking $k:=|\mathcal{X}|$, it is then straightforward to define $\varpi$ as a convex combination of $k$ point masses on fixed non-negative vectors $(\mathrm{u}^1,\mathrm{v}^1),\ldots,(\mathrm{u}^k,\mathrm{v}^k)$ in such a way that \eqref{eq:NS_MV} becomes equivalent to \eqref{CMV} with $F=\mathrm{Id}$, $g_{u,v} \equiv 1$, and
\[
\mathcal{L}_{l}(t)  = \!\displaystyle\int_{\mathbb{R}^k\times\mathbb{R}^k}  u_l\, \tilde{F}\bigl(\mathbb{P}( t\geq \tau_{u,v} \, | \, B_0) \bigr) d\varpi(u,v), \quad t \geq0,\quad l=1,\ldots,k.
\]
Since $\varpi$ defined from \eqref{eq:NS_MV} has finite support, \eqref{CMV} further simplifies to a system of $k$ representative processes $X_{\mathrm{u}^l,\mathrm{v}^l}(t)$, for $l=1,\ldots,k$, in this case. We note that it would not pose any problems to incorporate a function $\tilde{F}$, as above, into our analysis (in addition to $F$), but we leave this out to avoid clouding the notation (noting also that our main motivating applications do not call for it).
\end{rem}

\subsection{On the jumps of the heterogeneous McKean--Vlasov system}\label{sec:jump_sizes}

 It is known already for the simplest homogeneous version of \eqref{CMV} (as studied in \cite{DIRT15b, HLS18, NS17}) that, if the feedback effect is strong enough, then solutions \emph{cannot} be continuous globally in time \cite[Theorem 1.1]{HLS18} and one is furthermore left with a \emph{non-unique} choice of the jump times and jump sizes even when restricting to c\`adl\`ag solutions (see \cite[Example 2.2]{HLS18} for a stylised example).

 For clarity, let us consider the multi-type version of \eqref{CMV} with $F=\mathrm{Id}$ and $g_{u,v}\equiv1$, as just discussed in Remark \ref{sec:multi-type} above (concerning \eqref{eq:NS_MV}, we also take $\tilde{F}:=\mathrm{Id}$). Due to the finite support of $\varpi$, we can then let $\pi_l := \varpi(\{ (\mathrm{u}^l ,\mathrm{v}^l)  \}) $ and observe that the contagion processes simplify to
 \[
 \mathcal{L}_{l}(t)  =\pi_l \mathbb{P}( t\geq \tau_{\mathrm{u}^l,\mathrm{v}^l} \, | \, B_0), \quad \text{for} \quad l=1,\ldots,k.
 	\]
 In the homogeneous problem, it is natural to let the jump times and jump sizes be fixed by the so-called physical jump condition, first introduced in \cite{DIRT15b}. A possible analogue of this for the multi-type system \eqref{eq:NS_MV} was briefly discussed in \cite{NS18}. The authors did not attempt to work with this condition, but rather emphasised that it would appear less clear if there is a condition more natural than others in the multi-type setting. After some straightforward manipulations of \cite[Equations (2.16)-(2.17)]{NS18}, the condition discussed in \cite{NS18} can be seen to take the simpler form
\begin{multline}\label{eq:NS_PJC_simplified}
\qquad\qquad\qquad\Delta \mathcal{L}_l(t) = \pi_l \mathbb{P}\bigl(X_l(t-)\in [0,D_t] ,\, t\leq\tau_l \, | \, B_0\bigr), \quad l=1,\ldots,k, \\ D_t := \inf \Bigl\{ z > 0     : 
\sum_{i=1}^{k}  \mathrm{v}^i_j  \pi_i \mathbb{P}\bigl(X_i(t-)\in [0,z] ,\, t\leq\tau_i \, | \, B_0\bigr)   < z\;\; \forall\, j=1,\ldots, k,
\Bigr\},\qquad\quad
\end{multline}
where we have set $X_l:=X_{\mathrm{u}^l,\mathrm{v}^l}$ and $\tau_l:=\tau_{\mathrm{u}^l,\mathrm{v}^l}$ for $l=1,\ldots, k$. If $k=1$, this is precisely the aforementioned physical jump condition for the homogeneous problem. When $k\geq 2$, however, an issue presents itself, which highlights some of the new difficulties in the multi-type setting. Indeed, \eqref{eq:NS_PJC_simplified} dictates that, for each type $l=1,\ldots,k$, the jump size of $\mathcal{L}_l$ at time $t$ is equal to a multiple $\pi_l$ of the total mass given by the density of $X_l(t-)$ on $[0,D_t]$. At the same time, a careful inspection of the dynamics in \eqref{CMV} reveal that such jump sizes, for each type, must cause mean-field particles of type $j$ to shift exactly the mass given by its density on $[0,\sum_{i=1}^k\mathrm{v}^i_j\pi_i\Delta\mathcal{L}_i(t)]$ through the origin, meaning that it is this amount of mass that ends up being absorbed. However, this does \emph{not} agree with the previous observation, unless it happens that $D_t$ equals $\sum_{i=1}^k\mathrm{v}^i_j\pi_i\Delta\mathcal{L}_i(t)$ for all $j=1,\ldots,k$ at the given time. Thus, the jump condition \eqref{eq:NS_PJC_simplified} will not in general be consistent with the prescribed dynamics of the mean-field particles.

Throughout the paper, we shall often refer to the total loss process
\begin{equation}\label{eq:total_loss}
\mathbf{L}_v(t):=\sum_{l=1}^k v_l \mathcal{L}_l(t),\quad \text{for all} \quad t\geq 0, \; v\in S(\mathbf{v}),
\end{equation}
defined in agreement with the general formulation \eqref{CMV1}. With this notation, it was suggested in \cite{FS20} that a sensible condition for the jump sizes could amount to insisting that, for every $t\geq 0$,
\begin{equation}\label{eq:mf_cascade_cond}
\begin{cases}
\Delta \mathcal{L}_l(t)=\Xi_l\bigl(t, \Delta \mathbf{L}_{(\cdot)}(t) \bigr),\quad \displaystyle \Delta \mathbf{L}_v(t) =  \lim_{\varepsilon\downarrow 0}  \lim_{m\rightarrow \infty}\Delta_{t,v}^{(m,\varepsilon)}, \\[7pt]
\displaystyle \Delta_{t,v}^{(m,\varepsilon)} :=\sum_{l=1}^kv_l \Xi_l(t;\varepsilon + \Delta_{t,(\cdot)}^{(m-1,\varepsilon)}) ,\quad \displaystyle \Delta_{t,v}^{(0,\varepsilon)}:= \sum_{1=1}^kv_l \Xi_l(t;\varepsilon),
\end{cases}
\end{equation}
almost surely, where the maps $\Xi_l$, for $l=1,\ldots, k$, are defined by
\begin{align}\label{eq:Xi_map}
\Xi_l(t;f) &:= \int_{\mathbb{R}^k\times\mathbb{R}^k} u_l \mathbb{P} \bigl(     
X_{u,v}(t-)\in[0,\Theta_{u,v}(t;f)],\; t\leq \tau_{u,v} \, | \, B_0 \bigr) d\varpi(u,v), \\[3pt]
\Theta_{u,v}(t;f) &:= F\biggl(  \int_0^{t-}\!\!g_{u,v}(s)d\mathbf{L}_v(s) + g_{u,v}(t)f(v)   \biggr) - F\biggl(  \int_0^{t-}\!\!g_{u,v}(s)d\mathbf{L}_v(s)   \biggr),\label{eq:Theta_map}
\end{align}
for all functions $f: \mathbb{R}^k \rightarrow \mathbb{R}_+$. This corresponds to shocking the system by a small amount and tracking the contagious effects for infinitely many rounds ($m\rightarrow \infty$) before then sending the order of the shock to zero ($\varepsilon \downarrow 0$). It was observed in \cite[Prop.~3.5]{FS20} that any c\`adl\`ag solution \eqref{CMV} must satisfy $\Delta \mathbf{L}_v(t) = \sum_{l=1}^k v_l \Xi_l\bigl(t, \Delta \mathbf{L}_{(\cdot)}(t) \bigr)$ for all $v \in S(\mathbf{v})$ at every $t\geq 0$. Noting that $\Delta_{t,v}^{(m,\varepsilon)}$ is a bounded sequence, increasing in both $m$ and $\varepsilon$, we can apply dominated convergence in \eqref{eq:mf_cascade_cond} to see that the aforementioned constraint is satisfied, so \eqref{eq:mf_cascade_cond} is consistent with the dynamics in \eqref{CMV}.

In Section \ref{Subsect:particle_sys}, we shall work with a similar looking condition \eqref{particle_cascade_cond1} for the jump sizes of the finite particle system. We suspect that \eqref{eq:mf_cascade_cond} will be satisfied by the limit points of this particle system. However, we have not yet been able to prove this. Compared to the arguments in \cite{DIRT15b, LS18a}, which show that limit points of the homogeneous particle system must satisfy the physical jump condition mentioned above, it becomes harder to work with notions of minimality for comparing solutions and it is complicated to keep track of the influence of the heterogeneity.

In this paper, the only result we prove involving \eqref{eq:mf_cascade_cond} is the last part of Theorem \ref{thm:idio_reg} below. Other than that, any further analysis of jump size conditions is left for future research along with the question of whether \eqref{eq:Xi_map} is satisfied by the limits points of the particle system.

\subsection{Two results on uniqueness and regularity}\label{subsect:uniqueness_regularity}

In \cite[Theorem 3.4]{FS20}, a simple continuity condition was derived for a particular version of \eqref{CMV}, which ensures continuity globally in time for all realisations of the common noise $B^0$. Here we show that this condition naturally leads to pathwise global uniqueness of \eqref{CMV}.

\begin{thm}[Common noise uniqueness under smallness condition]\label{thm:global_unique} Let Assumptions \ref{uv_assump} and \ref{MV_assump} be satisfied.  Moreover, suppose that all the mean-field particles $X_{u,v}$ have initial densities $V_{0}(\cdot|u,v)$ satisfying the condition
	\begin{equation}\label{smallness_cond}
\Vert V_{0}(\cdot|u,v) \Vert_{\infty} < g_{u,v}(0)\Vert F \Vert_{\mathrm{Lip}} \frac{1 }{\max\{u\cdot \hat{v} \; | \;  \hat{v}\in S(\mathbf{v}) \;\mathrm{s.t.}\; u\cdot \hat{v} >0\}},
\end{equation}
for all  $(u,v) \in S(\mathbf{u},
\mathbf{v})$, where $1/\max \emptyset=+\infty$. Then each $X_{u,v}$ must be continuous in time, for any solution to \eqref{CMV}, and there is uniqueness of solutions.
\end{thm}

We note that \eqref{smallness_cond} corresponds to \cite[(3.15)]{FS20} except for an explicit choice of $g_{u,v}$ in \cite{FS20}. Thus, the statement about continuity in time of each $X_{u,v}$ follows precisely as in \cite[Theorem 3.4]{FS20}. The uniqueness part of the theorem is proved in Section \ref{Sect:com_nois}, using the ideas from \cite{LS18b}. Unfortunately, we are unable to say much about general uniqueness and regularity properties of \eqref{CMV} in the absence of the above smallness condition \eqref{smallness_cond}. Nevertheless, we do have the following \emph{local} uniqueness and regularity result for the purely `idiosyncratic' problem without the common noise. 

\begin{thm}[Idiosyncratic regularity and local uniqueness up to explosion]\label{thm:idio_reg} Let Assumptions \ref{uv_assump}, \ref{MV_assump}, and \ref{X0_assump} be in place. Then there exists a solution to \eqref{CMV} on some interval $[0,T_\star)$ for which the contagion processes $\mathcal{L}_1\ldots,\mathcal{L}_k$ are continuously differentiable up to the explosion time
	\[
	T_{\star}:= \sup \Bigl\{ t>0 :  \sum_{l=1}^k \Vert \partial_t\mathcal{L}_l(\cdot) \Vert_{L^2(0,t)} < \infty   \Bigr\} >0,
	\]
and, for every $t<T_\star$, we have $\partial_s \mathcal{L}_l(s) \leq K s^{-(1-\beta)/2}$ on $[0,t]$ for some constant $K> 0$, for each $l=1,\ldots, k$. Moreover, if there is another c\`adl\`ag solution to \eqref{CMV} with jump sizes smaller than or equal to those given by the cascade condition \eqref{eq:mf_cascade_cond}, then such a solution must coincide with the above solution on $[0,T_\star)$.
\end{thm}

The proof of Theorem \ref{thm:idio_reg} is the subject of Section \ref{Sect:Idio}.  It is proved by suitably adapting the arguments from \cite{HLS18}. For Theorem \ref{thm:idio_reg} to be truly interesting, one would need to know that the cascade condition \eqref{eq:mf_cascade_cond} is satisfied by the limit points of the particle system presented in the next section. We believe this to be true, but the question is left for future research. Likewise, this paper does not address whether there exist solutions satisfying \eqref{eq:mf_cascade_cond} beyond their first jump time.

\subsection{The connection to a finite particle system}\label{Subsect:particle_sys}

In this section, we introduce a general form of the particle system studied in \cite{FS20}. For details on the motivating application to solvency contagion, we refer to the balance sheet based formulation in \cite[Section 2]{FS20} and the reformulation as a stochastic particle system in \cite[Proposition 3.1]{FS20}. Here we consider the general system of interacting real-valued c\`adl\`ag processes $\{X_i\}_{i=1}^n$ satisfying
	\begin{equation}\label{particle_sys}
\begin{cases}
dX_i(t) =  b_{i}(t)dt  +  \sigma_{i}(t)dW_i(t) -dF\Bigl( \displaystyle \sum_{l=1}^{k} v_{l}^i \! \int_0^t g_{i}(s) d\mathcal{L}^n_{l,i}(s) \Bigr)\\
\displaystyle  \mathcal{L}^n_{l,i}(t) =\frac{1}{n} \sum_{j =1 }^{n} \mathbf{1}_{j\neq i}u_{l}^j \mathbf{1}_{t\geq\tau_j}, \quad \tau_i=\inf \{ t\geq 0 :  X_i(t) \leq 0  \},\\[4pt]
W_i(t)=\rho B_0(t) + \sqrt{1-\rho^2}B_i(t), \quad
\displaystyle X_i(0) = \varphi(u^i,v^i,\bar{u}^i,\bar{v}^i,\xi_i ),\vspace{2pt}
\end{cases}
\end{equation}
where $\{B_i\}_{i=0}^n$ is a family of independent Brownian motions, the vectors $\bar{u}^i,\bar{v}^i\in\mathbb{R}^k$ are given by
\[
\bar{u}^i_l:=\frac{1}{n}\sum_{j=1}^n \mathbf{1}_{j \neq i}u^j_l \quad\text{and}\quad \bar{v}^i_l:=\frac{1}{n}\sum_{j=1}^n \mathbf{1}_{j \neq i} v^j_l,
\]
and $\{\xi_i\}_{i=1}^n$ is a sequence of i.i.d.~random variables with common law $\mathbb{P}_\xi$.  As in the mean-field problem, we require solutions to satisfy $\mathcal{L}_{l,i}^n(0)=0$ for all $l=1,\ldots,k$ and all $i=1,\ldots, n$. The coefficients $b_i$, $\sigma_i$, and $g_i$ are taken to be continuous functions of the form
\[
b_i(t) =b(u^i,v^i,t),\quad\sigma_i(t)=\sigma(u^i,v^i,t),\quad \text{and} \quad  g_i(t)=g(u^i,v^i,\bar{u}^i,\bar{v}^i,t),
\]
with continuity in all variables. In order to have convergence of this system, as $n\rightarrow\infty$, the key requirement is that there is an underlying distribution $\varpi \in \mathcal{P}(\mathbb{R}^k\times\mathbb{R}^k)$ for the indexing vectors such that we have weak convergence
\begin{equation}\label{eq:emp_meas_network}
\frac{1}{n} \sum_{i=1}^n d(\delta_{u^i}\otimes \delta_{v^i})(u,v) \rightarrow d\varpi(u,v), \quad \text{as}\quad n\rightarrow \infty,
\end{equation}
in $\mathcal{P}(\mathbb{R}^k\times\mathbb{R}^k)$, and in turn also weak convergence
\begin{equation}\label{eq:emp_meas_network1}
\frac{1}{n} \sum_{i=1}^n d(\delta_{u^i}\otimes \delta_{v^i}\otimes \delta_{X_i(0)})(u,v,x) \rightarrow  d\nu_0 (x|u,v)d\varpi(u,v), \quad \text{as}\quad n\rightarrow \infty,
\end{equation}
in $\mathcal{P}(\mathbb{R}^k\times\mathbb{R}^k\times(0,\infty))$, where $\nu_0 (\cdot |u,v):=\mathbb{P}_\xi\circ \tilde{\varphi}^{-1}_{u,v}$ for $\tilde{\varphi}_{u,v}(x):=\varphi(u,v,\mathbb{E}[\mathbf{u}],\mathbb{E}[\mathbf{v}],x)$ with the usual notation $(\mathbf{u},\mathbf{v})\sim \varpi$. This holds, e.g., when  $\{(u^i,v^i)\}_{i=1}^\infty$ are i.i.d.~samples from a desired distribution $\varpi$, drawn independently of $\{\xi_i\}_{i=1}^\infty$, which is the setting of the financial model in \cite{FS20}.

The full set of assumptions for the particle system are collected in a single statement below. As we detail in Lemma \ref{eq:uv-compact-nonneg}, these assumptions automatically ensure that the limiting distribution $\varpi$ in \eqref{eq:emp_meas_network} satisfies Assumption \ref{uv_assump} from above.

\begin{assump}[Structural conditions for the particle system]\label{assump:paramters_emperical}
	We assume the coefficients satisfy Assumption \ref{MV_assump}, where $g$ is now of the slightly more general form $g(u,v,\bar{u},\bar{v},t)$. Additionally, we assume that $\varphi$ in the definition of $X_{i}(0)$ is a measurable function, and that $\sigma_{u^i,v^i}(\cdot) \in\mathcal{C}^{\beta}([0,T])$ for some $\beta>1/2$ with the H{\"o}lder norms being bounded uniformly in $i=1,\ldots,n$ and $n\geq 1$. Furthermore, we assume there is a constant $C>0$ such that $|u_i|+|v_i|\leq C$ for all $i=1,\ldots,n$ and $n\geq 1$, and we assume that $u^i\cdot v^j\geq 0$ for all $i,j=1,\ldots,n$ and $n\geq 1$. Finally, we ask that the weak convergence \eqref{eq:emp_meas_network}-\eqref{eq:emp_meas_network1} holds with $\nu_0$ satisfying Assumption \ref{X0_assump}.
\end{assump}

After some inspection, the particle system \eqref{particle_sys} reveals itself to be non-unique, as it is, since the dynamics may allow for different sets of absorbed particles whenever a particle reaches the origin. Thus, it becomes necessary to make a choice. In the interbank model of \cite{FS20}, it was shown that a Tarski fixed point argument gives a greatest and least c\`adl\`ag clearing capital solution to \eqref{particle_sys}, and \cite[Proposition 3.1]{FS20} established that selecting the greatest clearing capital solution from for this system amounts to letting the sets of absorbed particles be given by the discrete `cascade condition' that we now describe. 

By analogy with the total loss process \eqref{eq:total_loss}, we also introduce a total loss process
\begin{equation}
(v,t)\mapsto \mathbf{L}^n_{v}(t):= \frac{1}{n}\sum_{j=1}^n v\cdot u^j \mathbf{1}_{t\geq \tau_j} = \sum_{l=1}^k v_l
\bigl( 
\mathcal{L}_{l,i}^n(t)+u_l^i\mathbf{1}_{t\geq \tau^i}
\bigr),
\end{equation}
for the finite particle system, and, similarly to \eqref{eq:Xi_map}--\eqref{eq:Theta_map}, we then define the auxiliary maps
\begin{equation}\label{eq:maps_cascade}
\begin{cases}
\Xi^n(t,f,v):= \displaystyle \sum_{l=1}^k v_l 	\Xi^n_l(t;f), \quad 	\Xi^n_l(t,f):=  \sum_{j= 1}^{n} u_{l}^j  \mathbf{1}_{ \{ X_j(t-) \in [0, \Theta^n_j(t; f) ]  , \, t\leq \tau_j \} } ,\\
\Theta^n_j(t,f) := \displaystyle F \Bigl(  \int_0^{t-} \!g_j(s)d\mathbf{L}^n_{v^j}(s) + g_j(t)f(v^j) \Bigr)  - F \Bigl(  \int_0^{t-} \!g_j(s)d\mathbf{L}^n_{v^j}(s) \Bigr),\vspace{-2pt}
\end{cases}
\end{equation}
for $l=1,\ldots,k$ and $j=1,\ldots,n$, where we stress that, at any time $t$, the values of $\Xi^n(t,f,v)$ and $\Theta_j^n(t,f)$ are completely specified in terms of the `left-limiting state' of the particle system. That is, the maps are well-defined at time $t$ without any a priori knowledge of which particles (if any) will be absorbed at time $t$. Armed with these definitions, we declare that, at any time $t$, the set of absorbed particles (possibly the empty set) is given by
\begin{equation}\label{eq:default_set}
\mathcal{D}_t :=\{ i \; | \; \tau_i = t  \} := \bigl\{i \; | \; X_i(t-)-\Theta^n(t;\Delta \mathbf{L}_{(\cdot)}^n(t);i) \leq 0,\; \tau_i\geq t  \bigr\},
\end{equation}
with corresponding jump sizes
\[
\Delta \mathcal{L}^{n}_{l,i}(t) = \frac{1}{n} \sum_{j \in \mathcal{D}_t} \mathbf{1}_{j\neq i}u_{l}^j=\Xi^n_l(t,\Delta\mathbf{L}_{(\cdot)}^n(t))-u_l^i\mathbf{1}_{i\in \mathcal{D}_t},
\]
where the mapping $v\mapsto \Delta \mathbf{L}^n_v(t)$ in \eqref{eq:default_set} is determined by the iterated cascade condition
\begin{equation}\label{particle_cascade_cond1}
\Delta\mathbf{L}_{v}^{n}(t)={\displaystyle \lim_{m\rightarrow n}\Delta_{t,v}^{n,(m)}},\quad\text{for}\quad\begin{cases}
\Delta_{t,v}^{n,(m)}:=\Xi^{n}(t,\Delta_{t,(\cdot)}^{n,(m-1)},v), & m=1,\ldots,n,\\[4pt]
\Delta_{t,v}^{n,(0)}:=\Xi(t;0,v).
\end{cases}
\end{equation}
Together, \eqref{eq:default_set} and \eqref{particle_cascade_cond1} specify the set of absorbed particles at any given time and the corresponding jumps in the total feedback felt by the particles. Rephrasing the main conclusion of \cite[Proposition 3.1]{FS20}, the two defining properties of this cascade condition for the contagion mechanism are:~(i) the specification of the set of absorbed particles $\mathcal{D}_t$ is consistent with the dynamics \eqref{particle_sys}, and (ii) it gives the c\`adl\`ag solution with the greatest values of $(X_1(t),\ldots,X_n(t))$ in the sense that any other consistent c\`adl\`ag  specification of $\mathcal{D}_t$ would yield lower values for at least one of the particles while not increasing the values of any of them. This is what gives us the greatest clearing capital in the interbank system from \cite{FS20}.

To connect the particle system \eqref{particle_sys} with the coupled McKean--Vlasov problem \eqref{CMV}, we will work with the empirical measures
\begin{equation}\label{empirical_measures_parameters}
\mathbf{P}^n:= \frac{1}{n} \sum_{i=1}^n \delta_{u^i}\otimes \delta_{v^i} \otimes \delta_{X_i}, \quad n\geq 1,
\end{equation}
where the family $\{X_i\}_{i=1}^n$ is the unique strong solution to \eqref{particle_sys} equipped with the cascade condition \eqref{eq:default_set}-\eqref{particle_cascade_cond1}. We stress that each $X_i$ is viewed as a random variable with values in the Skorokhod path space $D_\mathbb{R}=D_\mathbb{R}[0,T]$ for a given terminal time $T>0$. When working to identify the limiting behaviour of \eqref{empirical_measures_parameters}, as $n\rightarrow \infty$, it will be helpful to have a precise concept of random Markov kernels, which we introduce next.

\begin{defn}[Random Markov kernels]\label{defn:random_markov_kernel} By a Markov kernel $P=\{P_{x}\}_{x\in \mathcal{X}}$ on $D_\mathbb{R}$, for a given Polish space $\mathcal{X}$, we understand a probability measure valued mapping $x\mapsto P_{x}\in\mathcal{P}(D_\mathbb{R})$ so that $x \mapsto P_{x}(A) $ is Borel measurable for any $A \in \mathcal{B}(D_\mathbb{R})$. We will say that $\mathbf{P}=\{\mathbf{P}_{\!x}\}_{x\in \mathcal{X}}$ is a \emph{random Markov kernel} on $D_\mathbb{R}$ if the mapping $x\mapsto \mathbf{P}_{\!x}$ assigns to each $x\in \mathcal{X}$ a \emph{random} probability measure $ \mathbf{P}_{\!x}: \Omega  \rightarrow \mathcal{P}(D_\mathbb{R})$ on a fixed Polish background space such that $(x,\omega)\mapsto \mathbf{P}_{\!x}(\omega)(A)$ is Borel measurable for all $A\in \mathcal{B}(D_\mathbb{R})$.
	\end{defn}

We can now state our main result about the limit points of the empirical measures \eqref{empirical_measures_parameters} as we send the number of particles to infinity.

\begin{thm}[Mean-field convergence]\label{thm:mf_conv} 
	Let $(\mathbf{P}^n)_{n\geq1}$ be the sequence of empirical measures defined in \eqref{empirical_measures_parameters}, for $n\geq1$, and let Assumption \ref{assump:paramters_emperical} be satisfied. Then any subsequence of $(\mathbf{P}^n)_{n\geq1}$ has a further subsequence, still indexed by $n$, such that $(\mathbf{P}^n,B_0)$ converges in law to a limit point $(\mathbf{P}^*,B_0)$. For any such limiting pair, we can identify $\mathbf{P}^*$ with a random Markov kernel $\{\mathbf{P}^*_{u,v}\}_{(u,v)\in \mathbb{R}^k\times \mathbb{R}^k}$ on $D_\mathbb{R}$ (in the sense of Definition \ref{defn:random_markov_kernel}) satisfying
	\begin{equation}\label{eq:limit_markov_kernel_law}
	\int \!\phi\, d \mathbf{P}^\star_{u,v}=  \mathbb{E}[\phi(X^\star_{u,v})\,|\,B_0,\mathbf{P}^\star],\qquad \text{for} \quad (u,v)\in \mathbb{R}^k \times \mathbb{R}^k,
	\end{equation}
	for all Borel measurable functions $\phi: D_\mathbb{R}\rightarrow \mathbb{R}$, where the processes $X^\star_{u,v}$ are c\`adl\`ag solutions to the coupled McKean--Vlasov system
		\begin{equation}\label{CMV_relaxed}
	\begin{cases}
	dX^\star_{u,v}(t) =  b_{u,v}(t)dt  + \!\sigma_{u,v}(t)dW_{u,v}(t) - d F\Bigl(  \displaystyle\sum_{l=1}^k v_l \!\int_0^t g_{u,v}(s) d\mathcal{L}^\star_l(s)   \Big), \\[4pt]
	\mathcal{L}^\star_{l}(t)  = \!\displaystyle\int_{\mathbb{R}^k\times\mathbb{R}^k}  u_l \mathbb{P}( t\geq \tau^\star_{u,v} \, | \, B_0, \mathbf{P}^\star) d\varpi(u,v), \quad l=1,\ldots,k,\\[10pt]
	\tau^\star_{u,v} =\inf \{ t > 0 : X^\star_{u,v}(t) \leq 0 \}, \quad W_{u,v}(t)=\rho B_0(t) + \sqrt{1-\rho^2}B_{u,v}(t).\vspace{2pt}
	\end{cases}
	\end{equation}
	Furthermore, the pair $(\mathbf{P}^\star,B_0)$ is independent of the particle-specific Brownian motions $B_{u,v}$ and the initial conditions $X_{u,v}(0)$.
\end{thm}

The proof of Theorem \ref{thm:mf_conv} is the subject of Section \ref{Sect:conv}. Here we only briefly discuss the intuition behind the identification of $\mathbf{P}^\star$ as a random Markov kernel and what the proof of Theorem \ref{thm:mf_conv} can then be boiled down to. Firstly, note that the empirical measures $\mathbf{P}^n$ are random probability measures on $\mathbb{R}^k\times \mathbb{R}^k \times D_\mathbb{R}$, so it is natural to work with weak convergence on this space. Thus, a given  limit point $(\mathbf{P}^\star , B^0)$ yields, in the first instance, a random probability measure on $\mathbb{R}^k\times \mathbb{R}^k \times D_\mathbb{R}$. Nevertheless, due to \eqref{eq:emp_meas_network}, we can consider $\mathbf{P}^\star$ as a random variable $\mathbf{P}^\star : \Omega \rightarrow \mathcal{P}_{\!\varpi}( D_\mathbb{R})$, where the range $\mathcal{P}_{\varpi}(D_\mathbb{R})$ denotes the space of all Borel probability measures $\mu\in \mathcal{P}(\mathbb{R}_k\times \mathbb{R}_k \times D_\mathbb{R})$ with fixed marginal $\mu \circ \pi_{1,2}^{-1}=\varpi$ for $\pi_{1,2}(u,v,\mu)=(u,v)$. This space $\mathcal{P}_{\!\varpi}(D_\mathbb{R})$ can then be seen to be isomorphic to the space of all Markov kernels $\{\nu_{u,v}\}_{(u,v)\in\mathbb{R}^k \times \mathbb{R}^k}$ for $D_\mathbb{R}$ under the identification
\[
\{\nu_{u,v}\}_{(u,v)\in\mathbb{R}^k \times \mathbb{R}^k} \cong \nu(d\eta) := \int_{\mathbb{R}^k \times \mathbb{R}^k} \nu_{u,v}(d\eta)d\varpi(u,v).
\]
For any limit point of the empirical measures, we thus have a random probability measure $\mathbf{P}^\star: \Omega \rightarrow \mathcal{P}_\varpi(D_\mathbb{R})$ corresponding to a random Markov kernel for $D_\mathbb{R}$ via the identification
\[
 \{\mathbf{P}^\star_{u,v}(\omega)\}_{(u,v)\in\mathbb{R}^k \times \mathbb{R}^k} \cong \mathbf{P}^\star(\omega),
\]
understood $\omega$ by $\omega$, where the key requirement is that the mapping $(u,v,\omega)\mapsto \mathbf{P}^\star_{u,v}(\omega)(A)$ must be measurable for every $A\in \mathcal{B}(D_\mathbb{R})$.

In view of the above, checking that a limit point $\mathbf{P}^\star$ yields a solution to the desired McKean--Vlasov system \eqref{CMV_relaxed} amounts to checking that each $\mathbf{P}^\star_{u,v}$ realises the conditional law of $X_{u,v}^\star $ on the path space $D_\mathbb{R}$ given the pair $(B_0,\mathbf{P}^\star)$, which is precisely \eqref{eq:limit_markov_kernel_law}. In particular, we are looking for the relations
		\begin{align}
\int f(u,v)\phi(\eta) \,d\mathbf{P}^\star(u,v,\eta)   &= \int_{\mathbb{R}^k\times \mathbb{R}^k } f(u,v) \mathbb{E}[\phi(X^\star_{u,v})\,|\,B_0,\mathbf{P}^\star]d\varpi(u,v) \nonumber\\
&= \int_{\mathbb{R}^k\times\mathbb{R}^k} f(u,v) \int_{D_\mathbb{R}} \phi(\eta) \,d\mathbf{P}_{u,v}^\star(\eta)d\varpi(u,v)\nonumber
	\end{align}
for all Borel measurable $f:\mathbb{R}^k \times \mathbb{R}^k \rightarrow \mathbb{R}$ and $\phi: D_\mathbb{R}\rightarrow \mathbb{R}$. The proof of Theorem \ref{thm:mf_conv} is implemented in this spirit, by constructing a random Markov kernel from a given limiting pair $(\mathbf{P}^\star,B_0)$ in a way that the above relations can be conveniently verified (see, in particular, Proposition \ref{prop:P_independent}).

\begin{rem}[Relaxed solution concept] We note that the empirical measures $\mathbf{P}^n$ become independent of the idiosyncratic noise in the limit, but the theorem does not guarantee that it becomes strictly a function of the common noise, which explains the appearance of $\mathbf{P}^\star$ on the right-hand side of \eqref{eq:limit_markov_kernel_law} and \eqref{CMV_relaxed}. Nonetheless, $(\mathbf{P}^\star,B^0)$ being independent of the idiosyncratic noise means that this system is qualitatively the same as \eqref{CMV}. Analogously to \cite{LS18a}, \eqref{eq:limit_markov_kernel_law}-\eqref{CMV_relaxed} is a `relaxed' solution of the heterogeneous system \eqref{CMV}. This is similar to the weak solution concepts for mean-field games introduced in \cite{carmona2016} and later considered for general McKean--Vlasov SDEs in \cite{hammersley2021}.
\end{rem}

Whenever there is pathwise uniqueness of \eqref{CMV_relaxed}, we get a unique mean-field limit, given by a pair $(\mathbf{P}^\star,B_0)$ for which $\mathbf{P}^\star$ is in fact $B_0$-measurable. In particular, we then have full convergence of the particle system and the additional appearance of $\mathbf{P}^\star$ in the conditioning on the right-hand sides of \eqref{eq:limit_markov_kernel_law} and \eqref{CMV_relaxed} can be dropped. Since we have established results on conditions for uniqueness with continuous dynamics, we get the following result.


\begin{thm}[Full convergence with continuous dynamics]\label{thm:yamada-watanabe}
	Under the assumptions of Theorem \ref{thm:global_unique}, there is uniqueness of the limit points in Theorem \ref{thm:mf_conv}, and therefore $(\mathbf{P}^n,B_0)$ converges in law to a unique limit $(\mathbf{P}^\star,B_0)$. Moreover, this limit is now characterised by a Brownian motion $B_0$ and a random Markov kernel $\{\mathbf{P}^\star_{u,v}\}_{(u,v)\in \mathbb{R}^k\times \mathbb{R}^k}$ satisfying
	\begin{equation}\label{eq:boldP}
	\int \! \phi \,d\mathbf{P}^\star_{u,v} = \mathbb{E}[\phi(X_{u,v})\,|\,B_0],
	\end{equation}
	for all Borel-measurable $\phi: D_\mathbb{R}\rightarrow \mathbb{R}$,
	where the system $\{X_{u,v}\}_{(u,v)\in \mathbb{R}^k\times \mathbb{R}^k}$ constitutes the unique family of continuous processes obeying the dynamics \eqref{CMV} with $B_0$ independent of each $B_{u,v}$.
\end{thm}

Following the arguments in \cite[Theorem 2.3]{LS18b}, the proof of Theorem \ref{thm:yamada-watanabe} is a consequence of Theorem \ref{thm:mf_conv} together with the estimates in Section \ref{Sect:com_nois}.


\section{Limit points of the particle system}\label{Sect:conv}

This section is dedicated to the proof of Theorem \ref{thm:mf_conv}. Our first task is to establish a suitable tightness result for the pairs $(\mathbf{P}^n,B_0)$, which is done in Section \ref{subsect:tight}, and then we conclude in Section \ref{subsect:characterise_limit} that the resulting limit points $(\mathbf{P}^\star,B_0)$ can be characterised as solutions to \eqref{eq:limit_markov_kernel_law}-\eqref{CMV_relaxed}.

To implement these arguments, we follow the broad approach of \cite{LS18a}, which in turn builds on several antecedent ideas from \cite{DIRT15b}. Since \cite{LS18a} deals with a symmetric particle system, substantial adjustments to the arguments are needed, but the overall flavour remains that of identifying a limiting martingale problem, as is typical for convergence results of this type. More generally, we stress that, our particle system is not exchangeable and the positive feedback from defaults is of a rather singular nature, so the setting is quite different from the classical frameworks for propagation of chaos. In particular, the singular interaction leads us to work with Skorokhod's M1 topology as in \cite{DIRT15b, LS18a,NS17}. This is very different from the related work \cite{Inglis_Talay}, discussed in the next subsection, where the particles are smoothly interacting. For a careful introduction to the M1 topology on Skorokhod space, we refer to the excellent monograph \cite{whitt_2002}.

\subsection{A different approach to the heterogeneity}\label{sec:inglis-talay}

As we mentioned in the introduction, the first paper to look at heterogeneity in particle systems with a contagion mechanism similar to \eqref{particle_sys} is \cite{Inglis_Talay}, and this remains, to the best of our knowledge, the only paper to have examined the issue of mean-field convergence for such particle systems.

The analysis in \cite{Inglis_Talay}, however, differs quite substantially from ours, since the contagion mechanism is smoothed out in time, so there are no explosion times nor jumps to consider, neither in the approximating particle system nor in the mean-field limit; see the dynamics in \eqref{inglis-talay:particles} below. Also, there is no common noise to deal with, as the system is driven by fully independent Brownian motions. Unsurprisingly, however, there are certainly some similarities in the proofs of tightness and convergence, but we need a different topology in order to have tightness, and, as we turn to next, the heterogeneity in our system plays out very differently in the analysis.

Indeed, the `philosophy' of how the heterogeneity is dealt with in \cite{Inglis_Talay}, when passing to a mean-field limit, is entirely different from ours. While our aspiration is to see the heterogeneous structure reflected in the limiting problem, \cite{Inglis_Talay} is interested in justifying how a homogeneous limiting problem can also serve as a reasonable approximation to a large particle system with heterogeneous interactions. Slightly simplified and reformulated for the positive half-line, the contagious particle system in \cite{Inglis_Talay} takes the form
\begin{equation}\label{inglis-talay:particles}
dX_i(t)=b(X_i(t))dt+\sigma(X_i(t))dB_i(t) - d\!\textstyle\int_0^t \varrho(t-s) L^n_i(s)ds, \quad L_i^n(t):=\frac{1}{S_i^n}\sum_{j=1}^n  J_{ij} \mathbf{1}_{t\geq \tau_j},
\end{equation}
with $\tau_i=\inf\{ t>0 : X_i(t) \leq0  \}$ and $S_i^n:=\sum_{j=1}^n  J_{ij}$, for $i=1,\ldots,n$. Clearly, the asymmetry of the so-named synaptic weights $J_{ij}$ means that each particle in \eqref{inglis-talay:particles} can feel the contagion in a very different way. However, following on from the above, \cite{Inglis_Talay} is interested in connecting the limiting behaviour of this heterogeneous system to a single homogeneous McKean--Vlasov problem
\begin{equation}\label{inglis-talay:mckean}
dX(t)=b(X(t))dt+\sigma(X(t))dB_i(t) - d\!\textstyle\int_0^t \varrho(t-s) L(s)ds, \quad L(t):=\mathbb{P}({t\geq \tau}),
\end{equation}
where, as usual, $\tau=\inf\{ t>0 : X(t) \leq0  \}$. After some results on the well-posedness of \eqref{inglis-talay:mckean}, \cite{Inglis_Talay} attacks the aforementioned idea by tracking the  particle system \eqref{inglis-talay:particles} through the particular families of weighted empirical measures
\[
P_i^n:=\frac{1}{S_i^n}\sum_{j=1}^n  J_{ij} \delta_{X_i},\quad  n\geq1, \quad \text{for each} \quad i=1,\ldots,n.
\]
The main result then says that, provided we have
\begin{equation*}
\frac{1}{(S_i^n)^2} \sum_{j=1}^n (J_{ij})^2 \rightarrow 0,\quad \text{as} \quad n\rightarrow \infty,\quad \text{for all}\quad i=1,\ldots,n,
\end{equation*}
every sequence $(P_i^n)_{n\geq 1}$ converges weakly to the law of the unique solution to \eqref{inglis-talay:mckean}, independently of the index $i$ and independently of any other properties of the heterogeneity.

As a caveat to this result, however, we wish to flag up what seems to be a problem with the proof of \cite[Theorem 2.4]{Inglis_Talay} underlying the above convergence. Specifically, the first objective is to verify that, for any given $i$, the weighted empirical measures $P_i^n$ converge to some $P$ as $n\rightarrow \infty$, where $P$ solves a suitable nonlinear martingale problem associated with the desired McKean--Vlasov limit independently of $i$. On close inspection, a crucial step in \cite[Section 5.3]{Inglis_Talay}, towards affirming the aforementioned result, relies on having
\[
\phi(X_i(t))-\phi(X_i(s))-\int_s^t \mathfrak{L}_{(r,P_j^n)}\phi(X_i(r))dr = \int_s^t\sigma(X_i(r))\phi^\prime(X_i(r))dB_i(r),
\]
for $i\neq j$, for a suitable time- and measure-dependent differential operator $\mathfrak{L}$, yielding the martingale problem. However, since $\mathfrak{L}_{(\cdot,P_i^n)}$ can be seen to give the generator of the $i$'th particle, and since $P_i^n$ can differ significantly from $P_j^n$ depending on the structure of the synaptic weights $(J_{ij})_{i,j\leq n}$, the above equality would not appear to hold in general. It is not clear that this discrepancy can be fixed given just the assumptions mentioned above, but one should certainly be able to handle it by placing additional constraints on the $J_{ij}$'s, so that an error term can be singled out which becomes negligible in the limit. Alternatively, one could apply the methodology of the present paper in order to arrive at a heterogeneous mean-field limit, using a kernel structure  $J_{ij}=\kappa(u^j,v^i)$ to capture the asymmetry of the synaptic weights.

\subsection{Tightness for the finite particle systems}\label{subsect:tight}

We wish to work with convergence in the Skorokhod space on compact time intervals $[0,T]$. This involves pointwise convergence at the initial time $t=0$ with the limiting system having zero loss at this time, as required by our notion of solution. To ensure this, we establish uniform control over the smallness of the total feedback near the initial time, which is the subject of the next lemma. Concerning the final time, we only obtain convergence for continuity points of the limiting system. To establish tightness without an analogue of the next lemma at the end-point, we will extend the particle system continuously beyond its final time $T$ as in \cite{DIRT15a}. If one is not interested in ensuring initial regularity, one could also consider continuously embedding the particle system in the larger time interval $[-1,T]$, as in the recent paper \cite{Cuchiero20}, allowing for more general initial conditions.

\begin{lem}[Small-time control on the feedback]
	\label{lem:contagion_at_0} 
	Let Assumption \ref{assump:paramters_emperical} be satisfied, and define
	\begin{equation}\label{eq:F_i}
	F^n_{i}(t) := F\Bigl( \sum_{l=1}^{k} v_{l}^i \! \int_0^t g_{i}(s) d\mathcal{L}^n_{l,i}(s) \Bigr), \quad i=1,\ldots,n,\;n\geq1,
	\end{equation}
	for all $t\geq0$. For any $\epsilon>0$ and $\delta>0$ there is a small enough $t=t(\epsilon,\delta)>0$ such that
	\begin{equation}\label{eq:control_loss}
	\limsup_{n\rightarrow\infty} \max_{i\leq n}\mathbb{P}\bigl(F^n_{i}(t)\geq\delta \bigr) <\epsilon.
	\end{equation}
\end{lem}

\begin{proof} 
Let $\epsilon,\delta>0$ be given. Define the positive constants
\[
M_i:= g_i(0)\Vert F \Vert_{\text{Lip}} \max\{u^i\cdot v :  v\in S(\mathbf{v}) \},
\]
for $i\leq n$. Then $F_i^n$, as defined in \eqref{eq:F_i}, satisfies $F^n_i(s)-F^n_i(s-)\leq M_i|\mathcal{D}_s|$ at any time $s\geq 0$, where $|\mathcal{D}_s|$  is the number of particles absorbed at time $s$, given by the cascade condition \eqref{eq:default_set}-\eqref{particle_cascade_cond1}. By Assumption \ref{assump:paramters_emperical}, the set $S(\mathbf{v})$ is compact and the $u^i$'s likewise belong to the compact set $S(\mathbf{u})$, so we can take $M>0$ large enough such that $M_i\leq M$ uniformly in $i\leq n$ and $n\geq 1$. This yields an upper bound on the jumps of each particle, namely
\begin{equation}\label{eq:jump}
0 \leq -( X_i(s)-X_i(s-)) =  F^n_i(s)-F^n_i(s-)\leq \frac{M}{n} |\mathcal{D}_t|.
\end{equation}
for any $s > 0$. Next, we consider the number of particles starting within a distance of $\delta$ from the origin, as given by
\[
 N^n_{0,\delta}:=\sum_{j=1}^n \mathbf{1}_{\{X_j(0)\in(0,\delta]\}},
\]
and we then split the probability of interest \eqref{eq:control_loss} into
\begin{equation}\label{eq:F_bound_k}
\mathbb{P}\bigl( F_{i}(t)\geq \delta \bigl)\;   =  \mathbb{P}(F_{i}(t)\geq \delta, N^n_{0,\delta}\leq \lfloor n\delta/4M\rfloor ) + \mathbb{P}(N^n_{0,\delta} > \lfloor n\delta/4M\rfloor ).
\end{equation}
By the weak convergence \eqref{eq:emp_meas_network1}, as enforced by
Assumption \ref{assump:paramters_emperical}, we get
\begin{equation}\label{eq:initial_bound}
\limsup_{n\rightarrow\infty} \mathbb{P}(N^n_{0,\delta} > \lfloor n\delta/4M\rfloor )= 0
\end{equation}
provided
\[
\int_{\mathbb{R}^k\times\mathbb{R}^k}\int_0^\delta V_0(x|u,x)dxd\varpi(u,v) <\delta/4M.
\]
Using the bound on $V_0(\cdot|u,v)$ from Assumption \ref{X0_assump}, we see that the left-hand side of the above inequality is of order $O(\delta^{1+\beta})$ as $\delta \downarrow 0$, for some $\beta>0$. However, decreasing $\delta>0$ only increases the probability \eqref{eq:control_loss}. Therefore, without loss of generality, we can assume that 
 \eqref{eq:initial_bound} is satisfied for our fixed $\delta>0$.

Now let $\varsigma_i=\inf\{s>0 : F^n_i(s)>\delta/2 \}$. On each of the events $\{ F_{i}(t)\geq \delta ,\,N^n_{0,\delta}\leq \lfloor n\delta/4M\rfloor \}$, for any $i\leq n$, we claim that there must be at least $\lceil   n\delta /8M \rceil $ of the more than $n-\lfloor n\delta/4M\rfloor $ particles $X_j$ with initial positions $ X_j(0)> \delta$ which also satisfy
\begin{equation}\label{eq:absorb}
\inf_{s<\varsigma_i\land t}X_j(s) - X_j(0) \leq 5\delta/8.
\end{equation}
Indeed, if there were strictly less than $\lceil   n\delta /8M \rceil$ such particles (on any of the given events), then there could be at most a total of
\[
(\lceil   n\delta /8M \rceil-1)+k \leq  \lfloor   n\delta /8M \rfloor+ \lfloor   n\delta /4M \rfloor \leq 3n\delta /8M
\]
particles $X_j$ with $\tau_j<\varsigma_i\land t$ or $X^j(\varsigma_i\land t-)=0$ (on that event), which by \eqref{eq:jump} can only cause a downward jump of all other particles by at most $3\delta/8$. But this would be insufficient  for any particles $X_j$ with  $X_j(0)> \delta$ that does not satisfy \eqref{eq:absorb}  to be absorbed before or at time $\varsigma_i\land t$. Thus, again by \eqref{eq:jump},  we would end up with $F_i^n(\varsigma_i\land t)\leq 
 3\delta /8\leq \delta/2$, and hence also $F_i^n( t) \leq 3\delta /8$, which contradicts $F_i^n( t) \geq \delta$. It follows that, for each $i\leq n$, we must in particular have
\begin{equation}\label{eq:1st_bound}
\mathbb{P}\bigl(F^n_{i}(t)\geq \delta, N^{n}_{0,\delta}\leq \lfloor n\delta/4M\rfloor \bigr) \leq \mathbb{P}\bigl(  \hat{N}_{t,\delta}^{i,n} \geq \lceil   n\delta /8M \rceil   \bigr),
\end{equation}
where
\[
\hat{N}_{t,\delta}^{i,n}:=\sum_{j=1}^n \mathbf{1}_{\{ \inf_{s<\varsigma_i\land t}X_j(s) - X_j(0) \leq 5\delta/8  \}}.
\]

By definition of each $\varsigma_i$ and Assumption \ref{assump:paramters_emperical} there is a uniform $c>0$ such that \eqref{eq:absorb} implies
\[
\inf_{s\leq t}Y^i(s) -ct  -\sup_{s\leq t}|Z^i(s)|  - 3\delta/8 \leq \inf_{s<\varsigma_i\land t} X_i(s) - X_i(0) \leq -\delta/2,
\]
and hence
\[
 \inf_{s\leq t}Y^i(s) \leq  ct  + \sup_{s\leq t}|Z^i(s)| -\delta/8,
\]
for all $i\leq n$, where
$Y_i(s) :=\int_0^s\sigma_{i}(r) \sqrt{1-\rho^2}dB_i(r)$ and $Z_i(s) :=\int_0^s\sigma_{i}(t) \rho dB_0(r)$. Taking $t<\delta 2^{-5}$, and defining
\[
\tilde{N}^n_{t,\delta}:=\sum_{j=1}^n\mathbf\mathbf{1}_{\{\inf_{s\leq t}Y_j(s) \leq -\delta 2^{-4} \}},
\]
we thus have
\begin{align*}
 \mathbb{P}\bigl(  \hat{N}_{t,\delta}^{i,n} \geq \lceil   n\delta /8M \rceil   \bigr) &\leq \mathbb{P}\bigl( \tilde{N}^n_{t,\delta} \geq \lceil   n\delta /8M \rceil \bigr)+ \mathbb{P}\bigl(\sup_{s\leq t}|Z_i(s)|\geq \delta 2^{-5}\bigr) 
\end{align*}
Using the weak convergence \eqref{eq:emp_meas_network} from Assumption \ref{assump:paramters_emperical} and the strong law of large numbers for the independent Brownian motions $B^i$, we can deduce that
\[
\limsup_{n\rightarrow\infty}\mathbb{P}\bigl( n^{-1}\tilde{N}^n_{t,\delta} \geq   \delta /8M  \bigr) = 0
\]
provided
\[
\int_{\mathbb{R}^k\times\mathbb{R}^k}\mathbb{P}\Bigl(  \inf_{s\leq t} \int_0^s \sigma_{u,v}(r)dB(r) \leq \delta 2^{-4} \Bigr)d\varpi(u,v) < \delta/8M,
\]
where $B$ is a standard Brownian motion. As $\sigma$ is bounded uniformly in $(u,v,r)$ this can certainly be achieved by taking $t$ small enough relative to $\delta$. Moreover, taking $t$ small enough relative to $\delta$ also ensures that we can make
\[
\limsup_{n\rightarrow\infty}\max_{i\leq n} \mathbb{P}\bigl(\sup_{s\leq t}|Z_i(s)|\geq \delta 2^{-5}\bigr)
\]
as small as we like. Consequently, recalling \eqref{eq:F_bound_k}--\eqref{eq:initial_bound}, we can indeed find a small enough $t=t(\delta,\epsilon)>0$ such that
	\[
\limsup_{n\rightarrow\infty} \max_{i\leq n}\mathbb{P}\bigl(F^n_{i}(t)\geq\delta \bigr) <\epsilon,
\]
which completes the proof.
\end{proof}

As already mentioned in \eqref{empirical_measures_parameters}, given a family of solutions to the particle system \eqref{particle_sys}, for all $n\geq 1$, on an arbitrary time interval $[0,T]$, we define the empirical measures
\begin{equation}\label{eq:emprical_measures}
\mathbf{P}^n:= \frac{1}{n} \sum_{i=1}^n \delta_{u^i} \otimes \delta_{v^i} \otimes \delta_{X_i}, \qquad \text{for} \quad n \geq 1,
\end{equation}
where each $X_i$ is a random variable with values in $D_\mathbb{R}=D_\mathbb{R}[0,T]$ for the given $T>0$. Note that each $\mathbf{P}^n$ is then a random probability measure with values in $\mathcal{P}(\mathbb{R}^k\times\mathbb{R}^k\times D_\mathbb{R})$.

Our next result shows that these empirical measures are tight in a suitable sense. The previous lemma serves as a crucial ingredient in the proof.

\begin{prop}[Tightness of the empirical measures]\label{prop:tightness}
Let $T>0$ be given, and consider the solutions $\{(X_i(t))_{t\in[0,T]}\}_{i=1}^n$ of \eqref{particle_sys}, for all $n\geq 1$, where we assume Assumption \ref{assump:paramters_emperical} is satisfied. For an arbitrary $S>T$, we extend the paths of $X_i$ from $[0,T]$ to $[0,S]$, by setting $X_i(t):=X_i(T)$ for all $t\in[T,S]$. Let $\mathbf{P}^n$ be the empirical measures defined as in \eqref{eq:emprical_measures} but with each $X_i$ taking values in $D_\mathbb{R}=D_\mathbb{R}[0,S]$, where we endow $D_\mathbb{R}$ with Skorokhod's M1 topology (see e.g.~\cite{AT89,whitt_2002}). Moreover, let $\mathcal{P}(\mathbb{R}^k\times \mathbb{R}^k\times D_\mathbb{R})$ be endowed with the topology of weak convergence of measures, as induced by the M1 topology on $D_\mathbb{R}$. Then the empirical measures $\mathbf{P}^n$, for $n\geq1$, form a tight sequence of random variables with values in $\mathcal{P}(\mathbb{R}^k\times\mathbb{R}^k\times D_\mathbb{R})$.
\end{prop}
\begin{proof}
By Assumption \ref{assump:paramters_emperical}, there is uniform constant $c>0$ such that $|\mathbf{u}_i|+|\mathbf{v}_i|\leq c $ for all $i=1,\ldots,n$ and $n\geq 1$. If we consider any set $K \in \mathcal{B}(\mathbb{R}^k\times\mathbb{R}^k)\otimes\mathcal{B}( D_\mathbb{R})$ of the form $K = \bar{B}_c(0) \times A$, where $\bar{B}_c(0)$ is the closed ball of radius $c>0$ around the origin in $\mathbb{R}^{k}\times\mathbb{R}^{k} $, we thus have
\begin{equation}\label{eq:tight0}
\mathbb{E}[\mathbf{P}^n(K^\complement)]\leq  \mathbb{E}[\mathbf{P}^n(\mathbb{R}^k\times\mathbb{R}^k\times A^\complement)] \leq  \max_{i=1,\ldots, n} \mathbb{P}(X_i(\cdot) \in A^\complement).
\end{equation}
As in \cite{AT89}, we define
\[
\\M_{X_{i}}(t_{1},t,t_{2}) :=\begin{cases}
	|X_{i}(t_{1})-X_{i}(t)|\land|X_{i}(t_{2})-X_{i}(t)|, & X_{i}(t)\notin[X_{i}(t_{1}),X_{i}(t_{2})]\\
	0, & X_{i}(t)\in[X_{i}(t_{1}),X_{i}(t_{2})]
\end{cases}
\]
and consider the oscillation function for the M1 topology given by
\[
w_{\delta,\text{M1}}(X_i):=\sup\bigl\{ M_{X_{i}}(t_{1},t,t_{2}) : 0\leq t_1 < t < t_2 \leq T, \; t_2-t_1 \leq \delta \bigr\}.
\]
Recalling the definition of $F_i^n$ in \eqref{eq:F_i}, it follows from Assumption \ref{assump:paramters_emperical} that the paths $t\mapsto F^n_i(t)$ are increasing. Exploiting this fact, we can check that
\[
M_{X_{i}}(t_{1},t,t_{2}) \leq |(X_{i}+F^n_i)(t_1)-(X_{i}+F^n_i)(t)| + |(X_{i}+F^n_i)(t_2)-(X_{i}+F^n_i)(t)|,
\]
and hence, using the bounds on the coefficients given by Assumption \ref{assump:paramters_emperical}, the nice continuous dynamics of
\[
d(X_{i}+F^n_i)(t) = b_i(t)dt+\rho\sigma_i(t)dB_0(t) + \sqrt{1-\rho^2}\sigma_i(t)dB_i(t) 
\]
allow for a simple application of Markov's inequality and  Burkholder-Davis-Gundy to deduce
\[
\mathbb{P}\bigl( M_{X_{i}}(t_{1},t,t_{2}) \geq \varepsilon \bigr) \leq C_0\varepsilon^{-2}|t_2-t_1|^2,\quad 0\leq t_1 \leq t \leq t_2 \leq T,\quad\varepsilon>0,
\]
for a fixed constant $C_0>0$ that is uniform in $i\leq n$ and $n\geq1$. Armed with this estimate, it follows from \cite[Theorem 1]{AT89} that we get the bound
\begin{equation}\label{eq:tight1}
\max_{i=1,\ldots,n}\mathbb{P}(w_{\delta,\text{M1}}(X_i) \geq \varepsilon) \leq \tilde{C_0}\varepsilon^{-2}\delta
\end{equation}
for all $\delta,\varepsilon>0$ and $n\geq1$, for another fixed constant $\tilde{C_0}>0$. Moreover, we can apply Lemma \ref{lem:contagion_at_0} to deduce that
\begin{equation}\label{eq:tight2}
\lim_{\delta \downarrow 0}\;\limsup_{n\rightarrow \infty} \max_{i=1,\ldots,n} \mathbb{P}\bigl( \sup_{t_1,t_2\in[0,\delta]} |X_i(t_1)-X_i(t_2)| \geq \varepsilon   \bigr) =0,
\end{equation}
which is the key part of why we can get tightness for the M1 topology. At the other endpoint, we automatically have
\begin{equation}\label{eq:tight3}
\lim_{\delta \downarrow 0}\sup_{n\geq 1} \max_{i=1,\ldots,n} \mathbb{P}\bigl( \sup_{t_1,t_2\in[S-\delta,S]} |X_i(t_1)-X_i(t_2)| \geq \varepsilon   \bigr) =0,
\end{equation}
since the probability vanishes for all $\delta<T-S$, by our continuous extension $X_i(t)=X_i(T)$ for $t\in[T,S]$. Next, the dynamics of each $X_i$ and Assumption \ref{assump:paramters_emperical} are easily seen to imply the compact containment condition
\begin{equation}\label{eq:tight4}
\limsup_{R\rightarrow\infty} \;\sup_{n\geq 1}\max_{i=1,\ldots,n}\mathbb{P}\bigl(\sup_{t\in[0,T]}|X_i(t)| \geq R\bigr) = 0.
\end{equation}
Letting
\[
A(\delta,m):=\Bigl\{\eta \in D_\mathbb{R} : w_{\delta,\text{M1}}(\eta) \lor \sup_{s,t\in[T-\delta,T]} |\eta(s)-\eta(t)| \lor \sup_{s,t\in[0,\delta]} |\eta(s)-\eta(t)|  < m^{-1}  \Bigr\},
\]
it follows from \eqref{eq:tight1}, \eqref{eq:tight2}, and \eqref{eq:tight3}, that we can find $\delta_{\epsilon,m}>0$ such that
\[
\limsup_{n\rightarrow\infty}\max_{i=1,\ldots,n}\mathbb{P}\bigl( X_i(\cdot) \in A(\delta_{\varepsilon,m},m)^\complement  \bigr) \leq \frac{\varepsilon}{2^m}.
\]
Using also \eqref{eq:tight4}, and setting
\[
K_{\varepsilon,l}:=\bar{B}_c(0)\times \bar{A}_{\varepsilon,l} ,\quad A_{\varepsilon,l}:=\bigcap_{m=1}^\infty A(\delta_{\varepsilon,m+2l},m)\cap \{\eta: \sup_{s\in[0,T]}|\eta(s)| < R_\epsilon \},
\]
for a large enough $R_\varepsilon>0$, it follows from the above and \eqref{eq:tight0} that
\begin{equation}\label{eq:tight5}
\limsup_{n\rightarrow\infty}\mathbb{E}[\mathbf{P}^n (K^\complement_{\varepsilon,l})] \leq  \limsup_{n\rightarrow\infty}\max_{i=1,\ldots,n}\mathbb{P}\bigl( X_i(\cdot) \in  A_{\varepsilon,l}^\complement  \bigr) \leq \sum_{m=1}^\infty \frac{\varepsilon}{2^{m+2l}} =
\frac{\varepsilon}{4^{l}} ,
\end{equation}
for any $\varepsilon>0$. By construction, each $A_{\varepsilon,l}$ is relatively compact in $D_\mathbb{R}$ for the M1 topology, as e.g.~follows from the characterisation in \cite[Thereom 12.12.2]{whitt_2002}, so each $K_{\varepsilon,l}$ is compact for the M1 topology.
Moreover, \eqref{eq:tight5} yields
\begin{align*}
\limsup_{n\rightarrow\infty} \mathbb{P}\Bigl( \bigcup_{l=1}^\infty \bigl\{ \mathbf{P}^n( K^\complement_{\varepsilon,l}) > 2^{-l} \bigr\} \Bigr) 
&\leq \limsup_{n\rightarrow\infty}
\sum_{l=1}^\infty \mathbb{P}\Bigl(  \mathbf{P}^n( K^\complement_{\varepsilon,l}) > 2^{-l} \Bigr) \\
 & \leq  \sum_{l=1}^\infty 2^{l} \limsup_{n\rightarrow\infty}\mathbb{E}[\mathbf{P}^n (K^\complement_{\varepsilon,l})] \leq  \sum_{l=1}^\infty \frac{\varepsilon}{2^{l}} = \varepsilon.
 \end{align*}
It remains to note that the set $\cap_{l=1}^\infty \{\mu:\mu(K^\complement_{\varepsilon,l})\leq 2^{-l}\}$ is closed in $\mathcal{P}(D_\mathbb{R})$ by the Portmanteu theorem (under the topology of weak convergence of measures induced by the M1 topology on $D_\mathbb{R}$), as each $K^\complement_{\varepsilon,l}$ is open, and that it forms a tight family of probability measures by construction, as each $K_{\varepsilon,l}$ is compact. Therefore, Prokhorov's theorem gives that the set is compact, since $\mathcal{P}(D_\mathbb{R})$ is a Polish space for the topology we are working with. In turn, we can conclude that $(\mathbf{P}^n)_{n\geq1}$ is indeed a tight sequence of random probability measures when $\mathcal{P}(D_\mathbb{R})$ is given the topology of weak convergence induced from the M1 topology on $D_\mathbb{R}$.
\end{proof}

\subsection{Identifying a suitable probabilistic setup for the mean-field limit}\label{subsect:characterise_limit}

Recall that the empirical measures $\mathbf{P}^n$ are random variables valued in the space of probability measures $\mathcal{P}(\mathbb{R}^k\times\mathbb{R}^k\times D_\mathbb{R})$. For $(u,v,\eta)\in\mathbb{R}^k\times\mathbb{R}^k\times D_\mathbb{R} $, we define the coordinate projections
\[
\pi_{1,l}(u,v,\eta):=u_l,\quad \pi_{2,l}(u,v,\eta):=v_l,\quad \text{and} \quad \pi_{3}(u,v,\eta)(t):=\eta(t)
\]
as well as 
\[
\pi_{t}(u,v,\eta):=(u,v,\pi_3(u,v,\eta)(t))=(u,v,\eta(t))\quad \text{and} \quad \pi_{(1,2)}(u,v,\eta):=(u,v).
\]
Writing $\mathbf{P}^n_t:=\mathbf{P}^n \circ \pi_{t}^{-1}$, for $t\geq0$, and $\varpi^n:=\mathbf{P}^n \circ \pi_{(1,2)}^{-1}$, the conditions \eqref{eq:emp_meas_network} and \eqref{eq:emp_meas_network1} from Assumption \ref{assump:paramters_emperical} read as
\begin{equation}\label{eq:weak_conv_net_initial}
d\varpi^n(u,v) \rightarrow d\varpi(u,v)\quad \text{and}\quad
d\mathbf{P}^n_0(u,v,x) \rightarrow d\nu_0(x|u,v)d\varpi(u,v),
\end{equation}
where the mode of convergence is weak convergence of measures.  Given this, we can make the following simple observation, guaranteeing that the limiting distribution $\varpi$ behaves as we would like it to behave.

\begin{lem}[The limiting type-distribution]\label{eq:uv-compact-nonneg}
	Let Assumption \ref{assump:paramters_emperical} be in place. Writing $(\mathbf{u},\mathbf{v})\sim \varpi$, we let $S(\mathbf{u})$ and $S(\mathbf{u})$ denote the support of $\mathbf{u}$ and $\mathbf{v}$ respectively. Then $S(\mathbf{u})$ and $S(\mathbf{v})$ are both compact in $\mathbb{R}^k$, and we have that, for
	\[
	u\cdot v \geq 0, \qquad \text{for all}\;\; u\in S(\mathbf{u}) \;\; \text{and} \;\; v\in S(\mathbf{v}).
	\]
\end{lem}
\begin{proof}First of all, the compactness follows by noting that \eqref{eq:weak_conv_net_initial} gives
\[
\mathbb{P}( \sqrt{|\mathbf{u}|^2 + |\mathbf{v}|^2} \leq C ) \geq \limsup_{n\rightarrow\infty} \varpi^n( \{ (u,v) : \sqrt{|u|^2 + |v|^2} \leq C \}) =1,
\]
for a large enough $C>0$,
due to the Portmanteu theorem and Assumption \ref{assump:paramters_emperical}. Moreover, again by Assumption \ref{assump:paramters_emperical},  it holds for each $n\geq1$ that
	\[
	\int_{\mathbb{R}^k}\varpi^n(\{u:u\cdot v \geq 0\}\times \mathbb{R}^k) d\varpi^n(v) = \frac{1}{n^2}\sum_{i,j=1}^n \mathbf{1}_{\{u^i\cdot v^j \geq 0\}} = 1.
	\]
Since the marginals are weakly convergent, by \eqref{eq:weak_conv_net_initial}, we have weak convergence of the product measures $\varpi^n_1 \otimes \varpi^n_2$ to $\varpi_1 \otimes \varpi_2$, and hence the Portmanteu theorem gives
	\begin{equation*}
\int_{\mathbb{R}^k}\int_{\mathbb{R}^k} \mathbf{1}_{\{u\cdot v\geq 0\}} d \varpi_1(u) d\varpi_2(v) \\ \geq \limsup_{n\rightarrow \infty} \int_{\mathbb{R}^k}\int_{\mathbb{R}^k} \mathbf{1}_{\{u\cdot v\geq 0\}} d\varpi^n_1(u) d\varpi^n_2(v) = 1,
\end{equation*}
by the previous equality. In turn, for $\varpi_1$-a.e.~$v\in\mathbb{R}^k$, it holds that $u \cdot v \geq 0$ for $\varpi_2$-a.e.~$u\in\mathbb{R}^k$, which finishes the proof.
\end{proof}

Throughout the rest of the paper, we fix a terminal time $T>0$ and an arbitrary $S>T$, as in Proposition \ref{prop:tightness}. Naturally, our arguments will apply for any $T>0$.

By Prokhorov's theorem, Proposition \ref{prop:tightness} gives that any subsequence of $(\mathbf{P}^n,B_0)_{n\geq1}$ has a further subsequence converging in law to a limit point $(\mathbf{P}^\star,B_0)$, whose law we denote by $\mathbb{P}^\star_0$. The limiting law $\mathbb{P}^\star_0$ is realised as a Borel probability measure on the product $\sigma$-algebra  $\mathcal{B}(\mathcal{P}(\mathbb{R}_k\times\mathbb{R}_k) \times D_\mathbb{R}) \otimes \mathcal{B}(C_\mathbb{R}) $ such that the second marginal of $\mathbb{P}^\star_0$ is the law of a standard Brownian motion $B_0:\Omega \rightarrow C_\mathbb{R}$ with $C_\mathbb{R}=C_\mathbb{R}[0,S]$, while the first marginal is the law of a random probability measure $\mathbf{P}^\star:\Omega \rightarrow \mathcal{P}(\mathbb{R}_k\times\mathbb{R}_k \times D_\mathbb{R})$ with $D_\mathbb{R}=D_\mathbb{R}[0,S]$. Moreover, from Assumption \ref{assump:paramters_emperical} we have that $\mathbb{P}^\star(\omega) \circ \pi_{(1,2)}^{-1} =\varpi$ for all $\omega\in\Omega$, so we can write $\mathbf{P}^\star:\Omega \rightarrow \mathcal{P}_\varpi(D_\mathbb{R})$, where $\mathcal{P}_\varpi(D_\mathbb{R})$ denotes the subspace of all $\mu \in \mathcal{P}(\mathbb{R}_k\times\mathbb{R}_k \times D_\mathbb{R})$ with fixed marginal $\mu\circ \pi_{(1,2)}^{-1}=\varpi$, as in the discussion after the statement of Theorem \ref{thm:mf_conv}. Later, we will use this to identify $\mathbf{P}^\star$ with a random Markov kernel, as per Definition \ref{defn:random_markov_kernel}.

Throughout what follows, we fix a given limit point $(\mathbf{P}^\star,B_0)$. For concreteness, we take this limit point $(\mathbf{P}^\star,B_0)$ to be defined on the canonical background space $(\Omega_0,\mathcal{B}(\Omega_0),\mathbb{P}^\star_0)$, where we set $\Omega_0:=\mathcal{P}_{\varpi}(D_\mathbb{R})\times C_\mathbb{R}$ and, as discussed above, we let $\mathbb{P}^\star_0\in \mathcal{P}(\Omega_0)$ denote the limiting law of the joint laws of $(\mathbf{P}^n,B_0)_{n\geq 1}$ along the given subsequence. Then we can simply write $(\mathbf{P}^\star,B_0)$ as the identity map $(\mathbf{P}^\star,B_0)(\mu,w)=(\mu,w)$ on $\Omega_0$ and we have $\mathcal{B}(\Omega_0)=\sigma(\mathbf{P}^\star,B_0)$.

Given $\mathbb{P}_0^\star$, we can define a probability measure $\mathbb{P}^\star$ on
\[
 \mathcal{B}(\mathbb{R}^k \times \mathbb{R}^k) \otimes\mathcal{B} ( \mathcal{P}_{\varpi}(D_\mathbb{R})\times  C_\mathbb{R} ) \otimes \mathcal{B}( D_\mathbb{R})
\]
by letting
\begin{equation}\label{eq:defn_big_P_star}
\mathbb{P}^\star(O \times E \times A ) := \int_{ E } \mu(O \times A) d \mathbb{P}_0^\star(\mu , w) 
\end{equation}
for $O\in\mathcal{B}(\mathbb{R}^k \times \mathbb{R}^k)$,  $E\in \mathcal{B} ( \mathcal{P}_\varpi(D_\mathbb{R}) \times  C_\mathbb{R} )$, and $A\in \mathcal{B}( D_\mathbb{R})$. Now consider the disintegration of $\mathbb{P}^\star$ with respect to the projection $\hat{\pi}_{1,2}((u,v),(\mu,w),\eta):=((u,v),(\mu ,w))$. This yields a Markov kernel
\begin{equation}\label{eq:1st_Markov_kernel}
((u,v),(\mu,w)) \mapsto \mathbb{P}_{u,v}^{\mu,w}\in \mathcal{P}(D_\mathbb{R})
\end{equation}
such that
\begin{equation}\label{eq:disintegration}
\mathbb{P}^\star(O \times E \times A )=\int_O\int_{E } \mathbb{P}_{u,v}^{\mu,w}(A) d\mathbb{P}_0^\star(\mu, w) d\varpi(u,v),
\end{equation}
for any $O\times E \times A \in \mathcal{B}(\mathbb{R}^k \times \mathbb{R}^k) \otimes\mathcal{B} ( \mathcal{P}_{\varpi}\times  C_\mathbb{R} ) \otimes \mathcal{B}(D_\mathbb{R}) $, by Tonelli's theorem, since
\[
\mathbb{P}^\star \circ \hat{\pi}_{1,2}^{-1} = \mathbb{P}_0^\star \otimes \varpi.
\]
Indeed, we have $\mu(O\times D_\mathbb{R})=\varpi(O)$ for $\mathbb{P}_0^\star$-a.e.~$(\mu,w)\in E$, and hence $\mathbb{P}^\star \circ \hat{\pi}_{1,2}^{-1}(O\times E)=\mathbb{P}^\star(O\times E\times A)=\varpi(O)\mathbb{P}_0^\star(E)$, for any given set $O\times E$ as above.

From here, we define a family of probability measures $ \mathbb{P}^\star_{u,v}$ on $\mathcal{B} ( \mathcal{P}_{\varpi}(D_\mathbb{R})\times  C_\mathbb{R} ) \otimes \mathcal{B}( D_\mathbb{R})$ by

\begin{equation}\label{eq:P_uv_defn}
 \mathbb{P}^\star_{u,v} (E \times A) := \int_E \mathbb{P}_{u,v}^{\mu, w}(A) d\mathbb{P}^\star_0(\mu, w),\quad(u,v)\in\mathbb{R}^k\times\mathbb{R}^k,
\end{equation}
where the joint measurability of each $\mathbb{P}_{u,v}^{\mu, w}(A)$ ensures that 
\[
(u,v)\mapsto \mathbb{P}^\star_{u,v} \in\mathcal{P}(\Omega_\star)\quad \Omega_\star:= \Omega_0 \times D_\mathbb{R} =\mathcal{P}_{\varpi}(D_\mathbb{R})\times C_\mathbb{R} \times D_\mathbb{R}
\]
is a Markov kernel. Notice also that that marginal of $\mathbb{P}_{u,v}^\star$ on $(\Omega_0,\mathcal{B}(\Omega_0))$ is always $\mathbb{P}_0^\star$, since, for any $E\in \mathcal{B}(\Omega_0)$,
\[
\mathbb{P}_{u,v}^\star(E\times D_\mathbb{R}) =  \int_E \mathbb{P}_{u,v}^{\mu, w}(D_\mathbb{R}) d\mathbb{P}^\star_0(\mu, w) = \mathbb{P}^\star_0(E).
\]
The family of probability measures $ \mathbb{P}^\star_{u,v}$ will play a critical role in what follows. 

On the background space $(\Omega_\star,\mathcal{B}(\Omega_\star))$, we define the three random variables
\begin{equation}\label{eq:P-B_0-Z}
\mathbf{P}^\star(\mu,w,\eta):=\mu, \quad B_0(\mu,w,\eta):=w, \quad \text{and}\quad Z(\mu,w,\eta)=\eta.
\end{equation}
By construction, for any $(u,v)\in \mathbb{R}^k\times\mathbb{R}^k$, we then have
\[
\mathbb{P}_{u,v}^\star\bigl(Z\in A , (\mathbf{P}^\star, B_0)\in E_1\times E_2 \bigr) = \int_{E_1\times E_2} \mathbb{P}_{u,v}^{\mu, w}(A) d\mathbb{P}^\star_0(\mu, w)
\]
and
\[
\mathbb{P}_{u,v}^\star\bigl(  \mathbf{P}^\star\in E_1, B_0\in E_2 \bigr) = \mathbb{P}^\star_0(E_1\times E_2),
\]
for any $A\in \mathcal{B}(D_\mathbb{R})$, $E_1\in \mathcal{B}(\mathcal{P}_{\varpi})$, and $E_2\in \mathcal{B}(C_\mathbb{R})$. Consequently, we have
\[
\mathbb{P}^\star_{u,v}(Z\in A \, | \, \mathbf{P}^\star, B_0)= \mathbb{P}_{u,v}^{\mathbf{P}^\star\!\!, B_0}(A)
\]
for all $A\in \mathcal{B}(D_\mathbb{R})$, and we stress that the joint law of $(\mathbf{P}^\star, B_0)$ is the same under every $\mathbb{P}_{u,v}^\star$.

Now consider the hitting-time map $\tau_0: D_\mathbb{R} \rightarrow \mathbb{R}$ given by
\begin{equation}\label{eq:hitting-time_map}
\tau_0(\eta):=\inf \{ t\geq 0 : \eta_s  \leq 0   \}.
\end{equation}
Then $\tau_0(Z):\Omega^\star\rightarrow \mathbb{R}$ is the first hitting time of zero for the process $Z$ defined above. The above constructions lead us to define the following candidates for the limiting feedback in our mean-field problem, namely $\mathcal{L}^\star_l : \Omega_0 \rightarrow D_\mathbb{R}$, for $l=1,\ldots,k$, given by
\begin{align}\label{eq:limit_feedback_processes}
\mathcal{L}_l^\star(t) &:= \int_{\mathbb{R}^k\times\mathbb{R}^k} u_l \mathbb{P}^\star_{u,v}\bigl( t\geq \tau_0(Z)  \, | \, B_0, \mathbf{P}^\star \bigr)d\varpi(u,v) \nonumber \\
&=\int_{\mathbb{R}^k\times\mathbb{R}^k} u_l \mathbb{P}_{u,v}^{\mathbf{P}^\star\!\!, B_0}\bigl( t\geq \tau_0(Z) \bigr)d\varpi(u,v),\quad \text{for} \quad t\geq0,
\end{align}
on the background space $(\Omega_0,\mathcal{B}(\Omega_0))$. Naturally, we may also view these as stochastic processes defined on $(\Omega_\star,\mathcal{B}(\Omega_\star))$. In the next section, we confirm that these candidates are indeed inducing the limiting laws of the empirical feedback $\mathcal{L}^n_l$, for $l=1,\ldots,k$, when considered on the probability space $(\Omega_\star,\mathcal{B}(\Omega_\star), \mathbb{P}_{u,v}^\star)$, for $\varpi$-almost every pair of type vectors $(u,v)\in\mathbb{R}^k\times\mathbb{R}^k$. To this end, it will be useful to consider the particular set of continuity times
\begin{equation}\label{eq:cont_times}
\mathbb{T}_\star:= \Bigl\{  t\in[0,T] : \int_{\mathbb{R}^k\times \mathbb{R}^k}\!\!\mathbb{P}_{u,v}^\star \bigl( \{Z(t)=Z(t-)\} \textstyle{\bigcap_{l=1}^k}\{ \mathcal{L}^\star_l(t)=\mathcal{L}^\star_l(t-)\} \bigr)d\varpi(u,v) =1  \Bigr\}.
\end{equation} 
Observe that the complement of $\mathbb{T}_\star$ in $[0,T]$ is at most countably infinite. Indeed, the assignment $A\mapsto \int_{\mathbb{R}^k\times \mathbb{R}^k}\mathbb{P}^\star_{u,v}(A)d\varpi(u,v)$ yields a well-defined probability measure on $\mathcal{B}(\Omega_\star)$, and, since $Z$ is c\`adl\`ag by definition, dominated converge also shows that each $\mathcal{L}^\star$ is c\`adl\`ag. Hence the claim follows from \cite[Sect.~13]{Billingsley}. We shall make abundant use of this fact in our convergence arguments.

Throughout what follows, we always take the Skorokhod space $D_\mathbb{R}$ to be endowed with Skorokhod's M1 topology. Moreover, we let $\mathfrak{T}_{wk}^{\text{M1}}$ denote the topology corresponding to weak convergence of measures in $\mathcal{P}_\varpi(D_\mathbb{R}) \cong \mathcal{P}(\mathbb{R}\times \mathbb{R} \times D_\mathbb{R})$ induced by the M1 topology $D_\mathbb{R}$.

\subsection{Convergence of the feedback along with the empirical measures}\label{subsect:feedback_conv}

In this subsection, we study the feedback from defaults felt by each institution, as the number of institutions tends to infinity along a convergent subsequence of the empirical measures. These results are essential to the motivating applications and, given tightness of the system, they form the critical technical hurdles towards obtaining the mean-field limit.

\begin{prop}[Hitting-time continuity]\label{prop:hitting_time_cont} As in Section \ref{subsect:feedback_conv}, let $\mathbf{P}^\star$ be a given limit point, in law, of the empirical measures $\mathbf{P}^n$, and consider the resulting Markov kernel $\{\mathbb{P}^\star_{u,v}\}_{(u,v)\in\mathbb{R}^k\times\mathbb{R}^k}$ on $D_\mathbb{R}$ defined in \eqref{eq:P_uv_defn}. For $\varpi$-almost every $(u,v)\in \mathbb{R}^k\times\mathbb{R}^k$, the hitting-time map $\tau_0 : D_\mathbb{R}\rightarrow  \mathbb{R}$ from \eqref{eq:hitting-time_map} is continuous in the M1 topology on $D_\mathbb{R}$ at $\mathbb{P}^\star_{u,v}$-almost every $\eta \in D_\mathbb{R}$.
\end{prop}
\begin{proof}Let $Y_i(t):=\int_0^t\sigma(s)dW_i(s)$ where $W_i(t)=\rho B_0(t) + \sqrt{1-\rho^2}B_i(t)$. Then consider the family of probability measures $\mathbb{Q}^n$ on the Borel $\sigma$-algebra of $(\mathbb{R}^k\times \mathbb{R}^k) \times D_\mathbb{R}\times C_\mathbb{R}$, defined by
	\[
	\mathbb{Q}^n(O\times A \times B):= \frac{1}{n}\sum_{i=1}^n \mathbb{P}(X_i\in A, Y_i \in B)\mathbf{1}_{\{(u^i,v^i)\in O\}},\quad n\geq1,
	\]
	where we are averaging over the joint laws of each particle and its martingale part for the particles within a given set of types $O$. As in the proof of Proposition \ref{prop:tightness}, we can show $(\mathbb{Q}^n)_{n\geq1}$ is tight. Following the procedure in Section \ref{subsect:characterise_limit}, and exploiting the continuity of the marginal projections, we can then deduce that there is a limit point $\mathbb{Q}^\star \cong \{\mathbb{Q}^\star _{u,v}\}_{(u,v)\in\mathbb{R}^k \times \mathbb{R}^k} $ in $\mathcal{P}_\varpi( D_\mathbb{R}\times C_\mathbb{R})$ such that
	\begin{equation}\label{eq:from_joint_to_marginal}
	\int_{\mathbb{R}^k\times\mathbb{R}^k}\mathbb{P}^\star_{u,v}(A)d\varpi(u,v)= 	\int_{\mathbb{R}^k\times\mathbb{R}^k}\mathbb{Q}^\star_{u,v}(A \times C_\mathbb{R})d\varpi(u,v)
	\end{equation}
	for all $A\in \mathcal{B}(D_\mathbb{R})$, which we will utilise at the end of the proof. Moreover, we see that
	\[
 \mathbb{Q}^n(\mathbb{R}^k\times\mathbb{R}^k\times D_\mathbb{R} \times B)=\mathbb{P}(Y_1\in B), \quad \text{for all}\quad B \in \mathcal{B}(C_\mathbb{R}),
\]
gives the law of Brownian motion time-changed by $t\mapsto\sigma(t)$, for all $n \geq 1$. Noting also that future increments of the $Y_i$'s are independent of the filtration generated by all the particles up to any given time, we can therefore let $(Z^\star,Y^\star)(\eta,w):=(\eta ,w)$, for all $(\eta,w)\in D_\mathbb{R}\times C_\mathbb{R}$, and conclude from the weak convergence that $Y^\star$ has the law of a time-changed Brownian motion under $\mathbb{Q}_{u,v}^\star$ with respect to the filtration generated by the pair $(Z^\star,Y^\star)$. Let $\mathbb{T}:=\{t\geq 0 : \int_{\mathbb{R}^k\times\mathbb{R}^k} \mathbb{Q}^\star_{u,v}(Z^\star(t) = Z(t-))d\varpi(u,v) = 1 \}$ and consider the events
	\[
	E_{u,v} := \bigcap_{q\leq r \in \mathbb{Q}\cap \mathbb{T}}\Bigl\{ Z^\star(r) -  Z^\star(q) \,\leq\, Y^\star(r) - Y^\star(q) + \int_r^q b_{u,v}(s)ds   \Bigr\}
	\]
in $\mathcal{B}(D_\mathbb{R}\times C_\mathbb{R})$. By definition of $\mathbb{T}$, it follows from the continuity of $(s,u,v)\mapsto b_{u,v}(s)$ and the M1 continuity of marginal projections at continuity points \cite[Thm.~12.4.1]{whitt_2002} that $(u,v,\eta ,w)\mapsto \mathbf{1}_{E_{u,v}}(\eta,w)$ is upper semi-continuous with probability one under $\mathbb{Q}^\star$ (for the product topology induced by the M1 topology on $D_\mathbb{R}$ and the uniform topology on $C_\mathbb{R}$). In turn, the weak convergence of $\mathbb{Q}^n$ to $\mathbb{Q}^\star$  implies
\begin{align*}
&\int_{\mathbb{R}^k\times\mathbb{R}^k} \mathbb{Q}^\star_{u,v}(E_{u,v})d\varpi(u,v)=\mathbb{Q}^\star\bigl(   \{ (u,v,\eta , w) : (\eta , w)\in E_{u,v}    \}     \bigr)  \\ 
&\quad \geq \limsup_{n\rightarrow \infty}\mathbb{Q}^n\bigl(   \{ (u,v,\eta , w) : (\eta , w)\in E_{u,v}    \}     \bigr)  = \limsup_{n\rightarrow \infty} \frac{1}{n}\sum_{i=1}^n\mathbb{P}\bigl((X_i,Y_i)\in E_{u^i,v^i}\bigr)= 1,
\end{align*}
where the last two equalities simply follow from the definition of $\mathbb{Q}^n$ and the definition of the particle system. For the rest of the proof, fix an arbitrary pair $(u,v)$ such that $\mathbb{Q}^\star_{u,v}(E_{u,v})=1$. As per the previous observation, such pairs $(u,v)$ have full measure under $\varpi$. By the right-continuity of $Z^\star$ and the continuity of $Y^\star$, we can then conclude that, $\mathbb{Q}^\star_{u,v}$-almost surely, the increment bounds in the definition of $E_{u,v}$ hold for all pairs of times. In particular, we know that $Z^\star$ can only jump downwards, with probability 1 under $\mathbb{Q}^\star_{u,v}$. Moreover, we know that the re-started process $Y_0^\star(t) :=Y^\star(t+\tau_0(Z^\star))-Y^\star(\tau_0(Z^\star))$ defines a new time-changed Brownian motion under $\mathbb{Q}^\star_{u,v}$, so it follows from the law of the iterated logarithm that we have
\begin{equation}\label{eq:iterated_log}
\liminf_{t\downarrow0}\bigl(Z^\star(t+\tau_0(Z^\star))-Z^\star(\tau_0(Z^\star))\bigr)/ h(t)\leq  -1
\end{equation}
$\mathbb{Q}^\star_{u,v}$-almost surely with $h(t)=c\sqrt{t\ln\ln(1/t)}$ for some $c>0$ (depending only on $t\mapsto\sigma(t)$). 

Now consider the set
\[
E_0:= \{ \eta\in D_\mathbb{R} : \tau_0\;\text{is M1-continuous at}\;\eta \}.
\]
in $D_\mathbb{R}$. Clearly, if a path $\eta$ comes with a (non-empty) right-neighbourhood of $\tau_0(\eta)$ where it only takes non-negative values, then there is an endless supply of uniformly convergent sequences $\eta_n \rightarrow \eta$ such that $\tau_0(\eta_n)$ does not converge to $\tau_0(\eta)$ as $n\rightarrow\infty$, and so $\eta \notin E_0$. Conversely, one can easily deduce from the parametric representations in the definition of M1 convergence that, if a given path $\eta$ assumes strictly negative values on any right-neighbourhood of $\tau_0(\eta)$, then all M1 convergent sequences $\eta_n \rightarrow \eta$ must satisfy $\tau_0(\eta_n)\rightarrow \tau_0(\eta)$ as $n\rightarrow \infty$, which implies $\eta \in E_0$. Hence the set $E_0$ is equivalent to the event that $Z^\star$ assumes strictly negative values on any right-neighbourhood of $\tau_0(Z^\star)$. We can readily express this event in terms of countable unions and intersections of Borel sets, so this event is an element of $\mathcal{B}(D_\mathbb{R})$. Moreover, it is immediate from \eqref{eq:iterated_log} that this event has probability one under $\mathbb{Q}^\star_{u,v}$. Since we fixed an arbitrary vector $(u,v)$ in a set of full measure under $\varpi$, we can conclude from \eqref{eq:from_joint_to_marginal} that
\[
	\int_{\mathbb{R}^k\times\mathbb{R}^k}\mathbb{P}^\star_{u,v}( E_0)d\varpi(u,v) = 	\int_{\mathbb{R}^k\times\mathbb{R}^k}\mathbb{Q}^\star_{u,v}( E_0 \times C_\mathbb{R})d\varpi(u,v) = 	1,
\]
and hence $\mathbb{P}^\star_{u,v}( E_0)=1$ for $\varpi$-almost every $(u,v)\in \mathbb{R}^k \times \mathbb{R}^k$, which completes the proof.
\end{proof}

The above proposition is interesting in its own right, but most importantly it allows us to take a generalised continuous mapping approach to the convergence of the feedback, when seen as suitable functionals of the laws of the empirical measures. The starting point is the following lemma.

\begin{lem}[Marginal feedback as a continuous mapping]\label{lem:first_conv_result}
	Suppose $(\mathbf{Q}^n,B^n)\rightarrow (\mathbf{Q}^\star, B^\star)$ almost surely in the product space $(\mathcal{P}(\mathbb{R}^k\times\mathbb{R}^k\times D_\mathbb{R}),\mathfrak{T}_{wk}^{\text{M1}}) \times (C_\mathbb{R},\Vert \!\cdot \!\Vert_\infty)$, for a given probability space $(\Omega_1, \mathcal{F}_1, \mathbb{P}_1)$, with each $(\mathbf{Q}^n, B^n)$ having the same law as $(\mathbf{P}^n, B_0 )$, and the limit $(\mathbf{Q}^\star, B^\star)$ having the same law as $(\mathbf{P}^\star, B_0 )$. Let $(\mathbf{u}^n,\mathbf{v}^n)\rightarrow (\mathbf{u},\mathbf{v})$ be an almost surely convergent sequence in $\mathbb{R}^k \times \mathbb{R}^k$ on some probability space $(\Omega_2,\mathcal{F}_2,\mathbb{P}_2)$ for which the joint law of $(\mathbf{u}^n,\mathbf{v}^n)$ is $\varpi^n$. Writing
	\begin{equation}\label{eq:loss_mappings}
	\mathcal{L}(\mu)_l:= \int_{\mathbb{R}^k\times{\mathbb{R}^k}\times D_\mathbb{R}}u_l\mathbf{1}_{\{\eta \,:\, t \geq \tau_0(\eta) \}}d\mu(u,v,\eta), \quad\text{for}\quad l=1,\ldots,k,
	\end{equation}
	for $\mu \in \mathcal{P}(\mathbb{R}^k\times\mathbb{R}^k\times D_\mathbb{R})$, there is an event $E \in \Omega_1$ with $\mathbb{P}_1(E)=1$ such that, for every $\omega \in E$, we have the marginal convergence
	\[
	\sum_{l=1}^k \mathbf{v}^n_l \int_0^t g_{\mathbf{u}^n,\mathbf{v}^n}(s) d\mathcal{L}(\mathbf{Q}^n(\omega))_l(s) \rightarrow \sum_{l=1}^k \mathbf{v}_l \int_0^t g_{\mathbf{u},\mathbf{v}}(s) d\mathcal{L}(\mathbf{Q}^\star(\omega))_l(s) ,\quad \text{as }\;n\rightarrow \infty,
	\]
	in $\mathbb{R}$, $\mathbb{P}_2$-almost surely, whenever $t$ is a continuity point of each $s\mapsto \mathcal{L}(\mathbf{Q}^\star(\omega))_l(s)$, $l=1,\ldots,n$.
\end{lem}
\begin{proof} By assumption, we can take an event $E \in \mathcal{F}_1$ with $\mathbb{P}_1(E)=1$ on which there is pointwise convergence $(\mathbf{Q}^n(\omega),B^n(\omega))\rightarrow (\mathbf{Q}^\star(\omega), B^\star(\omega))$ for all $\omega \in E$. Moreover, we have
	\[
	\mathbb{E}\Bigl[ 
	\int_{\mathbb{R}^k\times\mathbb{R}^k }\mathbf{Q}^n(\{ u : u\cdot \hat{v}\geq 0     \} )d\varpi^n(\hat{u},\hat{v})\Bigr]  = \int_{\mathbb{R}^k\times\mathbb{R}^k} \varpi^n(\{ u : u\cdot \hat{v}\geq 0     \} )d\varpi^n(\hat{u},\hat{v})=1,
	\]
	as in the proof of Lemma \ref{eq:uv-compact-nonneg}, and hence we can restrict the event $E$ in such a way that we still have $\mathbb{P}_1(E)=1$, while also having that, for all $\omega\in E$, 
	\begin{equation}\label{eq:type_vectors_dot-product}
	\mathbf{Q}^n(\omega)(\{ u : u\cdot \mathbf{v}^n \geq 0     \})=1,\quad \text{for all } n\geq 1,
	\end{equation}
	$\mathbb{P}_2$-almost surely. We can view $\mathbf{Q}^n(\omega)\rightarrow \mathbf{Q}^\star(\omega)$ as convergence in law for suitable random variables and hence apply Skorokhod's representation theorem to yield
	\begin{equation}\label{eq:marginal_feedback_fct} 
		\Phi(\mathbf{Q}^n(\omega), \mathbf{v}^n , t) = \sum_{l=1}^k \mathbf{v}^n_l \mathbb{E} \bigl[    \hat{\mathbf{u}}^n_l \mathbf{1}_{\{t\geq \tau_0(Z^n)\}}\bigr]  ,\quad 	\Phi(\mu, v , t) := \sum_{l=1}^k v_l \mathcal{L}(\mu)_l(t),
		\end{equation}
	for $n\geq 1$, where $\hat{\mathbf{u}}^n \rightarrow \hat{\mathbf{u}}$ almost surely in $\mathbb{R}$ and $Z^n\rightarrow Z$ almost surely in $(D_\mathbb{R},\text{M1})$. Additionally, \eqref{eq:type_vectors_dot-product} ensures that $\mathbf{v}^n\cdot \hat{\mathbf{u}}^n \geq0$, and we have that $\mathbf{v}^n,\hat{\mathbf{u}}^n$ are bounded uniformly in $n\geq 1$. In particular, each $t\mapsto \Phi(\mathbf{Q}^n(\omega), \mathbf{v}^n , t)$ is of finite variation with total variation bounded by a constant uniformly in $n\geq 1$ on any compact time interval, which we use at the end of the proof. 

	Next, by the constructions in Section \ref{subsect:characterise_limit}, it follows from Proposition \ref{prop:hitting_time_cont} that
	\[
	\mathbb{E}[\mathbf{P}^\star(\tau_0\;\text{is M1 continuous at}\;Z)]=1,
	\]
	and $\mathbf{Q}^\star$ has the same law as $\mathbf{P}^\star$ by assumption, so we may further restrict $E$ such that taking an arbitrary $\omega \in E$ in \eqref{eq:marginal_feedback_fct} implies $\tau_0(Z^n)\rightarrow \tau_0(Z)$ almost surely, while retaining $\mathbb{P}_1(E)=1$ (recall the laws of $Z^n, Z$ are fixed by the realisations $\mathbf{Q}^n(\omega), \mathbf{Q}^\star(\omega)$). Consequently, we have $\mathbf{1}_{\{t\geq \tau_0(Z^n)\}}\rightarrow\mathbf{1}_{\{t\geq \tau_0(Z)\}}$ on an event of full probability minus the event $\{\tau_0(Z)=t\}$, on the common probability space where these processes are defined (given by Skorokhod's representation). Now, fix an arbitrary $\omega\in E$ and let $t$ be an arbitrary continuity point of $s\mapsto \mathcal{L}(\mathbf{Q}^\star(\omega))_l(s)$ for each $l=1,\ldots,n$. Then \eqref{eq:loss_mappings} and dominated convergence (along with right-continuity of each $\mathcal{L}(\mathbf{Q}^\star(\omega))_l$) implies
	\[
	0= \Phi(\mathbf{Q}^\star(\omega), \mathbf{v} , t) - \Phi(\mathbf{Q}^\star(\omega), \mathbf{v} , t-) =  \mathbb{E} \bigl[   v\cdot \hat{\mathbf{u}}_l \mathbf{1}_{\{ \tau_0(Z)=t\}}\bigr]\bigr|_{v=\mathbf{v}}
	\]
	Fixing a realisation $v$ of $\mathbf{v}$, if $v\cdot \hat{\mathbf{u}}$ is non-zero (hence strictly positive) on an event of non-negligible probability (for $\mathbb{P}_1$), we must therefore have $\tau_0(Z)\neq t$ on that event (up to a $\mathbb{P}_1$-null set). Therefore, we can conclude from dominated convergence that
	\begin{equation}\label{eq:conv_feedback_cont_point}
	\Phi(\mathbf{Q}^n(\omega), \mathbf{v}^n , t) =  \sum_{l=1}^k \mathbf{v}^n_l \mathbb{E} \bigl[    \hat{\mathbf{u}}^n_l \mathbf{1}_{\{t\geq \tau_0(Z^n)\}}\bigr] \rightarrow \sum_{l=1}^k \mathbf{v}^n_l \mathbb{E} \bigl[    \hat{\mathbf{u}}_l \mathbf{1}_{\{t\geq \tau_0(Z)\}}\bigr] = \Phi(\mathbf{Q}^\star(\omega), \mathbf{v} , t) 
	\end{equation}
	for our arbitrary $\omega \in E$, $\mathbb{P}_2$-almost surely, for any common continuity point of $s\mapsto \mathcal{L}(\mathbf{Q}^\star(\omega))_l(s)$, for $l=1,\ldots,n$. Using integration by parts for Riemann-Stieljtes integrals, we get
	\begin{align}\label{eq:IBP_Stieltjes}
	\int_0^t g_{\mathbf{u}^n,\mathbf{v}^n}(s) d	\Phi(\mathbf{Q}^n(\omega), \mathbf{v}^n , s) &=  	\int_0^t \bigl(g_{\mathbf{u}^n,\mathbf{v}^n}(s) - g_{\mathbf{u},\mathbf{v}}(s)\bigr) d	\Phi(\mathbf{Q}^n(\omega), \mathbf{v}^n , s)\nonumber \\ 
	&+  g_{\mathbf{u},\mathbf{v}}(t)	\Phi(\mathbf{Q}^n(\omega), \mathbf{v}^n , t)- \int_0^t 	\Phi(\mathbf{Q}^n(\omega), \mathbf{v}^n , s) d g_{\mathbf{u},\mathbf{v}}(s).
	\end{align}
	By Assumption \ref{assump:paramters_emperical}, each function  $t\mapsto g_{u,v}(t)$ is continuous and non-decreasing. In particular, it is a standard fact of real analysis that the pointwise convergence $g_{\mathbf{u}^n,\mathbf{v}^n}(s) \rightarrow  g_{\mathbf{u},\mathbf{v}}(s) $ is in fact uniform over $s\in[0,t]$ (alternatively, in the spirit of the present paper, one gets M1 relative compactness from the monotonicity, and the a priori pointwise convergence to a continuous limit then yields the uniform convergence to that limit). Since the total variation of $s\mapsto \Phi(\mathbf{Q}^n(\omega), \mathbf{v}^n , s)$ on $[0,t]$ is bounded uniformly in $n\geq 1$, the first term on the right-hand side of \eqref{eq:IBP_Stieltjes} vanishes as $n\rightarrow 0$. By \eqref{eq:conv_feedback_cont_point}, the second term on the right-hand side tends to $g_{\mathbf{u},\mathbf{v}}(t)\Phi(\mathbf{Q}^\star(\omega), \mathbf{v} , t)$ whenever $t$ is a continuity point. Finally, $dg_{\mathbf{u},\mathbf{v}}(s)$ induces a well-defined Lebesgue--Stieljtes measures, and we get pointwise convergence of the integrands on a dense set of times $s\in[0,t]$ by \eqref{eq:conv_feedback_cont_point}, so dominated convergence and another integration by parts completes the proof.
\end{proof}

We shall use the previous lemma several times. A first application is the following convergence result for the total feedback felt by each particle. This result is important for practical implementations of the model, showing that, after fixing a particular type of bank from the true financial system, the actual feedback from defaults felt by this bank can be approximated via only the $k$ feedback processes for the mean-field model.

\begin{prop}[Total feedback for a tagged institution]\label{prop:conv_total_feedback}
Let Assumption \ref{assump:paramters_emperical} be satisfied, and fix any given pair of indexing vectors $(u^i,v^i)$. Let the limit point $(\mathbf{P}^\star,B_0)$ be achieved along a subsequence still indexed by $n\geq1$. Then, as $n\rightarrow\infty$, the total feedback
\begin{equation}\label{eq:total_feedback|_process}
t\mapsto \sum_{l=1}^k {v^i_l}\int_0^tg_i(s)d\mathcal{L}_{l,i}^{n}(s),
\end{equation}
 felt by the $i$'th particle, converges in law at the process level on $D_\mathbb{R}[0,T_0]$, for any $T_0\in \mathbb{T}_\star$, to
\begin{equation}\label{eq:total_feedback|_process_limit}
t\mapsto \sum_{l=1}^k {v^i_l}\int_0^tg_{u^i,v^i}(s)d\mathcal{L}^\star_{l}(s),
\end{equation}
where $g_{u,v}(s)\!:=\!g(u,v,\mathbb{E}[\mathbf{u}],\mathbb{E}[\mathbf{v}],s)$ and $\mathcal{L}^\star_{1},\ldots,\mathcal{L}^\star_{k}$ are defined by \eqref{eq:limit_feedback_processes} on $(\Omega_0,\mathcal{B}(\Omega_0),\mathbb{P}^\star_0)$.
\end{prop}
\begin{proof} 
	First of all, Assumption \ref{assump:paramters_emperical} implies that the total feedback processes \eqref{eq:total_feedback|_process} are non-decreasing on $[0,T]$ for each $i=1,\ldots,n$ and $n\geq1$. Based on Lemma \ref{lem:contagion_at_0}, it is thus straightforward to verify the conditions \eqref{eq:tight1}--\eqref{eq:tight4} for these processes in place of the different particle trajectories $t\mapsto X^i(t)$, and so the arguments of Proposition \ref{prop:tightness} yield tightness of \eqref{eq:total_feedback|_process} in $D_\mathbb{R}[0,T_0]$, for any $T_0\in \mathbb{T}_\star$,  under the M1 topology. Note that we can write
	\[
	\sum_{l=1}^k {v^i_l}\int_0^tg_i(s)d\mathcal{L}_{l,i}^{n}(s)=\sum_{l=1}^k {v^i_l}\int_0^tg_i(s)d\mathcal{L}(\mathbf{P}^{n})_l(s)- \frac{1}{n} {v^i}\cdot u^ig_i(\tau_i)\mathbf{1}_{t\geq \tau_i},
	\]
	where each $\mathcal{L}(\mu)$ is defined as in \eqref{eq:loss_mappings}. Sending $n\rightarrow \infty$, the second term on the right-hand side vanishes uniformly in $i\leq n $ and $t\in[0,T]$, by the assumptions on $u^i$, $v^i$, and $g_i$ in Assumption \ref{assump:paramters_emperical}, so we only need to consider the convergence of the first term. To this end, we can see from the definition of $\mathbb{P}_{u,v}^{\mu,w}$ in Section \ref{subsect:characterise_limit} that, as stochastic processes,
	\begin{align}\label{eq:feedback_general_formulation}
	\mathcal{L}^\star_{l}(t) = \int_{\mathbb{R}^k\times\mathbb{R}^k} u_l \mathbb{P}_{u,v}^{\mathbf{P}^\star\!\!, B_0}\bigl( t\geq \tau_0(Z)\bigr)d\varpi(u,v) = \int_{\mathbb{R}^k\times{\mathbb{R}^k}\times D_\mathbb{R}}u_l\mathbf{1}_{\{\eta\, :\, t\geq  \tau_0(\eta) \}}d\mathbf{P}^\star(u,v,\eta)
	\end{align}
	almost surely for the probability space $(\Omega_0,\mathcal{B}(\Omega_0),\mathbb{P}^\star_0)$ from Section \ref{subsect:characterise_limit}. That is, we in fact have $\mathcal{L}^\star_{l}=\mathcal{L}(\mathbf{P}^\star)_l$ for each $l=1,\ldots,k$. From the definition of $\mathbb{T}_\star$ in \eqref{eq:cont_times} and the relation between $\mathbb{P}^\star_{u,v}$ and $\mathbb{P}_0^\star$ in \eqref{eq:P_uv_defn}, we can furthermore see that
	\[
	\mathbb{T}_\star \subseteq \bigl\{t\in[0,T_0] : \mathbb{P}^\star_0 \bigl( \mathcal{L}(\mathbf{P}^\star)_l(t) = \mathcal{L}(\mathbf{P}^\star)_l(t-),\;l=1,\ldots,k    \bigl)=1 \bigr\}.
	\]
	Consequently, it holds for $\mathbb{P}^\star_0$-almost all $\omega\in \Omega_0$ that every $t\in\mathbb{T}_\star$ is a common continuity point of the paths $s\mapsto \mathcal{L}^\star_{l}(\omega)(s)=\mathcal{L}(\mathbf{P}^\star(\omega))_l(s)$, for $l=1,\ldots,k$. In turn, by first applying Skorokhod's representation theorem to the sequence $(\mathbf{P}^n,B_0)$, converging in law to $(\mathbf{P}^\star,B_0)$, we are now in a position to apply Lemma \ref{lem:first_conv_result} for every $t\in \mathbb{T}_\star$ (noting that the proof of Lemma \ref{lem:first_conv_result} also holds when we take $(\mathbf{u}^n,\mathbf{v}^n)=(u^i,v^i)$, since \eqref{eq:type_vectors_dot-product} is then still satisfied). This gives us that the finite dimensional distributions of any limit point of \eqref{eq:total_feedback|_process} agrees with those of \eqref{eq:total_feedback|_process_limit} for all finite sets of times $t_1,\ldots,t_m \in \mathbb{T}_\star$. Since we also have M1-tightness in $D_\mathbb{R}[0,T_0]$, the conclusion follows (recalling that the Borel sigma algebra on $D_\mathbb{R}[0,T_0]$ for the M1 topology is generated by the finite dimensional projections).
\end{proof}

\subsection{Martingale properties with respect to the limiting Markov kernel}

The previous subsections guide us towards certain martingale properties with respect to the Markov kernel $(\mathbb{P}^\star_{u,v})_{(u,v)\in\mathbb{R}^k \times \mathbb{R}^k}$. Using these, we will be able to identify the limiting mean-field problem.

\begin{prop}[Limiting martingale properties] \label{prop:martingale_props}
	Let Assumption \ref{assump:paramters_emperical} be satisfied. For any given $(u,v)\in\mathbb{R}^k\times \mathbb{R}^k$, we consider the probability spaces $(\Omega_\star, \mathcal{B}(\Omega_\star),\mathbb{P}_{u,v}^\star)$ with $\mathbb{P}_{u,v}^\star$ given by \eqref{eq:P_uv_defn}. We then define a c\`adl\`ag stochastic process
	$M_{u,v}^\star: \Omega_\star \rightarrow D_\mathbb{R}$ by
	\begin{equation}\label{Z_uv}
	M^\star_{u,v}(t) := Z(t) - Z(0) - \int_0^t b_{u,v}(s)ds - F\Bigl( \sum_{l=1}^k v_l \int_0^t g_{u,v}(s)d \mathcal{L}_l^\star(s) \Bigr),
	\end{equation}
	where each $\mathcal{L}_l^\star$ is defined in \eqref{eq:limit_feedback_processes}. We write $b_{u,v}(t)=b(u,v,t)$ and $\sigma_{u,v}(t)=\sigma(u,v,t)$ as well as $g_{u,v}(t):=g(u,v,\mathbb{E}[\mathbf{u}],\mathbb{E}[\mathbf{v}],t)$. Under $\mathbb{P}_{u,v}^\star$, for $\varpi$-almost every $(u,v)\in\mathbb{R}^k\times\mathbb{R}^k$, the process $M_{u,v}^\star$ is a continuous martingale for its natural filtration on $[0,T]$ with
	\begin{equation}\label{eq:quad_var}
	\langle M_{u,v}^\star \rangle_t = \int_0^t\sigma_{u,v}^2(s)ds \quad \text{and}\quad \langle M_{u,v}^\star, B_0 \rangle_t = \int_0^t \rho \sigma_{u,v}(s)ds.
	\end{equation}
	\end{prop}
\begin{proof} 
	 We begin by fixing a countable family $\{ \phi_j \}_{j=1}^\infty$ of bounded continuous functions $\phi_j:\mathbb{R} \rightarrow \mathbb{R}$ such that $M^\star_{u,v}$ is a martingale under $\mathbb{P}^\star_{u,v}$ provided we have the identity
	\begin{equation}\label{eq:martingale_prop}
	\mathbb{E}_{\mathbb{P}^\star_{u,v}} \Bigl[ M^\star_{u,v}(t^\prime) \prod_{i=1}^m\phi_{j_i}(M^\star_{u,v}(t_i))  \Bigr]= 	\mathbb{E}_{\mathbb{P}^\star_{u,v}} \Bigl[ M^\star_{u,v}(t) \prod_{i=1}^m\phi_{j_i}(M^\star_{u,v}(t_i))  \Bigr],
	\end{equation}
	for any choices of $\phi_{j_1},\ldots, \phi_{j_m}$ in $\{ \phi_j \}_{j=1}^\infty$ and $t_1,\ldots,t_m\leq t < t^\prime$ in $[0,T]$, where $m\geq 1$ is arbitrary. Since we also have right-continuity of the paths, to establish \eqref{eq:martingale_prop} it suffices to consider functionals
	\begin{equation}\label{eq:theta_functional}
	\Phi^{j_1,\ldots,j_m}_{q,q^\prime;q_1,\ldots,q_m}\!(\eta):=\bigl( (\eta)(q^\prime))-(\eta)(q)\bigr)\prod_{i=1}^m\phi_{j_i}(\eta(q_i)), \quad \eta \in D_\mathbb{R},
	\end{equation}
	for all rationals $q,q^\prime,q_1,\ldots,q_m \in \mathbb{T}_\star \cap \mathbb{Q}$ with $q_1,\ldots,q_m\leq q <q^\prime$, where $\mathbb{T}_\star $ is given by \eqref{eq:cont_times}, and show that, for any such functional, we have
	\begin{equation}\label{eq:limting_martingale_functional}
	\int_{\mathbb{R}^k\times\mathbb{R}^k}f(u,v)\mathbb{E}_{\mathbb{P}^\star_{u,v}}\bigl[ \Phi^{j_1,\ldots,j_m}_{q,q^\prime;q_1,\ldots,q_m}\!(M^\star_{u,v})\bigr]d\varpi(u,v) =0,
	\end{equation}
	for all bounded continuous functions $f: \mathbb{R}^k \times \mathbb{R}^k\rightarrow \mathbb{R}$. Using the notation of \eqref{eq:loss_mappings}, consider also the functionals
	\begin{equation}\label{eq:2nd_functional}
\hat{\Phi}^{j_1,\ldots,j_m}_{q,q^\prime;q_1,\ldots,q_m}\!(u,v ,\eta, \mu):= \Phi^{j_1,\ldots,j_m}_{q,q^\prime;q_1,\ldots,q_m}\!\Bigl( \eta_\cdot - \eta_0 -\int_0^\cdot\!b_{u,v}(s)ds - F\Bigl( \sum_{l=1}^kv_l \int_0^\cdot \!g_{u,v}(s)d\mathcal{L}(\mu)_l(s) \!\Bigr)\!  \Bigr).
	\end{equation}
By \eqref{eq:defn_big_P_star}, \eqref{eq:disintegration} and \eqref{eq:P_uv_defn}, we can then confirm that the left-hand side of \eqref{eq:limting_martingale_functional} coincides with
\begin{equation}\label{eq:rewritten_expectation}
\mathbb{E}_{\mathbb{P}_0^\star} \Bigl[  \int_{\mathbb{R}^k\times \mathbb{R}^k \times D_\mathbb{R}} f(u,v)\hat{\Phi}^{j_1,\ldots,j_m}_{q,q^\prime;q_1,\ldots,q_m}\!(u,v,\eta, \mathbf{P}^\star  )d\mathbf{P}^\star(u,v,\eta)     \Bigr].
\end{equation}
	Applying Skorokhod's representation theorem, we can express this as
	\begin{equation}\label{eq:1st_main_functional}
	\mathbb{E} \Bigl[  \int\!\! f(u,v)\hat{\Phi}^{j_1,\ldots,j_m}_{q,q^\prime;q_1,\ldots,q_m}\!(u,v,\eta, \mathbf{Q}^\star  )d\mathbf{Q}^\star(u,v,\eta)     \Bigr],
	\end{equation}
	and we can then apply Skorokhod's representation theorem again, for any given realisation of $\mathbf{Q}^\star$, to be able to write
	\begin{equation}\label{eq:1st_main_functional_skorokhod_rep}
	\int\!\! f(u,v)\hat{\Phi}^{j_1,\ldots,j_m}_{q,q^\prime;q_1,\ldots,q_m}\!(u,v,\eta, \mathbf{Q}^\star  )d\mathbf{Q}^\star(u,v,\eta)=\mathbb{E} \bigl[   f(\mathbf{u},\mathbf{v})\hat{\Phi}^{j_1,\ldots,j_m}_{q,q^\prime;q_1,\ldots,q_m}\!(\mathbf{u},\mathbf{v},Z^\star, \mathbf{Q}^\star  ) \bigr]
	\end{equation}
	where there is a sequence $(\mathbf{u}^n,\mathbf{v}^n,Z^n) \rightarrow (\mathbf{u},\mathbf{v},Z^\star)$ almost surely in $\mathbb{R}^k\times \mathbb{R}^k \times D_\mathbb{R}$, for some common probability space $(\Omega_2,\mathcal{F}_2,\mathbb{P}_2)$, and a sequence $\mathbf{Q}^n\rightarrow \mathbf{Q}^\star$ almost surely in $\mathcal{P}(\mathbb{R}^k\times \mathbb{R}^k \times D_\mathbb{R})$, for some common probability space $(\Omega_1,\mathcal{F}_1,\mathbb{P}_1)$. Naturally, the point is that each $(\mathbf{u}^n,\mathbf{v}^n,Z^n) $ is distributed according to $\mathbf{Q}^n(\omega_1)$, for a given $\omega_1 \in 
	\Omega_1$, while $\mathbf{Q}^n$ as a random variable has the same law as $\mathbf{P}^n$. In particular, we note that the expectation on the right-hand side in the above expression does not act on $\mathbf{Q}^\star$, and that the sequence $(\mathbf{u}^n,\mathbf{v}^n,Z^n)\rightarrow(\mathbf{u},\mathbf{v},Z^\star)$ will be different for each $\omega_1\in \Omega_1$.
	
	It follows from the above that the sequences $(\mathbf{u}^n,\mathbf{v}^n)_{n\geq1}$ and $(\mathbf{Q}^n)_{n\geq 1}$ satisfy the assumptions of Lemma \ref{lem:first_conv_result}. Moreover, from the definition of $\mathbb{T}_\star$ and the distributional properties of  $\mathbf{Q}^\star$, we can see that
	\[
	\mathbb{E} \bigl[ \mathbf{Q}^\star \bigl( \{ \eta \in D_\mathbb{R}: \eta(t)=\eta(t-) \}  \bigr) \mathbf{1}_{\{ \mathcal{L}(\mathbf{Q}^\star)_l(t)=\mathcal{L}(\mathbf{Q}^\star)_l(t-), \;l=1,\ldots, k \}} \bigr] =1
	\]
	for all $t\in \mathbb{T}_\star$. Therefore, every $t \in \mathbb{T}_\star$ satisfies that, for $\mathbb{P}_1$-almost all $\omega_1\in \Omega_1$, $t$ is a common continuity point of the paths $\mathcal{L}(\mathbf{Q}(\omega_1))_l$, for $l=1,\ldots, k$, and it is $\mathbb{P}_2$-almost surely a continuity point of $Z^\star$ when the latter has the law $\mathbf{Q}^\star(\omega_1)\circ \pi_{D_{\mathbb{R}}}^{-1}$ on $D_\mathbb{R}$. In particular, whenever $t\in \mathbb{T}_\star$, we have $Z^n_t\rightarrow Z^\star_t$, $\mathbb{P}_2$-almost surely, for $\mathbb{P}_1$-almost every $\omega_1\in \Omega_1$. Moreover, it also follows that we can apply Lemma \ref{lem:first_conv_result} at any set of times $q,q^\prime,q_1,\ldots,q_m\in \mathbb{T}_\star$ for a $\mathbb{P}_1$-almost sure subset of $\Omega_1$. Thus, we can deduce from the definition of $\hat{\Phi}$ in \eqref{eq:2nd_functional} that
	\[
	\hat{\Phi}^{j_1,\ldots,j_m}_{q,q^\prime;q_1,\ldots,q_m}\!(\mathbf{u}^n,\mathbf{v}^n,Z^n, \mathbf{Q}^n(\omega_1)  )\rightarrow \hat{\Phi}^{j_1,\ldots,j_m}_{q,q^\prime;q_1,\ldots,q_m}\!(\mathbf{u},\mathbf{v},Z^\star, \mathbf{Q}^\star(\omega_1)  ), \quad \text{as}\quad n\rightarrow \infty,
	\]
	$\mathbb{P}_2$-almost surely, for $\mathbb{P}_1$-almost every $\omega_1\in \Omega_1$. In turn, dominated convergence gives
	\[
	\mathbb{E} \bigl[   f(\mathbf{u},\mathbf{v})\hat{\Phi}^{j_1,\ldots,j_m}_{q,q^\prime;q_1,\ldots,q_m}\!(\mathbf{u},\mathbf{v},Z^\star, \mathbf{Q}^\star ) \bigr]=\lim_{n\rightarrow \infty} \mathbb{E} \bigl[   f(\mathbf{u}^n,\mathbf{v}^n)\hat{\Phi}^{j_1,\ldots,j_m}_{q,q^\prime;q_1,\ldots,q_m}\!(\mathbf{u}^n,\mathbf{v}^n,Z^n, \mathbf{Q}^n  ) \bigr]
	\]
	$\mathbb{P}_1$-almost surely. Recalling \eqref{eq:rewritten_expectation}--\eqref{eq:1st_main_functional_skorokhod_rep}, and passing through our repeated use of Skorokhod's representation theorem, another application of dominated convergence then shows that the left-hand side of \eqref{eq:limting_martingale_functional} is in fact equal to the limit
	\[
 \lim_{n\rightarrow \infty} \mathbb{E} \Bigl[  \int_{\mathbb{R}^k\times \mathbb{R}^k \times D_\mathbb{R}} f(u,v)\hat{\Phi}^{j_1,\ldots,j_m}_{q,q^\prime;q_1,\ldots,q_m}\!(u,v,\eta, \mathbf{P}^n )d\mathbf{P}^n(u,v,\eta)     \Bigr].
	\]
	Looking at \eqref{eq:2nd_functional}, the definition of the empirical measures $\mathbf{P}^n$ immediately gives that this limit is zero, and so we have established \eqref{eq:limting_martingale_functional}. Using a similar approach, we can pass to the limit in   Kolmogorov's continuity criterion to verify that, for $\varpi$-almost every $(u,v)\in\mathbb{R}^k\times\mathbb{R}^k$, each martingale $M^\star_{u,v}$ has a continuous version under $\mathbb{P}^\star_{u,v}$, which completes the first part of the proof.

It remains to identify the quadratic variation of $M^\star_{u,v}$ as well as its quadratic co-variation with the limiting common factor $B_0$. To this end, we can argue exactly as we did above in order to conclude, from the limiting procedure, that
	\begin{align*}
&\int_{\mathbb{R}^k\times\mathbb{R}^k}f(u,v)\mathbb{E}_{\mathbb{P}^\star_{u,v}}\Bigl[ \Phi^{j_1,\ldots,j_m}_{q,q^\prime;q_1,\ldots,q_m}\!\Bigl((M_{u,v}^\star)^2- \int_0^\cdot \sigma_{u,v}^2(s)ds\Bigr) \Bigr]d\varpi(u,v) =0,
\end{align*}
and 
	\begin{align*}
&\int_{\mathbb{R}^k\times\mathbb{R}^k}f(u,v)\mathbb{E}_{\mathbb{P}^\star_{u,v}}\Bigl[ \Phi^{j_1,\ldots,j_m}_{q,q^\prime;q_1,\ldots,q_m}\!\Bigl(M^\star_{u,v}B_0 - \int_0^\cdot \rho\sigma_{u,v}(s)ds\Bigr) \Bigr]d\varpi(u,v) =0,
\end{align*}
for all bounded continuous functions $f:\mathbb{R}^k\times\mathbb{R}^k$ and any choice of the countably many functionals $\Phi^{j_1,\ldots,j_m}_{q,q^\prime;q_1,\ldots,q_m}$ defined in \eqref{eq:theta_functional}.
It then follows that, for $\varpi$-almost every $(u,v)\in\mathbb{R}^k\times\mathbb{R}^k$, the two processes $(M^\star_{u,v})^2- \int_0^\cdot \sigma_{u,v}^2(s)ds$ and $M^\star_{u,v}B_0 - \int_0^\cdot \rho\sigma_{u,v}(s)ds$ are also continuous martingales on $[0,T]$, under $\mathbb{P}^\star_{u,v}$, and so we obtain the final conclusion \eqref{eq:quad_var}.
\end{proof}

\subsection{Characterising the limit points: proof of Theorem \ref{thm:mf_conv}}\label{subsect:characterise_limit_proof}

Based on the work in the previous subsections, we can now establish the following result, which completes the proof of Theorem \ref{thm:mf_conv}.

\begin{prop}[Limiting McKean--Vlasov solutions]\label{prop:P_independent}
	Let $(u,v)\mapsto\mathbb{P}_{u,v}^\star \in\mathcal{P}(\Omega_\star)$ be the Markov kernel defined in \eqref{eq:P_uv_defn} from a given limit point $(\mathbf{P}^\star,B_0)$. Then, for $\varpi$-almost every $(u,v)\in \mathbb{R}^k\times\mathbb{R}^k$, there is a Brownian motion $B_{u,v}$ on $(\Omega_\star, \mathcal{B}(\Omega_\star),\mathbb{P}_{u,v}^\star)$ such that $(\mathbf{P}^\star,B_0)$, $B_{u,v}$, and $Z(0)$ are mutually independent on $(\Omega_\star, \mathcal{B}(\Omega_\star),\mathbb{P}_{u,v}^\star)$ with
	\[
		\mathcal{L}_l^\star(s)=\int_{\mathbb{R}^k\times\mathbb{R}^k} u_l \mathbb{P}^\star_{u,v}\bigl(t \geq \tau_0(Z) \, | \, B_0, \mathbf{P}^\star \bigr)d\varpi(u,v),
	\]
where $Z$ has dynamics
	\[
	dZ(t)= b_{u,v}(t)dt + \sigma_{u,v}(t)d(\rho B_0(t)+\sqrt{1-\rho^2}B_{u,v}(t)) - dF\Bigl( \sum_{l=1}^k v_l \int_0^t g_{u,v}(s)d \mathcal{L}_l^\star(s) \Bigr),
	\]
under $\mathbb{P}_{u,v}^\star$ and its starting point $Z(0)$ has density $V(\cdot|u,v)$ under $\mathbb{P}_{u,v}^\star$.
\end{prop}
\begin{proof}
	Let $M^\star_{u,v}$ be given by \eqref{Z_uv}, which yields a continuous martingale under $\varpi$-almost every $\mathbb{P}_{u,v}^\star$ by Proposition \ref{prop:martingale_props}. Since $B_0$ is a Brownian motion under $\varpi$-almost every $\mathbb{P}_{u,v}^\star$, by construction of the probability space, we can define
 $\tilde{Z}_{u,v}:\Omega_\star \rightarrow D_\mathbb{R}$ by
	\begin{equation}\label{eq:Z-tilde}
	\tilde{Z}_{u,v}(t):= M^\star_{u,v}(t)-\int_0^t \rho \sigma_{u,v}(s) dB_0(s)
	\end{equation}
	and note that this is a continuous martingale under $\mathbb{P}_{u,v}^\star$ for $\varpi$-almost every $(u,v)\in\mathbb{R}^k\times \mathbb{R}^k$. Defining also $B_{u,v}: \Omega_\star \rightarrow D_\mathbb{R}$ by
	\begin{equation}\label{eq:B_uv}
	B_{u,v}(t):= \int_0^t \frac{1}{\sigma_{u,v}(s)\sqrt{1-\rho^2}} d\tilde{Z}_{u,v}(s),
	\end{equation}
	we can then conclude from \eqref{eq:quad_var} in Proposition \ref{prop:martingale_props} and Assumption \ref{assump:paramters_emperical} that, under $\mathbb{P}_{u,v}^\star$, for $\varpi$-almost every $(u,v)\in\mathbb{R}^k\times \mathbb{R}^k$, the processes 
	$B_0$ and $B_{u,v}$ are continuous martingales on $[0,T]$ with
	\[
	\langle B_{0}\rangle_t=t, \quad \langle B_{u,v}\rangle_t = t \quad \text{and} \quad \langle B_{u,v}, B_{0}\rangle_t = 0,
	\]
	for all $t\in[0,T]$. Therefore, Levy's characterisation theorem gives that $B_{u,v}$ and $B_0$ are independent Brownian motions on $(\Omega_\star , \mathcal{B}(\Omega_\star))$ under $\mathbb{P}_{u,v}^\star$ for $\varpi$-almost every $(u,v)\in\mathbb{R}^k\times \mathbb{R}^k$.
	
	Recalling the construction of the Markov kernel $\mathbb{P}_{u,v}^{\mu,w}$ from \eqref{eq:1st_Markov_kernel}-\eqref{eq:P_uv_defn}, we can readily deduce from the weak convergence of $\mathbf{P}_0^n$, see \eqref{eq:weak_conv_net_initial}, that, for $\varpi$-almost every $(u,v)\in\mathbb{R}^k\times\mathbb{R}^k$,
	\begin{equation}\label{eq:initial_limit}
	\mathbb{E}_{\mathbb{P}_{u,v}^{\mu,w}} \bigl[  \phi(Z(0))\bigr] = \mathbb{E}_{\mathbb{P}_{u,v}^\star} \bigl[  \phi(Z(0))\bigr]= \int_{\mathbb{R}}\phi(x) V_0(x|u,v)dx
	\end{equation}
	for $\mathbb{P}_0^\star$-almost all $(\mu,w)\in\mathcal{P}_\varpi( D_\mathbb{R})\times C_\mathbb{R}$. Next, we define
	\begin{equation}\label{eq:func_check_martingale}
	\Psi^{j;j_1,\ldots,j_m}_{q,q^\prime;q_1,\ldots,q_m}\!(\tilde{\eta}, \eta) :=\bigl(\tilde{\eta}(q^\prime)-\tilde{\eta}(q)\bigr)\prod_{i=1}^m\phi_{j_i}(\tilde{\eta}(q_i))\phi_{j}(\eta(0)).
	\end{equation}
	Then we can consider the functionals
	\[
	\hat{\Psi}^{j;j_1,\ldots,j_m}_{q,q^\prime;q_1,\ldots,q_m} (u,v,\mu,w,\eta) = \Psi^{j;j_1,\ldots,j_m}_{q,q^\prime;q_1,\ldots,q_m}\bigr( \tilde{Z}_{u,v}(\mu ,w, \eta), Z(\mu, w, \eta) \bigr),
	\]
	where we recall that $Z(\mu, w, \eta)=\eta$. Similarly to \eqref{eq:rewritten_expectation} in the proof of Proposition \ref{prop:martingale_props}, we can deduce from the constructions in Section \ref{subsect:characterise_limit} that 
	the expression
	\begin{equation}\label{eq:com_BM_martingale}
	\int_{\mathcal{P}_\varpi(D_\mathbb{R})\times C_\mathbb{R}} \biggl( \int_{\mathbb{R}^k\times \mathbb{R}^k} f(u,v) \mathbb{E}_{\mathbb{P}_{u,v}^{\mu,w}} \bigl[  \Psi^{j_0,j_1,\ldots,j_m}_{q,q^\prime;q_1,\ldots,q_m}\!\bigl(  \tilde{Z}_{u,v}(\mu,w, \cdot), Z(\mu,w,\cdot) \bigr) \bigr]  d\varpi(u,v) \biggr)^{\!2} d\mathbb{P}_0^\star(\mu,w)
	\end{equation}
	coincides with
\begin{equation*}\label{eq:rewritten_integral}
\mathbb{E}_{\mathbb{P}_0^\star} \biggl[  \biggl( \int_{\mathbb{R}^k\times \mathbb{R}^k \times D_\mathbb{R}} f(u,v)\hat{\Psi}^{j,j_1,\ldots,j_m}_{q,q^\prime;q_1,\ldots,q_m}\!(u,v, \mathbf{P}^\star, B_0, \eta  )d\mathbf{P}^\star(u,v,\eta) \biggr)^{\!2} \,   \biggr].
\end{equation*}
	By a repeated use of Skorokhod's representation theorem, we can write this as
	\[
	\mathbb{E} \biggl[  \biggl( \int_{\mathbb{R}^k\times \mathbb{R}^k \times D_\mathbb{R}} f(u,v)\hat{\Psi}^{j,j_1,\ldots,j_m}_{q,q^\prime;q_1,\ldots,q_m}\!(u,v, \mathbf{Q}^\star, B^\star, \eta  )d\mathbf{Q}^\star(u,v,\eta) \biggr)^{\!2} \,   \biggr]
	\]
	with 
	\begin{equation*}
	\int\!\! f(u,v)\hat{\Psi}^{j,j_1,\ldots,j_m}_{q,q^\prime;q_1,\ldots,q_m}\!(u,v, \mathbf{Q}^\star, B^\star , \eta  )d\mathbf{Q}^\star(u,v,\eta)=\mathbb{E} \bigl[   f(\mathbf{u},\mathbf{v})\hat{\Psi}^{j,j_1,\ldots,j_m}_{q,q^\prime;q_1,\ldots,q_m}\!(\mathbf{u},\mathbf{v}, \mathbf{Q}^\star, B^\star , Z^\star  ) \bigr],
	\end{equation*}
where $(\mathbf{Q}^\star, B^\star)$ and $(\mathbf{u},\mathbf{v}, Z^\star )$ are the limits of suitable almost surely converging sequences on two separate probability spaces $(\Omega_1 , \mathcal{F}_1, \mathbb{P}_1)$ and $(\Omega_2 , \mathcal{F}_2, \mathbb{P}_2)$. In particular, the expectation on the right-hand side above does not act on the pair $(\mathbf{Q}^\star, B^\star)$. At this point, the main difference from the arguments in Proposition \ref{prop:martingale_props} concerns how to ascertain the $\mathbb{P}_2$-almost sure convergence
\[
\tilde{Z}_{\mathbf{u}^n,\mathbf{v}^n}(\mathbf{Q}^n(\omega_1),B^n(\omega_1), Z^n) \rightarrow \tilde{Z}_{\mathbf{u},\mathbf{v}}(\mathbf{Q}^\star(\omega_1),B^\star(\omega_1), Z^\star) ,
\]
for $\mathbb{P}_1$-almost all $\omega_1 \in \Omega_1$. By the distribution of the $(\mathbf{u}^n,\mathbf{v}^n)$'s and Assumption \ref{assump:paramters_emperical}, we have that (almost surely) the sequence of functions $\sigma_{\mathbf{u}^n,\mathbf{v}^n}$ has H{\"o}lder norms on $[0,T]$ that are bounded uniformly in $n\geq 1$, for some H{\"o}lder exponent $\beta >1/2$. In particular, Arzela--Ascoli and the pointwise convergence gives that $\sigma_{\mathbf{u}^n,\mathbf{v}^n}$ converges uniformly to $\sigma_{\mathbf{u},\mathbf{v}}$ on $[0,T]$ (almost surely). Furthermore, we have that $B^n$ and $B^\star$ are Brownian motions with $B^n$ converging uniformly to $B^\star$ on $[0,T]$ (almost surely). Therefore, standard properties of Young integrals give that (almost surely) we can interpret the stochastic integral in \eqref{eq:Z-tilde} pathwise and we have
\[
\int_0^t \rho \sigma_{\mathbf{u}^n,\mathbf{v}^n}(s) dB^n(\omega_1)(s) \rightarrow \int_0^t \rho \sigma_{u,v}(s) dB^\star(\omega_1)(s),
\]
uniformly on $[0,T]$, $\mathbb{P}_2$-almost surely, for $\mathbb{P}_1$-almost all $\omega_1\in\Omega_1$. From here, arguing as in Proposition \ref{prop:martingale_props}, via Lemma \ref{lem:first_conv_result}, we arrive at the conclusion that \eqref{eq:com_BM_martingale} must be equal to the limit of
\begin{equation}\label{eq:rewritten_integral_prelimit}
\mathbb{E} \biggl[  \biggl( \int_{\mathbb{R}^k\times \mathbb{R}^k \times D_\mathbb{R}} f(u,v)\hat{\Psi}^{j,j_1,\ldots,j_m}_{q,q^\prime;q_1,\ldots,q_m}\!(u,v, \mathbf{P}^n, B_0, \eta  )d\mathbf{P}^n(u,v,\eta) \biggr)^{\!2} \,   \biggr]
\end{equation}
as $n\rightarrow \infty$. By the independence of the Brownian motions $B_i$ in the finite particle system, and their independence of $X_i(0)$, we can easily check that \eqref{eq:rewritten_integral_prelimit} is of order $1/n$ as $n\rightarrow\infty$, and so we conclude that \eqref{eq:com_BM_martingale} is zero. By the form of \eqref{eq:func_check_martingale}, we therefore obtain that, for $\mathbb{P}_0^\star$-almost all $(\mu,w)\in\mathcal{P}_\varpi \times C_\mathbb{R}$, the process $\tilde{Z}_{u,v}(\mu ,w, \cdot) : D_\mathbb{R} \rightarrow D_\mathbb{R}$ is a martingale under $\mathbb{P}_{u,v}^{\mu,w}$, independently of $Z(0)$, for $\varpi$-almost every $(u,v)\in\mathbb{R}^k\times\mathbb{R}^k$.

We can now repeat the above argument with $\tilde{Z}_{u,v}^2 - \int_0^\cdot(1-\rho^2)\sigma_{u,v}^2(s)ds$ to deduce that each process $\tilde{Z}_{u,v}(\mu ,w, \cdot) : D_\mathbb{R} \rightarrow D_\mathbb{R}$ has quadratic variation $\int_0^\cdot(1-\rho^2)\sigma_{u,v}^2(s)ds$, and, similarly, another limiting argument gives continuity of the paths. It then follows that, for $\mathbb{P}_0^\star$-almost all $(\mu,w)\in\mathcal{P}_\varpi(D_\mathbb{R}) \times C_\mathbb{R}$, the process $\tilde{B}_{u,v}(\mu ,w, \cdot) : D_\mathbb{R} \rightarrow D_\mathbb{R}$ is a Brownian motion under $\mathbb{P}_{u,v}^{\mu,w}$, independently of $Z(0)$, for $\varpi$-almost every $(u,v)\in\mathbb{R}^k\times\mathbb{R}^k$. Using this and \eqref{eq:initial_limit}, we get
	\begin{align}\label{eq:independence}
\int_{\mathbb{R}^k\times \mathbb{R}^k} &f(u,v)\mathbb{E}_{\mathbb{P}_{u,v}^\star} \bigl[ \phi(Z(0)) \varphi(B_{u,v}) \psi(\mathbf{P^\star},B_0) \bigr]d\varpi(u,v) \nonumber\\
&= \int_{\mathcal{P}_\varpi(D_\mathbb{R})\times C_\mathbb{R}} \psi(\mu,w)  \int_{\mathbb{R}\times \mathbb{R}^k} f(u,v) \mathbb{E}_{\mathbb{P}_{u,v}^{\mu,w}} \bigl[  \phi(Z(0)) \varphi(B_{u,v}(\mu,w,\cdot)) \bigr] d\varpi(u,v)d\mathbb{P}_0^\star(\mu,w)\nonumber \\
& = \int_{\mathcal{P}_\varpi(D_\mathbb{R})\times C_\mathbb{R}} \psi(\mu,w)  \int_{\mathbb{R}\times \mathbb{R}^k} f(u,v) \mathbb{E}_{\mathbb{P}_{u,v}^{\star}} \bigl[  \phi(Z(0))\bigr] \mathbb{E}_{\mathbb{P}_{u,v}^{\star}} \bigl[\varphi(B_{u,v}) \bigr] d\varpi(u,v)d\mathbb{P}_0^\star(\mu,w)\nonumber \\
& = \int_{\mathbb{R}\times \mathbb{R}^k} f(u,v)  \mathbb{E}_{\mathbb{P}_{u,v}^{\star}} \bigl[  \phi(Z(0))\bigr] \mathbb{E}_{\mathbb{P}_{u,v}^{\star}} \bigl[\varphi(B_{u,v}) \bigr] \mathbb{E}_{\mathbb{P}_{u,v}^\star} \bigl[  \psi(\mathbf{P^\star},B_0) \bigr] d\varpi(u,v)
\end{align}
		for any continuous function $f:\mathbb{R}\times \mathbb{R}^k \rightarrow \mathbb{R}$, and so we can deduce that the random variables $(\mathbf{P^\star},B_0)$, $B_{u,v}$, and $Z(0)$ are mutually independent under $\mathbb{P}_{u,v}$ for $\varpi$-almost every $(u,v)\in\mathbb{R}^k\times\mathbb{R}^k$. Finally, by the definition of $B_{u,v}$ in \eqref{eq:B_uv}, we have that
		\[
		dZ(t)= b_{u,v}(t)dt + \sigma_{u,v}(t)d(\rho B_0(t)+\sqrt{1-\rho^2}B_{u,v}(t)) - dF\Bigl( \sum_{l=1}^k v_l \int_0^t g_{u,v}(s)d \mathcal{L}_l^\star(s) \Bigr),
		\]
	where $B_0$ and $B_{u,v}$ are standard Brownian motions under $\mathbb{P}_{u,v}^\star$, for $\varpi$-almost every $(u,v)\in\mathbb{R}$, by the above. Recalling from \eqref{eq:limit_feedback_processes} that each $\mathcal{L}^\star_l$ can be written as
	\[
	\mathcal{L}_l^\star(t)=\int_{\mathbb{R}^k\times\mathbb{R}^k} u_l \mathbb{P}^\star_{u,v}\bigl(t \geq \tau_0(Z) \, | \, B_0, \mathbf{P}^\star \bigr)d\varpi(u,v),\quad \text{for}\quad t\geq 0,
	\]
	this completes the proof.
\end{proof}

\section{The common noise system under the smallness condition}\label{Sect:com_nois}

This section is dedicated to a short proof of Theorem \ref{thm:global_unique}. To this end, we assume throughout that each initial density $V_0(\cdot|u,v)\in L^\infty(0,\infty)$ satisfies the bound
\begin{equation}\label{eq:smallness_cond}
\Vert V_{0}(\cdot|u,v) \Vert_{\infty} < g_{u,v}(0)\Vert F \Vert_{\mathrm{Lip}}\frac{1 }{\max\{u\cdot \hat{v} \; | \;  \hat{v}\in S(\mathbf{v}) \;\mathrm{s.t.}\; u\cdot \hat{v} >0\}},
\end{equation}
for all $ (u,v) \in S(\mathbf{u},
\mathbf{v})$. When \eqref{eq:smallness_cond} holds, it follows from the same arguments as in \cite[Theorem 3.4]{FS20} that any solution to \eqref{CMV} must be continuous in time for all $t\in[0,\infty)$. Next, we show in Sections \ref{subsect:density} and \ref{Proof_thm_uniq_common} below that the condition \eqref{eq:smallness_cond} also entails global uniqueness of \eqref{CMV}, hence completing the proof of Theorem \ref{thm:global_unique}.

The uniqueness of limit points in Theorem \ref{thm:yamada-watanabe} follows by the same argument. One simply needs to observe that all the steps in Sections \ref{subsect:density} and \ref{Proof_thm_uniq_common} also hold if we condition on $(\mathbf{P}^\star,B_0)$ in place of $B_0$ throughout, exploiting that $(\mathbf{P}^\star,B_0)$ is independent of the particle-specific Brownian motions $B_{u,v}$. In this way, we get global continuity in time and global uniqueness for any solution to the relaxed formulation \eqref{CMV_relaxed}. Given the existence of relaxed solutions and their pathwise uniqueness, a Yamada--Watanabe argument as in \cite[Theorem 2.3]{LS18b} completes the proof of Theorem \ref{thm:yamada-watanabe}.

\subsection{Bounded densities for all time}\label{subsect:density}

We start by observing that there is a density $V_t(\cdot | u,v)\in L^\infty(0,\infty)$ for each $X_{u,v}(t)$ on the positive half-line with absorption at the origin. Let $p_t(\cdot|u,v)$ denote the density of the random variable
\[
Z_{u,v}(t):=\int_0^t\sqrt{1-\rho^2}\sigma_{u,v}(s)dB_{u,v}(s),
\]
for each $t\geq 0$ and $(u,v)\in\mathbb{R}^k\times\mathbb{R}^k$. Letting
\[
\nu_t(\cdot | u,v):=\mathbb{P}(X_{u,v}\in \cdot\,,t<\tau_{u,v} \, | \, B_0 ),
\]
and considering the random variables
\[
Y_{u,v}(t):=\int_0^tb_{u,v}(s)ds - \int_0^t\rho\sigma_{u,v}(s)dB_0(s)-F\Bigl( \displaystyle \sum_{l=1}^{k} v_{l}^i \! \int_0^t g_{i}(s) d\mathcal{L}^n_{l,i}(s) \Bigr),
\]
a simple manipulation and Tonelli's theorem give that, for any Borel set $A\in\mathbb{R}$, 
\begin{align*}
\nu_t(A\,|\,B_0) &= \int_{0}^\infty \mathbb{P} \bigl( X^x_{u,v}(t) \in A , \, t < \tau^x_{u,v}  \mid B_0\bigr) V_0(x|u,v)dx \\ &\leq \int_{\mathbb{R}} \int_{A} p(t,y+x+Y(t)) V_0(x|u,v)dy dx \\
& =  \int_{A} \int_{\mathbb{R}} p(t,x+y+Y(t)) V_0(x|u,v) dx dy \leq \Vert V_0(\cdot | u,v) \Vert_\infty \text{Leb}(A),
\end{align*}
since $p_t(\cdot|u,v)$ integrates to $1$ on $\mathbb{R}$. Consequently, there is a bounded density $V_t(\cdot|u,v)$ of $\nu_t(\cdot|u,v)$. Furthermore, we see that $\Vert V_t(\cdot|u,v) \Vert_\infty \leq \Vert V_0(\cdot | u,v) \Vert_\infty$, so the bound \eqref{eq:smallness_cond} continues to hold for all times $t\geq 0$.

\subsection{Global uniqueness under the smallness condition}\label{Proof_thm_uniq_common}   

To prove the uniqueness part of Theorem \ref{thm:global_unique}, we adapt the arguments from the proof of \cite[Thm.~2.3]{LS18b}. Let $(X,\mathcal{L})$ and $(\bar{X},\mathcal{\bar{L}})$ be any two solutions to (\ref{CMV}) coupled through the same Brownian drivers $B$ and $B_0$. Let $\mathbf{L}_v$ and $\bar{\,\mathbf{L}}_v$ be the total loss processes, as defined in \eqref{eq:total_loss}, for the two different solutions. Retracing the arguments of \cite[Lemma 2.1]{LS18b}, and applying Fubini's theorem, we can deduce that
\begin{align}\label{eq:diff_loss_bound}
\loss_v(s) - \bar{\,\loss}_v(s)  \leq \mathbb{E} & \biggl[ \sum_{l=1}^k v_l \!\int_{\mathbb{R}^k\times \mathbb{R}^k} \hat{u}_l \int_{\bar{I}_{\hat{u},\hat{v}}(s)}^{I_{\hat{u},\hat{v}}(s)} \!\!V_0(x|\hat{u},\hat{v})dx d\varpi(\hat{u},\hat{v}) \,\Big|\, B^0\biggr],
\end{align}
where we have introduced the auxiliary processes
\[
I_{u,v}(t) := \sup_{s\leq t}\biggl\{ F\Bigl( \!\int_0^s g_{u,v}(r) d\mathbf{L}_v(r)  \Bigr) - Y_{u,v}(s)  \biggr\}
\]
and 
\[
\bar{I}_{u,v}(t) := \sup_{s\leq t}\biggl\{ F\Bigl( \!\int_0^s g_{u,v}(r) d\!\bar{\;\mathbf{L}}_v(r)  \Bigr) - Y_{u,v}(s) \biggr\}
\]
with
\[
Y_{u,v}(s) :=\int_0^sb_{u,v}(r)dr+\int_0^s \sigma_{u,v}(r)d(\rho B_0 + \sqrt{1-\rho^2}B_{u,v})(r).
\]
By symmetry, we can observe that $\bar{\,\loss}_v(s)- \loss_v(s)$ satisfies a bound entirely analogous to \eqref{eq:diff_loss_bound} only with $I_{u,v}(s)$ and $\bar{I}_{u,v}(s)$ interchanged in the range of integration. Consequently, we in fact have
\begin{align}\label{eq:main_loss_diff}
|\loss_v(s) - \bar{\,\loss}_v(s)|  \leq \sum_{l=1}^k v_l \mathbb{E} & \biggl[  \int_{\mathbb{R}^k\times \mathbb{R}^k} \hat{u}_l \int_{\bar{I}_{\hat{u},\hat{v}}(s)}^{I_{\hat{u},\hat{v}}(s)} \!\!V_0(x|\hat{u},\hat{v})dx d\varpi(\hat{u},\hat{v}) \,\Big|\, B_0\biggr] \nonumber \\
&\lor \sum_{l=1}^k v_l \mathbb{E} \biggl[  \int_{\mathbb{R}^k\times \mathbb{R}^k} \hat{u}_l \int_{I_{\hat{u},\hat{v}}(s)}^{\bar{I}_{\hat{u},\hat{v}}(s)} \!\!V_0(x|\hat{u},\hat{v})dx d\varpi(\hat{u},\hat{v}) \,\Big|\, B_0\biggr],
\end{align}
for all $s\geq 0$. Now, by repeating the first estimate from the proof of Lemma \ref{lemma1_HLS} in Section \ref{Sect:Idio} with $\mathbf{L}$ and $\!\bar{\,\mathbf{L}}$ in place of $\ell$ and $\bar{\ell}$, respectively, we obtain that
\begin{align*}
F\Bigl( \!\int_0^t g_{u,v}(s) & d\mathbf{L}_v(s) \Bigr)  \geq F\Bigl( \int_0^t g_{u,v}(s)d\!\bar{\;\mathbf{L}}_v(s) \Bigr)   - g_{u,v}(0) \Vert F \Vert_{\text{Lip}} \Vert \mathbf{L} - \!\bar{\,\mathbf{L}}\Vert_{t}^\star,
\end{align*}
where
\[
\Vert \mathbf{L} - \!\bar{\,\mathbf{L}}\Vert_{t}^\star := \sup_{v\in S(\mathbf{v})} \sup_{s\in[0,t]}
| \loss_v(s) - \bar{\,\loss}_v(s) | .
\]
In turn, relying on this inequality together with the estimate \eqref{eq:main_loss_diff}, we can retrace the arguments of \cite[Theorem 2.2]{LS18b} in order to arrive at
\begin{align*}
| \loss_v(s) - \bar{\,\loss}_v(s) | \leq \Vert F \Vert_{\text{Lip}}  \Vert \loss- \bar{\,\loss}\Vert_{s}^\star  \sum_{l=1}^k v_l \!\int_{\mathbb{R}^k\times \mathbb{R}^k} g_{\hat{u},\hat{v}}(0)\hat{u}_l  \Vert V_0(\cdot |\hat{u},\hat{v}) \Vert_{\infty} d\varpi(\hat{u},\hat{v}) .
\end{align*}
From here, it simply remains to observe that the bound \eqref{eq:smallness_cond} gives
\begin{align*}
&\sup_{v\in S(\mathbf{v})}\Vert F \Vert_{\text{Lip}}\sum_{l=1}^k v_l \!\int_{\mathbb{R}^k\times \mathbb{R}^k} g_{\hat{u},\hat{v}}(0)\hat{u}_l  \Vert V_0(\cdot |\hat{u},\hat{v}) \Vert_{\infty} d\varpi(\hat{u},\hat{v}) < 1,
\end{align*}
and hence, for any $s \geq0$, we have
\[
\sup_{s \in S(\mathbf{v})}| \loss_v(s) - \bar{\,\loss}_v(s) | \leq (1-\delta)  \Vert \loss- \bar{\,\loss}\Vert_{s}^\star,
\]
for some $\delta >0$. Naturally, this implies $\Vert \loss- \bar{\,\loss}\Vert_{s}^\star=0$ for all $s\geq 0$, as we otherwise have a contradiction, and so there is indeed uniqueness of solutions.

\section{The idiosyncratic system up to explosion}\label{Sect:Idio}

In this section we give a succinct proof of Theorem \ref{thm:idio_reg}, based on the arguments from \cite{HLS18}. Let us start by introducing the notation $\Vert f \Vert_{t}:=\Vert f \Vert_{{L{}}^{\infty}(0,t)}$ and recalling the notation $S(\mathbf{v})$ for the support of the second marginal of $\varpi$. With this notation, we can then consider  the space of continuous maps $v\mapsto\ell_{v}(\cdot)$ from $S(\mathbf{v})$ to $L^{\infty}(0,t)$, denoted by
\[
C^\star_t:= C\bigl(S(\mathbf{v});L^{\infty}(0,t) \bigr),
\]
with respect to the following supremum norm
\[
\Vert \ell \Vert_t^\star := \sup_{v\in S(\mathbf{v})} \Vert \ell_v(\cdot)\Vert_t.
\]
Since the domain $S(\mathbf{v})$ is a compact subset of $\mathbb{R}^k$, by Assumption \ref{assump:paramters_emperical}, and noting that the codomain $L^{\infty}(0,t)$ is a Banach space, this norm makes $C_t^\star$ a Banach space. In the next subsection, we will show that one can find a nice solution to the McKean--Vlasov system \eqref{CMV} with $\rho=0$, for which the contagion processes are continuously differentiable in time up until an explosion time. This is achieved by All a fixed point argument for a mapping $\Gamma:C_t^\star\rightarrow C_t^\star$ (defined in \eqref{eq:Gamma_fp_map} below) for a small enough time $t>0$. Throughout, we assume Assumptions 1.1, 1.2, and 1.3 are satisfied.

\subsection{Existence of differentiable contagion processes up to explosion}\label{exist_reg_soln}

Given $T>0$, we define the map $\Gamma : C_T^\star \mapsto C_T^\star$ by 
\begin{equation}\label{eq:Gamma_fp_map}
\Gamma[\ell]_{v}(t):=\sum_{l=1}^k v_l \int_{\mathbb{R}^k \times \mathbb{R}^k } \int_0^\infty \hat{u}_l \mathbb{P}(t\geq \tau^{x,\ell}_{\hat{u},\hat{v}}) d\nu_0(x|u,v) d \varpi(\hat{u},\hat{v}),
\end{equation}
for all $t\in[0,T]$ and $v\in S(\mathbf{v})$, where
\begin{equation}\label{X^ell}
\begin{cases}
\tau^{x,\ell}_{u,v}=\inf\{t > 0: X^{x,\ell}_{u,v}(t)\leq0\} \\[6pt]
X^{x,\ell}_{u,v}(t)=x + \int_0^t \!b_{u,v}(s)ds + \int_0^t\!\sigma_{u,v}(s) dB_{u,v}(s) - F\bigl( \int_0^t \!g_{u,v}(s)d\ell_v(s) \bigr).
\end{cases}
\end{equation}
Note that, as long as $s \mapsto \ell_v(s)$ is continuous or of finite variation, the integral of each $g_{u,v}$ against $\ell_v$ in (\ref{X^ell}) is well-defined, since each $g_{u,v}$ is both continuous and of finite variation by Assumption \ref{MV_assump} (see e.g.~\cite[Sect.~1.2]{stroock_integration}). Naturally, all of the results that follow are stated for inputs such that the mapping $\Gamma$ makes sense.

The cornerstone of our analysis is the next comparison argument. Below, it leads us to the fixed point argument for existence of regular solutions, and it then reappears in the generic uniqueness argument of Section \ref{generic uniqueness} where we complete the proof of the full statement of Theorem \ref{thm:idio_reg}.

\begin{lem}[Comparison argument]\label{lemma1_HLS} Fix any two $\ell,\bar{\ell}\in C_T^\star$ such that $s\mapsto\ell_v(s)$ and $s\mapsto \bar{\ell}_v(s)$ are increasing with $\ell_v(0)=\bar{\ell}_v(0)=0$ for all $v\in S(\mathbf{v})$. Fix also $t_0>0$ and suppose $s\mapsto\ell_v(s)$ is continuous on $[0,t_0)$ for all $v\in S(\mathbf{v})$. Then we have
	\[
	\bigl( \Gamma[\ell]_v(t) - \Gamma[ \bar{\ell} \, ]_v(t) \bigr )^+ \leq C \Vert (\ell - \bar{\ell}\,)^+\Vert_{t}^\star \int_0^t (t-s)^{-\frac{1}{2}} d \Gamma[\ell]_v(s),
	\]
	for all $t<t_0$ and all  $v\in S(\mathbf{v})$, where $C>0$ is a fixed constant independent of $t_0$ and $v$.
\end{lem}
\begin{proof} Fix $t<t_0$. Recalling that $\ell_v(0)=0$, integration by parts for Riemann--Stieltjes integrals (see e.g., \cite[Sect.~1.2]{stroock_integration}) gives
	\begin{equation*}
	\int_0^t g_{u,v}(s)d\ell_v(s) = g_{u,v}(t)\ell_v(t) + \int_0^t \ell_v(s) d(-g_{u,v})(s),
	\end{equation*}
	and likewise for $\bar{\ell}_v$. Using this and the assumptions on $F$, $g_{u,v}$, $\ell$, and $\bar{\ell}$, we have
	\begin{align*}
	F\Bigl( \!\int_0^t g_{u,v}(s) & d\ell_v(s) \Bigr) - F\Bigl( \int_0^t g_{u,v}(s)d\bar{\ell}_v(s) \Bigr) \leq  \Vert F \Vert_{\text{Lip}} \Bigl( \int_0^t g_{u,v}(s)d\ell_v(s)  - \int_0^t g_{u,v}(s)d\bar{\ell}_v(s) \Bigr)^+\\
	&= \Vert F \Vert_{\text{Lip}} \Bigl( g(t)(\ell_v(t)-\bar{\ell}_v(t)) + \int_0^t (\ell(s)-\bar{\ell}(s)) d(-g)(s) \Bigr)^+\\
	&\leq g(t) \Vert F \Vert_{\text{Lip}} \bigl(\ell_v(t) - \bar{\ell}_v(t) \bigr)^+ +  \Vert F \Vert_{\text{Lip}} \int_0^t \bigl(\ell_v(s) - \bar{\ell}_v(s)\bigr)^+ d(-g_{u,v})(s) \\
	&\leq   g_{u,v}(0) \Vert F \Vert_{\text{Lip}} \Vert (\ell_v - \bar{\ell}_v)^+\Vert_{t}.
	\end{align*}
	Thus, taking the difference between the two processes $X^{x,\ell}_{u,v}$ and $X^{x,\bar{\ell}}_{u,v}$, as defined in (\ref{X^ell}) coupled through the same Brownian motion, it follows that
	\[
	X^{x,\bar{\ell}}_{u,v}(t) - X^{x,\ell}_{u,v}(t) \leq   g_{u,v}(0) \Vert F \Vert_{\text{Lip}} \Vert (\ell_v - \bar{\ell}_v)^+\Vert_{t}.
	\]
	Therefore, using the continuity of $\ell_v$, for any $s\in[0,t]$, it holds on the event $\{ \tau^{x,\ell}_{u,v} = s  \}$ that
	\[
	X^{x,\bar{\ell}}_{u,v}(s) = X^{x,\bar{\ell}}_{u,v}(s) - X^{x,\ell}_{u,v}(s) \leq  g_{u,v}(0) \Vert F \Vert_{\text{Lip}} \Vert (\ell - \bar{\ell})^+\Vert_{s}^\star.
	\]
	Based on this, we can replicate the arguments from \cite[Prop.~3.1]{HLS18}, by instead conditioning on the value of $\tau^{x,\ell}_{u,v}$ and using the previous inequality, to deduce that
	\begin{align*}
	\mathbb{P}(t\geq \tau^{x,\ell}_{u,v}) & - \mathbb{P}(t\geq \tau^{x,\bar{\ell}}_{u,v} )  \\ &\leq \int_0^t  \mathbb{P}\bigl( \inf_{r\in[s,t]} \textstyle\int_s^r\!\sigma_{u,v}(h) dB_{u,v}(h) > - g_{u,v}(0)\Vert F \Vert_{\mathrm{Lip}} \Vert (\ell - \bar{\ell})^+\Vert_{s}^\star \bigr ) d \mathbb{P}(s\geq \tau^{x,\ell}_{u,v}).
	\end{align*}
	Performing a time change in the Brownian integral, and using that there is a uniform $\epsilon>0$ such that $\epsilon \leq \sigma_{u,v} \leq \epsilon^{-1}$, by Assumption \ref{MV_assump}, it follows from the law of the infimum of a Brownian motion that
	\[
	\mathbb{P}(t\geq \tau^{x,\ell}_{u,v}) - \mathbb{P}(t\geq \tau^{x,\bar{\ell}}_{u,v} )
	\leq
	C \Vert (\ell - \bar{\ell}\,)^+\Vert_{t}^\star\int_0^t (t-s)^{-\frac{1}{2}} d \mathbb{P}(s\geq \tau^{x,\ell}_{u,v})
	\]
	where the constant $C>0$ is independent of $t$, $x$, $u$, and $v$. Now fix any $\tilde{v}\in S(\mathbf{v})$. Multiplying both sides of the above inequality by $\sum_{l=1}^k\tilde{v}_lu_l$ and recalling that this is non-negative for all $u$ in the support of $\varpi$, by Assumption \ref{assump:paramters_emperical}, we can then integrate both sides of the resulting inequality against $d\nu_0(x|u,v) d \varpi(u,v)$, for $(x,u,v)\in \mathbb{R}_+ \times \mathbb{R}^k \times \mathbb{R}^k$, to arrive at
	\[
	\Gamma[\ell]_{\tilde{v}}(t) -  \Gamma[\bar{\ell}\,]_{\tilde{v}}(t)   \leq C \Vert (\ell - \bar{\ell}\,)^+\Vert_{t}^\star\int_0^t (t-s)^{-\frac{1}{2}} d \Gamma[\ell]_{\tilde{v}}(t),
	\]
	for all $t<t_0$, for some fixed numerical constant $C>0$ independent of $t_0$ and $\tilde{v}$. As the right-hand side is positive, this proves the lemma.
\end{proof}

For any $\gamma\in(0,1/2)$, $A>0$, and $t>0$, we define the space $\mathcal{S}(\gamma, A, t) \subset C_t^\star$ by
\begin{equation}
\mathcal{S}(\gamma, A, t) := \bigl\{  \ell \in C\bigl(S(\mathbf{v});H^1(0,t) \bigr) : \ell_v^\prime(t) \leq A t^{-\gamma} \;\; \text{for a.e.} \; t\in[0,t], v\in S(\mathbf{v})     \bigr\},
\end{equation}
which is a complete metric space with the metric inherited from $C_t^\star$. Moreover, we define the map
\[
\hat{\Gamma}[\ell;u,v](t):= \int_0^\infty \mathbb{P}(t\geq \tau^{x,\ell}_{u,v})  V_0(x|u,v) dx,
\]
so that $\Gamma[\ell]_{\tilde{v}}(t)=\sum_{l=1}^k \tilde{v}_l \int_{\mathbb{R}^k\times\mathbb{R}^k} u_l\hat{\Gamma}[\ell;u,v](t) d\varpi(u,v)$. Then, for each $u$ and $v$, we can replicate the arguments from \cite[Sect.~4]{HLS18} for the function $t\mapsto \hat{\Gamma}[\ell;u,v](t)$ in place of the corresponding function considered there. Given $\varpi$ and $V_0(\cdot|u,v)$ satisfying Assumption \ref{X0_assump}, we can thus conclude (arguing by analogy with \cite[Prop.~4.9]{HLS18}), that there exists $A>0$ such that, for any $\varepsilon_0>0$, there is a small enough time $t_0>0$ for which
\begin{equation}\label{eq:Gamma_stable}
\Gamma : \mathcal{S}\bigl(\tfrac{1-\beta}{2}, A+\varepsilon_0, t_0 \bigr) \rightarrow  \mathcal{S}\bigl(\tfrac{1-\beta}{2}, A+\varepsilon_0, t_0 \bigr).
\end{equation}

Moreover, by analogy with \cite[Thm.~1.6]{HLS18}, we can deduce from Lemma \ref{lemma1_HLS}  that $\Gamma$ is a contraction on this space for small enough $t_0>0$. Therefore, the small time existence of a continuously differentiable total loss process $\mathbf{L}_v(t)=\sum_{l=1}^k v_l \mathcal{L}_l(t)$ for the system \eqref{CMV} with $\rho=0$ now follows from an application of Banach's fixed point theorem as in the proof of \cite[Thm.~1.7]{HLS18}. Finally, by replicating the bootstrapping argument from the proof of \cite[Cor.~5.3]{HLS18}, we conclude that the regular solution extends until the first time $T_{\star}$ such that the $H^1$ norm of $(\mathcal{L}_1,\ldots,\mathcal{L}_k)$ on $[0,T_\star)$ diverges. This proves the first part of Theorem \ref{thm:idio_reg}.

\subsection{Local uniqueness of c\`adl\`ag solutions with the cascade condition}\label{generic uniqueness}
It remains to verify that the approach of \cite[Thm.~1.8]{HLS18} can be extended to the present setting. This will be possible once we have the two lemmas that we turn to next. The first lemma concerns a family of auxiliary McKean--Vlasov problems given by
\begin{equation}\label{eq:X_eps}
\begin{cases}
X^{x,\epsilon}_{u,v}(t)=x\mathbf{1}_{x\geq\varepsilon} - \frac{\varepsilon}{4} + \int_0^t \!b_{u,v}(s)ds + \int_0^t \!\sigma_{u,v}(s) dB_{u,v}(s) - F\bigl( g_{u,v}(0) \lambda_v^\eps + \int_0^t \!g_{u,v}(s)d\mathbf{L}^\eps_v (s) \bigr)\\[6pt]
\mathbf{L}^{\eps}_{v}(t)= \sum_{l=1}^k v_l \int_{\mathbb{R}_+\times\mathbb{R}^k\times \mathbb{R}^k} \hat{u}_l \mathbb{P}(\tau^{x,\eps}_{\hat{u},
	\hat{v}}\leq t)d\nu_0(x|\hat{u},\hat{v})d\varpi(\hat{u},\hat{v}) \\[6pt] \tau^{x,\eps}_{u,v}=\inf\{t\geq0:{X}^{x,\eps}_{u,v}(t)\leq0\},
\end{cases}
\end{equation}
for $\varepsilon>0$, where $\lambda_{v}^\eps:= \sum_{l=1}^k v_l \int_{(0,\eps)\times\mathbb{R}^k\times\mathbb{R}^k} \hat{u}_l d\nu_0(x|\hat{u},\hat{v})d\varpi(\hat{u},\hat{v}) $.
Similarly to \cite[Sect.~5.2]{HLS18}, the idea is to create a family of approximating solutions, by artificially removing an $\eps$ amount of mass, which is then counted as a loss already at time zero. By rewriting each $\mathbf{L}_v^{\eps}$ as
\begin{equation}\label{eq:eps_loss}
\mathbf{L}_v^{\eps}(t)=\lambda_v^\eps+\tilde{\,\mathbf{L}}_v^{\varepsilon}(t), \quad 
\tilde{\,\mathbf{L}}_v^{\varepsilon}(t) := \sum_{l=1}^k v_l \int_{[\eps,\infty)\times\mathbb{R}^k\times \mathbb{R}^k}\hat{u}_l\mathbb{P}(\tau^{x,\eps}_{\hat{u},\hat{v}}\leq t)d\nu_0(x|\hat{u},\hat{v})d\varpi(\hat{u},\hat{v}),
\end{equation}
we can verify that there indeed exist solutions to the approximating problems \eqref{eq:X_eps}. Moreover, we can obtain regularity estimates for these solutions, on a small time interval, uniformly in $\eps>0$. Recall that we are taking Assumptions 1.1, 1.2, and 1.3 to be satisfied.

\begin{lem}[Uniformly regular approximations]\label{lemma2_HLS}
	There is an $\eps_0>0$ such that (\ref{eq:X_eps}) has a family of solutions $\{\mathbf{L}^\eps\}_{\eps\leq \eps_0}$ which are uniformly regular in the following sense: There exists $A>0$ and $t_0>0$ such that ${\mathbf{L}{}}^{\eps} \in \mathcal{S}(\tfrac{1-\beta}{2}, A, t_0)$ uniformly in $\eps\in(0,\eps_0]$. 
\end{lem}
\begin{proof}First of all, we can ensure that $\lambda_v^\eps \leq C_\star \eps^{1+\beta}/(1+\beta)$ uniformly in $v$, for small enough $\eps>0$, by Assumption \ref{X0_assump}, and clearly $F(x)=o(x^{1/(1+\beta)})$ as $x\downarrow0$, since $F$ is Lipschitz with $F(0)=0$, by Assumption \ref{MV_assump}. Hence there exists $\eps_0>0$ such that $F(g_{u,v}(0)\lambda_v^\eps)\leq \eps/4$ for all $\eps\in(0,\eps_0)$.  Next, using (\ref{eq:eps_loss}) and making the change of variables $y=x-\eps/4 - F(g_{u,v}(0)\lambda_v^\eps)$ in (\ref{eq:X_eps}) we obtain the equivalent formulation
	\begin{equation*}
	\begin{cases}
	\tilde{X}^{y,\eps}_{u,v}(t)=y+\int_0^t b_{u,v}(s)ds + \int_0^t\sigma_{u,v}(s) dB_{u,v}(s) - F\bigl( g_{u,v}(0)\lambda_v^\eps + \int_0^t g_{u,v}(s)d \tilde{\,\mathbf{L}}_v^{\varepsilon}(s) \bigr) + F\bigl( g_{u,v}(0)\lambda_v^\eps \bigr)
	\\[4pt]
	\tilde{\,\mathbf{L}}_v^{\varepsilon}(t)= \sum_{l=1}^k v_l \int_{\mathbb{R}_+\times\mathbb{R}^k\times \mathbb{R}^k}\hat{u}_l\mathbb{P}(\tilde{\tau}^{y,\eps}_{\hat{u},\hat{v}}\leq t)V^\eps_0(y|\hat{u},\hat{v})dyd\varpi(\hat{u},\hat{v})\\[4pt]
	V_0^\eps(y|u,v)=V_0\bigl(y+\frac{\eps}{4} + F(g_{u,v}(0)\lambda_v^\eps) \bigr) \mathbf{1}_{\{ y + \frac{\eps}{4} + F(g_{u,v}(0)\lambda_v^\eps)  \geq \eps  \}}
	\\[4pt]
	\tilde{\tau}^{y,\eps}_{u,v}=\inf \{ t\geq 0 : \tilde{X}^{y,\eps}_{u,v}(t) \leq 0 \} 
	\end{cases}
	\end{equation*}
	Now take $\eps_0 \leq x_\star$. Recalling that $F(g_{u,v}(0)\lambda_v^\eps)\leq \eps/4$ for all $\eps\in(0,\eps_0)$, we can then observe that
	\begin{align*}
	V_0^\eps(y|u,v) & \leq C_\star \bigl( y+\eps/4 + F(g_{u,v}(0)\lambda_v^\eps)  \bigr)^\beta \mathbf{1}_{y+\eps/4 + F(g_{u,v}(0)\lambda_v^\eps) \geq \eps} \\
	& \leq C_\star (y+\eps/2 )^\beta \mathbf{1}_{y \geq \eps/2} \leq 2^\beta C_\star y^\beta
	\qquad \text{for all} \quad y<x_\star/2,
	\end{align*}
	uniformly in $u$, $v$, and $\eps \in (0,\eps_0)$. Therefore, for each $\eps\in (0,\eps_0)$, we can indeed construct a regular solution $\tilde{\,\mathbf{L}}^{\varepsilon}$ to the above system by the first part of Theorem \ref{thm:idio_reg} proved in Section \ref{exist_reg_soln}. Moreover, since the boundary control on $V_0^\eps(\cdot|u,v)$ is uniform in $\eps\in(0,\eps_0)$, uniformly in $u$ and $v$, it follows from the fixed point argument in Section \ref{exist_reg_soln} that the regularity of the solutions $\tilde{\,\mathbf{L}}^\eps$ is also uniform in $\eps\in(0,\eps_0)$.
	By (\ref{eq:eps_loss}), the uniform regularity of the original family ${\mathbf{L}{}}^\eps$ follows a fortiori from that of $\tilde{\,\mathbf{L}}^{\eps}$, and thus the proof is complete.
\end{proof}

Armed with Lemma \ref{lemma2_HLS}, we can now proceed to the final ingredient of the general uniqueness result, namely the so-named monotonicity and trapping argument from \cite[Sect.~5.2]{HLS18}.

\begin{lem}[Strict domination with vanishing envelope]\label{lemma3_HLS} 
	Let $\mathbf{L{}}:S(\mathbf{v})\rightarrow L^\infty(0,T)$ be given by
	\[
	\mathbf{L{}}_v(t):=\sum_{l=1}^k v_l \mathcal{L}_l = \sum_{l=1}^k v_l  \int_{\mathbb{R}_+ \times \mathbb{R}^k \times \mathbb{R}^k} \hat{u}_l \mathbb{P}( t\geq \tau^x_{\hat{u},\hat{v}}) d\nu_0(x|\hat{u},\hat{v})d\varpi(\hat{u},\hat{v})
	\]
	for a generic solution to \eqref{CMV} with $\rho=0$, and suppose there are no jumps of $(\mathcal{L}_1,\ldots,\mathcal{L}_k)$ on $[0,t_0)$ so $s\mapsto \mathbf{L{}}_v(s)$ is continuous on $[0,t_0)$ for all $v$. If $\mathbf{L}^\eps$ is a continuous solution of \eqref{eq:X_eps} on $[0,t_0)$, then we have strict domination $\mathbf{L}^\eps_v> \mathbf{L{}}_v$ on $[0,t_0)$, for all $v \neq 0$.  Moreover, if $\mathbf{L}$ is differentiable on $[0,t_0)$ and the family $\{ \mathbf{L}^\eps \}$ is uniformly regular on $[0,t_0)$ in the sense of Lemma \ref{lemma2_HLS}, then there is another time $0<t_1\leq t_0$ such that a vanishing envelope is guaranteed on $[0,t_1]$, namely $\Vert \mathbf{L} - \mathbf{L}^\eps \Vert_{t_1}^\star \rightarrow 0$ as $\varepsilon\rightarrow 0$.
\end{lem}

\begin{proof}
	Noting that $\mathbf{L}^\eps_v(0)=\lambda_v^\epsilon > 0=\mathbf{L{}}_v(0)$ for $v\neq0$, towards a contradiction we let $t \in(0, t_0)$ be the first time $\mathbf{L}^\eps_v(t) = \mathbf{L}_v(t)$ for some $v\neq0$. Then it holds for any $s<t$ that
	\begin{align*}
	g_{u,v}(0)\lambda_v^\eps + \int_0^s \!g_{u,v}(r)d\mathbf{L}^\eps_v (r) & = \int_0^s \mathbf{L}^\eps_v(r) d(-g_{u,v})(r) + g_{u,v}(s)\mathbf{L}^\eps_v(s) \nonumber\\
	& \geq  \int_0^s \mathbf{L}_v(r) d(-g_{u,v})(r) + g_{u,v}(s)\mathbf{L{}}_v(s) = \int_0^s g_{u,v}(r)d\mathbf{L{}}_v(r).
	\end{align*}
	and, since $F$ is increasing, we thus have
	\begin{equation}\label{eq:eps/4_1}
	X^x_{u,v}(s)-X^{x,\eps}_{u,v}(s) = x\mathbf{1}_{x<\varepsilon}+\frac{\varepsilon}{4} +  F\Bigl( g_{u,v}(0) \lambda_v^\eps + \int_0^s \!g_{u,v}(s)d\mathbf{L}^\eps_v \Bigr) - F\Bigl(\int_0^s \!g_{u,v}(s)d\mathbf{L{}}_v \Bigr) \geq \frac{\varepsilon}{4},
	\end{equation}
	for all $s\in(0,t)$.
	Arguing as in the proof of \cite[Lemma 5.6]{HLS18}, it follows from (\ref{eq:eps/4_1})  that
	\begin{align*}
	\mathbf{L}_v^\eps(t) & \geq \mathbf{L{}}_v(t) + \sum_{l=1}^k v_l  \int_{\mathbb{R}_+ \times \mathbb{R}^k \times \mathbb{R}^k} \hat{u}_l \mathbb{P} \Bigl( \inf_{r\in[0,t] } X^x_{u,v}(s) \in (0, \eps/4 ] \Bigr) d\nu_0(x|\hat{u},\hat{v})d\varpi(\hat{u},\hat{v}) > \mathbf{L{}}_v(t),
	\end{align*}
	which contradicts the definition of $t$, thus proving the first claim.
	
	For the second claim, we can now rely on the fact that $\mathbf{L}^\eps_v>\mathbf{L{}}_v$ on $[0,t_0)$ for all $v\neq 0$. Consequently, since $X^{x,\varepsilon}_{u,v}(s)=0$ on the event $\{ \tau^{x,\varepsilon}_{u,v}=s \}$, we deduce that, on this event,
	\begin{equation}\label{eq:envelope}
	X^{x}_{u,v}(s) = X^{x}_{u,v}(s) - X^{x,\eps}_{u,v}(s) \leq \varepsilon + \frac{\varepsilon}{4} + g_{u,v}(0) \Vert F \Vert_{\mathrm{Lip}} \Vert \mathbf{L}^\eps_v -\mathbf{L{}}_v \Vert_s,
	\end{equation}
	where the inequality follows by the equality in (\ref{eq:eps/4_1}) and the same estimate as in the proof of Lemma \ref{lemma1_HLS}. From here, (\ref{eq:envelope}) allows us to retrace the last steps of \cite[Proposition 3.1]{HLS18} with 
	 the term $g_{u,v}(0) \Vert F \Vert_{\mathrm{Lip}} \Vert \mathbf{L}^\eps -\mathbf{L{}} \Vert_s^\star$ in place of the term $\alpha(\ell_s -\bar{\ell}_s)$ and hence we can replicate the proof of \cite[Lemma 5.7]{HLS18} to verify the second claim.
\end{proof}

Based on Lemmas \ref{lemma1_HLS} and \ref{lemma3_HLS}, we can now give the proof of Theorem \ref{thm:idio_reg}. First of all, if $\mathbf{L}$ corresponds to the regular solution obtained at the end of Section \ref{exist_reg_soln} and $\widetilde{\mathbf{L}}$ corresponds to an arbitrary c\`adl\`ag solution on $[0,T_\star]$, then the same arguments as in Lemma \ref{lemma1_HLS} yield
	\[
\bigl(\mathbf{L}_v(t) - \widetilde{\mathbf{L}}_v(t) \bigr )^+ \leq C \Vert (\mathbf{L} - \widetilde{\mathbf{L}}\,)^+\Vert_{t}^\star \int_0^t (t-s)^{-\frac{1}{2}} d \mathbf{L}_v(s),
\]
for all $v\in S(\mathbf{v})$, for all $t\in[0,T_\star)$. Arguing as in \cite[Proof of Thm.~1.6]{HLS18}, we can thus use the regularity of $\mathbf{L}_v(t)$ as the integrator to deduce that $\mathbf{L}_v(t) \leq \widetilde{\mathbf{L}}_v(t)$ for all $t\in[0,T_\star)$ and all $v\in S(\mathbf{v})$. Let $\tilde{t}>0$ be such $s\mapsto\widetilde{\mathbf{L}}_v(s)$ is continuous on $[0,\tilde{t})$ for all $v\in S(\mathbf{v})$ (possible since we can consider the first time one of the $k$ right-continuous processes $\widetilde{\mathcal{L}}_l$ jumps, and each of these has $\widetilde{\mathcal{L}}_l(0)=0$). Then the first part of Lemma \ref{lemma3_HLS} gives $\mathbf{L}_v(t) \leq \widetilde{\mathbf{L}}_v(t) < \mathbf{L}^\epsilon_v(t) $ for all $t\in[0,\tilde{t})$ and all $v\in S(\mathbf{v})$. Applying the second part of Lemma \ref{lemma3_HLS}, and sending $\varepsilon \downarrow 0$, we then get a time $\tilde{t}_1$ such that  $\mathbf{L}=\widetilde{\mathbf{L}}$ on $[0,\tilde{t}_1)$. Suppose $\tilde{t}_1 < T_\star$ (otherwise, we are done). By the regularity of $\widetilde{\mathbf{L}}$ on $[0,\tilde{t}_1)$ we can conclude as in \cite[Lemma 5.1]{HLS18} that the density of each $\mathbb{P}(\widetilde{X}_{u,v}(t-)\in dx,\, t\leq \widetilde{\tau}_{u,v})$ is H{\"o}lder continuous at zero, and hence it is straightforward to check that we get
\[\lim_{\varepsilon\downarrow 0}  \lim_{m\rightarrow \infty}\Delta_{\tilde{t}_1,v}^{(m,\varepsilon)}=0
\]
in \eqref{eq:mf_cascade_cond}, for all $v\in S(\mathbf{v})$. In turn, $\Delta \widetilde{\mathbf{L}}(\tilde{t}_1)=0$, since the statement of the theorem assumes that the jump sizes are smaller than or equal to those given by the cascade condition. This gives $\mathbf{L}=\widetilde{\mathbf{L}}$ on $[0,\tilde{t}_1]$. Using that we now have a nicely behaved density of $\mathbb{P}(\widetilde{X}_{u,v}(t)\in dx,\, t< \widetilde{\tau}_{u,v})$, we can proceed as in \cite[Proof of Corollary 5.3]{HLS18} and bootstrap the uniqueness argument to get uniqueness on all of $[0,T_\star)$. This completes the proof of Theorem \ref{thm:idio_reg}.

\section{Some directions for future research}

The results presented in this paper raise a series of interesting questions for future research.
As we have seen, both the finite particle system and the mean-field problem suffer from non-uniqueness. In the finite setting, we have argued that this is naturally resolved by working with the cascade condition \eqref{particle_cascade_cond1}, meaning that we select the greatest c\`adl\`ag clearing capital solution for the finite interbank system. In the mean-field setting, however, the issue of non-uniqueness is less clear. It would be interesting to understand if it is possible to show convergence of \eqref{particle_cascade_cond1} as $n\rightarrow \infty$, and if the limit points from Theorem \ref{thm:mf_conv} can be shown to satisfy \eqref{eq:mf_cascade_cond}. Also, there is the question of whether there can be solutions that are not limit points, similarly to \cite{nutz20}, where a mean-field game problem was shown to have various solutions failing to be limits of the finite player problems.

As we saw in Theorem \ref{thm:global_unique}, uniqueness becomes much easier when jump discontinuities can be ruled out. However, this cannot always be dismissed, and then uniqueness remains out of reach, even with a selection principle such as \eqref{eq:mf_cascade_cond}. Theorem \ref{thm:idio_reg} suggest that the methods of \cite{DNS21}, suitably adapted, could yield global uniqueness under \eqref{eq:mf_cascade_cond}. With the common noise, this is no longer the case, as several parts of the arguments in \cite{DNS21} break down. Fine results related to H{\"o}lder boundary regularity of the Dirichlet problem for one-dimensional stochastic partial differential equations with transport noise, such as \cite[Theorem 4.1]{krylov}, could prove to be key for making progress in this direction. To overcome the difficulty of jumps and non-uniqueness, it is also possible to consider a more gradual realisation of the financial contagion as in \cite{HS19}. Certainly, the convergence arguments presented here easily transfer to the corresponding heterogeneous formulation of \cite{HS19}. In this case, it remains of interest to understand the jumps as limiting cases where the time line of a continuously unfolding cascade is collapsed to zero.

Another interesting direction is to utilise the mean-field problem \eqref{CMV1} together with \eqref{eq:mf_cascade_cond} as a way of computing risk measures based on partial information from real world financial systems. In practice, the links between financial institutions are far from publicly known, so it could be practical to work with a distribution $\varpi$ on $\mathbb{R}^k\times \mathbb{R}^k$, for a not too large level of granularity $k$, such that the dot product structure $u\cdot v$ gives a satisfactory estimate of one's belief about the heterogeneity of the links. Furthermore, the convergence to the mean-field problem then justifies only attempting to estimate average parameters for different subsets of institutions, corresponding to different parts of the support of $\varpi$. Exploiting the structural modelling, it may be feasible to exploit equity data for estimating such average parameters for the drift, volatility, and correlation to a common factor. In contrast to the homogeneous problem, one can then look at notions of systemic importance. For example, building on the influential CoVaR methodology \cite{AB16}, one could study the value at risk or expected shortfall among peripheral parts of the system conditional on a period of steep increase in the probability of default within a core fraction.

Finally, we wish to stress that the cascade condition \eqref{eq:mf_cascade_cond} is instructive for numerical implementation. From one time-step to the next, there is inevitably a small increase in the probability of having defaulted before or at the current time-step for each of the mean-field particles, before factoring in any contagion. Using the size of these increases as the first input, we can then iterate \eqref{eq:mf_cascade_cond} for a finite number of times until the resulting change is below some acceptable threshold, thus giving the total change in the contagion processes before proceeding to the next time-step. Between time-steps, we can effectively think of all the mean-field particles as simply sitting in a heat bath, so all we have to do is convolve with a family of heat kernels and then adjust the heat profiles at each time step according to the above iterative procedure for \eqref{eq:mf_cascade_cond}. This yields an efficient way to simulate directly the mean-field system. It would be interesting to understand what the gains are in terms of computational complexity compared to simulating the finite system, as we let the number of particles increase.

\subsection*{Acknowledgements}
We wish to thank two anonymous referees, the AE, and the Co-Editor for several useful suggestions. We are particularly grateful to one of the referees for highlighting the lacuna in \cite[Lemma 3.13]{LS18a}.

\bibliographystyle{alpha}

\end{document}